\documentclass[11pt,a4paper,reqno]{amsart}
 \usepackage{latexsym, verbatim}
\usepackage[T1]{fontenc}
\usepackage{amsmath}
\usepackage{amsthm}
\usepackage{amsfonts}
\usepackage{amssymb}
\usepackage{graphicx}
\usepackage{amsbsy}
\usepackage{mathrsfs}
\usepackage{bbm}
\usepackage{url,color}
\def\pa{\partial}
\def\e{\epsilon}
\newtheorem{theorem}{Theorem}[section]
\newtheorem{lemma}[theorem]{Lemma}

\newtheorem{proposition}[theorem]{Proposition}

\theoremstyle{remark}
\newtheorem{remark}[theorem]{\it \bf{Remark}\/}

\numberwithin{equation}{section}
\catcode`@=11
\def\section{\@startsection{section}{1}%
  \z@{1.5\linespacing\@plus\linespacing}{.5\linespacing}%
  {\normalfont\bfseries\large\centering}}
\catcode`@=12
\newcommand{\be}{\begin{equation}}
\newcommand{\ee}{\end{equation}}
\newcommand{\bea}{\begin{eqnarray}}
\newcommand{\eea}{\end{eqnarray}}
\newcommand{\bee}{\begin{eqnarray*}}
\newcommand{\eee}{\end{eqnarray*}}

\def\pa{\partial}

\def\RR{\mathbb{R}}

\def\fref#1{{\rm (\ref{#1})}}

\def\calE{{\mathcal E}}

\def\pa{\partial}

\catcode`@=11
\def\supess{\mathop{\operator@font Sup\,ess}}
\catcode`@=12
\def\la{\langle}
\def\ra{\rangle}

\def\etat{\tilde{\eta}}

\def\bt{\tilde{b}}

\def\RR{\mathbb{R}}

\def\e{\varepsilon}

\def\fref#1{{\rm (\ref{#1})}}

\def\R2+{\RR ^2_+}

\def\lsl{\frac{\lambda_s}{\lambda}}

\def\pa{\partial}

\def\psit{\tilde{\psi}}
\def\lim{\mathop{\rm lim}}

\def\sup{\mathop{\rm sup}}

\def\e{\varepsilon}

\def\l{\lambda}

\def\ln{{\rm log}}

\def\lsl{\frac{\lambda_s}{\lambda}}

\def\Lamdba{\Lambda}

\def\pa{\partial}

\def\psih{\hat{\psi}}

\def\pa{\partial}

\def\H{\mathcal H}

\def\matchal{\mathcal}
\def\Mod{\rm Mod}
\def\at{\tilde{a}}
\def\etah{\hat{\eta}}
\def\lh{\hat{\lambda}}
\def\etaht{\tilde{\etah}}

\title[]{On melting and freezing for the 2d radial Stefan problem}
\author[M. Had\v zi\'c]{Mahir Had\v zi\'c}
\address{King's College London, London, UK} 
\email{mahir.hadzic@kcl.ac.uk}
\author[P. Rapha\"el]{Pierre Rapha\"el}
\address{Universit\'e de Nice Sophia-Antipolis, France\\
Institut Universitaire de France\\
\& European Research Council}
\email{praphael@unice.fr}

\begin{document}
\maketitle

\begin{abstract} 
We consider the two dimensional free boundary Stefan problem describing the evolution of a spherically symmetric ice ball $\{r\leq \l(t)\}$. We revisit the pioneering analysis of \cite{HeVe} and prove the existence in the radial class of finite time {\it melting} regimes 
$$
\l(t)=\left\{\begin{array}{ll} (T-t)^{1/2}e^{-\frac{\sqrt{2}}{2}\sqrt{|\ln(T-t)|}+O(1)}\\
(c+o(1))\frac{(T-t)^{\frac{k+1}{2}}}{|\ln (T-t)|^{\frac{k+1}{2k}}}, \ \ k\in \Bbb N^*\end{array}\right.
\quad\text{ as } t\to T
$$ 
which respectively correspond to the fundamental {\it stable} melting rate, and a sequence of codimension $k\in \Bbb N^*$ excited regimes. Our analysis fully revisits a related construction for the harmonic heat flow in \cite{RS1} by introducing a new and canonical functional framework for the study of type II  (i.e. non self similar) blow up. We also show a deep duality between the construction of the melting regimes and the derivation of a discrete sequence of global-in-time {\it freezing} regimes 
$$
\l_\infty - \l(t)\sim\left\{\begin{array}{ll} \frac{1}{\log t}\\
\frac{1}{t^{k}(\log t)^{2}}, \ \ k\in \Bbb N^*\end{array}\right.
\quad\text{ as } t\to +\infty
$$ 
which correspond respectively to the fundamental {\it stable} freezing rate, and excited regimes which are codimension $k$ stable. 
 \end{abstract}
 

\section{Introduction}



\subsection{Setting of the problem}\label{S:SETTING}


We consider the classical two dimensional one-phase Stefan problem on an {\em external} domain.
The unknowns are the moving domain $\Omega(t)\subset\mathbb{R}^2$ and
the temperature function $u:\Omega(t)\to\mathbb{R}$ which evolve according to:
\be
\label{E:STEFAN}
\left\{\begin{array}{lll}
\pa_tu-\Delta u=0\ \text{ in } \ \Omega(t)\\
\pa_nu=V_{\pa\Omega(t)} \ \text{ on } \ \pa\Omega(t)\\
u=0 \ \text{ on } \ \pa\Omega(t)
\end{array}\right.
\ee
where $V_{\pa \Omega(t)}$ stands for the normal velocity of the moving boundary $\pa\Omega(t)$\footnote{For any given parametrisation $\gamma(t,\cdot):\RR\to\RR^2$ of $\pa\Omega(t)$, the normal velocity is given by the formula 
$V_{\pa\Omega(t)}=\pa_t\gamma\cdot n,$ where $n$ stands for the outward pointing unit normal with respect to $\pa\Omega(t).$}. The temperature $u$ may either be assumed to be positive initially in $\Omega(0)$, in which case  the maximum principle and the Dirichlet boundary condition ensure that it will remain positive in $\Omega(t)$, or on the contrary the data may be undercooled with initially non positive temperature in some regions in space. The cavity represents a circular block of ice kept at constant temperature $u=0.$ If the cavity vanishes at a later time we refer to this process as {\em melting}
and if it expands, we refer to it as {\em freezing}.\\

In many applications it is important to include the effects of surface tension into the description of melting. This is done by replacing the Dirichlet condition in~\eqref{E:STEFAN} by the so-called Gibbs-Thomson correction
\be\label{E:GT}
u = \sigma \kappa_{\pa\Omega(t)} \ \text{ on } \ \pa\Omega(t),
\ee
where $\sigma\ge0$ is known as the coefficient of surface tension and $\kappa_{\pa\Omega(t)}$ is the mean curvature of the moving boundary $\pa\Omega(t)$. For the physical justification of the condition~\eqref{E:GT} we refer the reader to the monograph~\cite{Vi}. When $\sigma>0$ the model can account for phenomena such as superheating (undercooling), phase nucleation, crystal growth etc. It is therefore a small scale description of a phase transition and it adds a $\sigma$-dependent energy contribution from the phase interface, while the classical Stefan problem (corresponding to $\sigma=0$) neglects the energy due to surface tension and represents a {\em macroscopic} description of melting~\cite{HaSh}.\\

In this paper we shall only be interested in the description of melting and freezing within the framework of the one-phase classical Stefan problem $\sigma=0$. While it is fundamental from the physics point of view to understand the singularity formation in both cases, the energy concentration phenomenon displays essential differences related to the type II (for $\sigma=0$) vs type I (for $\sigma\neq 0$) singularity formation process as explained below. 

\subsection{Cauchy theory for the classical Stefan problem} There exists a large literature pertaining to the questions of existence, uniqueness, regularity, and well-posedness for the classical Stefan problem. 
{\em Weak} solutions were first defined and shown to exist in~\cite{Ka} 
and their properties were further studied in many works, see~\cite{Fr1968, Ca77, Ca78, Fr1982,FrKi1975,Vi} and references therein. 
The classical Stefan problem lends itself to a different notion of a weak solution, the so-called {\em viscosity} solutions. For an overview of the seminal works on the 
regularity theory for such solutions we point the reader to~\cite{CaSa} and references therein, and for  further results on existence, uniqueness, and regularity of viscosity solutions to~\cite{Ki,KiPo,sCiK2010,sCiK2012}.  
In the class of {\em classical} solutions, first existence results are traced back to~\cite{Ha,Me}. For well-posedness results in energy-based Sobolev-type spaces see~\cite{HaSh,HaSh2} and in $L^p$-type spaces~\cite{FrSo,PrSaSi}.\\

In this article, we work in a radially symmetric situation and therefore the Cauchy theory is simpler. The domain $\Omega(t)$ is given by 
\begin{align*}
\Omega(t) = \{x\in\mathbb{R}^2; \ |x|\ge \lambda(t)\},
\end{align*}
and the Cauchy problem \eqref{E:STEFAN} becomes:
\be
\label{E:STEFANPOLAR}
\left\{\begin{array}{lll}u_t-u_{rr}-\frac{1}{r}u_r=0\ \text{ in } \ \Omega(t)\\
u_r(t,\lambda(t))=-\dot{\lambda}(t)\\
u(t,\lambda(t))=0 \\
u(0,\cdot) = u_0, \ \l(0) = \l_0.
\end{array}\right.
\ee
It is well posed in $H^2$: for all $(u_0,\l_0)\in H^2\times \Bbb R^*_+$ with $u_0$ radially symmetric, there exists a unique solution 
$(u(t),\l(t))\in \mathcal C([0,T),H^2(\Omega))\times \mathcal{C}^1([0,T),\RR_+^*)$ to \eqref{E:STEFANPOLAR}, and 
$$
T<+\infty\ \ \mbox{implies}\ \ \left(\lim_{t\to T}\|u(t)\|_{H^2(\Omega(t))}=+\infty \ \text{ or } \ \lim_{t\to T}\l(t) = 0\right). 
$$ 
We recall the classical proof\footnote{see for instance~\cite{HeHeVe,Fr1964} and references therein for a Cauchy theory in the class of H\"older spaces.} in Appendix \ref{appendixcauchy}. A simple integration by parts using the boundary conditions \eqref{E:STEFANPOLAR}, see \eqref{E:ENERGY1}, ensures the uniform control of the Dirichlet energy:
\be
\label{decaydirichlet}
\int_{|x|\geq \l(t)}|\nabla u(t,x)|^2dx\leq \int_{|x|\ge \l(0)}|\nabla u_0(x)|^2dx
\ee
in the melting regime $\dot \l <0.$

\subsection{Previous results on melting} The first description of melting regimes is given in the pioneering work by Herrero-Vel\'azquez \cite{HeVe} which predicts the existence of a discrete sequence of melting rates\footnote{A different version of \eqref{E:MELTINGRATE} is computed in \cite{HeVe}, but a correction is suggested in \cite{HeHeVe}. We note that the rate  in~\cite{HeVe} arose due to a minor algebraic mistake, and the same proof would yield the correct rate stated in~\cite{HeHeVe}. We also mention that the correct asymptotic formula is given in~\cite{AnHeVe}.
A similar issue occurs in \cite{HV}, see the correct law in \cite{RSc}.}:
$$
 \lambda(t)\sim \left\{\begin{array}{ll} (T-t)^{1/2}e^{-\frac{\sqrt{2}}{2}\sqrt{|\ln(T-t)|}}\\
 \frac{(T-t)^{\frac{k+1}{2}}}{|\ln (T-t)|^{\frac{k+1}{2k}}},\ \ k\in \Bbb N^*.\end{array}\right.
$$
Their analysis adapts the methodology developed in \cite{Velas, HV} for the construction of non self similar type II blow up bubbles. Following in particular the chemotactic approach \cite{HV}, the authors first consider a change variable related to some partial mass, and analyse the problem as a connection problem between a soliton like behaviour near the origin, and a far out tail connection. As mentioned by the authors themselves \cite{HeHeVe,VICM}, this approach, which relies on a delicate matching procedure cannot as such address the question of stability of the corresponding regimes, and is restricted to radial data.\\

There has been for the past twenty years an immense activity in the field of construction of type II blow up bubbles, in both parabolic and dispersive problems, see for example \cite{P, MR4,KST,RaphRod,MRR, RS1,RS2,RSc,MMR}. The general approach is to split the energy concentration problem into a finite dimensional part which contains the leading order dynamics\footnote{see \eqref{bsyshte} in the strategy of the proof below.}, and an infinite dimensional part which is controlled using purely energy methods. In other words, the matching procedure is replaced by the derivation of the leading order finite dimensional system driving the concentration mechanism, and the full flow is closed using adapted robust energy estimates. This scheme allows both for the description of local manifolds of constructed solutions \cite{collot, MMN}, in particular the stability of the fundamental mode, which is an essential step for the complete description of the flow near the solitary wave \cite{MMR, NSbook}, and it is not in principle restricted to radial data \cite{MMR,MR4, collot2} or scalar problems \cite{MRR}.

\subsection{Main results} The main result of this paper is the existence and stability of a discrete sequence of melting rates with $k$ nonlinear instability directions, $k\in \Bbb N$, which revisits and completes the pioneering construction of \cite{HeVe}.

\begin{theorem}
\label{T:MAIN}[Melting dynamics]
There exists a set of data $(u_0,\l_0)$ in $H^2\times \RR^*_+,$ with $u_0$ arbitrarily small in $\dot{H}^1$, such that
 the corresponding solution $(u,\lambda)\in \mathcal{C}([0,T), H^2)\times\mathcal{C}^1([0,T),\RR_+)$ to the exterior Stefan problem~(\ref{E:STEFANPOLAR}) melts in finite time
 $0\leq T=T(u_0,\l_0)<\infty$ with the following asymptotics:\\
 \noindent\underline{{\em 1. Stable regime}}: the fundamental rate is given by 
  \be\label{E:MELTINGRATE}
 \lambda(t)=(T-t)^{1/2}e^{-\frac{\sqrt{2}}{2}\sqrt{|\ln(T-t)|}+O_{t\to T}(1)}
 \ee
 and is {\em stable} by small radial perturbations in $H^2$; this regime corresponds to positive data $u_0>0$.\\
 \noindent\underline{{\em 2. Excited regimes}}: the excited melting rates are given by
 \be
 \label{meltingk}
 \l(t)=(c^*(u_0, \l_0)+o_{t\to T}(1))\frac{(T-t)^{\frac{k+1}{2}}}{|\ln (T-t)|^{\frac{k+1}{2k}}},\ \ k\in \Bbb N^*,
 \ee
 for some $c^*(u_0,\l_0)>0$; it corresponds to undercooled initial data lying on a locally Lipchitz $H^2$ manifold of codimension $k$.\\
\noindent\underline{{\em 3. Non concentration of energy}}: In all cases, there exists $u^*\in \dot{H^1}$ such that 
\be \label{stronltwoconvergece}
 \lim_{t\to T}\|\nabla u\chi_{\{|x|\ge\l(t)\}}- \nabla u^*\|_{L^2(\RR^2)}=0.
\ee 
 \end{theorem}

{\it Comments on the result.}\\

\noindent{\it 1. Role of the dimension}. This paper deals with the case $d=2,$ which is the energy critical case, but the higher dimensions $d\geq 3$ could be treated by an entirely analogous approach. In fact, the case $d=2$ is the most complicated case, displaying small logarithmic gains only and strong coupling between the various components of the solution\footnote{see for example \eqref{equationmodulation}.}.\\

\noindent{\it 2. Stefan problem with Gibbs-Thomson correction}.  As already mentioned in Section~\ref{S:SETTING}, the Stefan problem with Gibbs-Thomson correction~\eqref{E:GT} ( as a replacement of the classical 
Dirichlet boundary condition $u|_{\pa\Omega(t)}=0$) is an important phase transition model taking surface tension effects into account (which are non-negligible at a certain microscopic spatial scale). Existence (without uniqueness) of global-in-time weak solutions was shown in deep works~\cite{Lu,AlWa} relying on the gradient-flow structure of the problem. In the realm of classical solutions, local well-posedness results as well as global-in-time stability results can be found in~\cite{Ra,EsPrSi, Ha1, HaGu1, PrSiZa}. In the context of melting for the one-phase  Stefan problem with surface tension, to the best of our knowledge the only available result is
an important theorem of Herraiz, Herrero, \& Velazquez~\cite{HeHeVe}. The authors show that radius $\l_{\text{GT}}$ of radially-symmetric melting solutions in the presence of surface tension in dimensions $d=2$ and $d=3$ obeys the
asymptotic law
\be\label{E:SS}
\l_{\text{GT}}(t) \sim_{t\to T} \left(3\sigma(T-t)\right)^{\frac13}.
\ee
The rate~\eqref{E:SS} thus exhibits a very different qualitative behaviour from the Type II rates observed in~\cite{HeVe} (in dimensions $d\ge2$) and in our Theorems~\ref{T:MAIN} and~\ref{T:MAINbis} (in dimension $d=2$). In fact, rate~\eqref{E:SS} honours the self-similar scaling invariance of the related Hele-Shaw problem with surface tension and it is dimension-independent see~\cite{HeHeVe}, and hence belongs to the setting of type I blow up. This stands in contrast to the melting rates for the classical Stefan problem in higher dimensions. 
It is an important open problem to understand whether the rates~\eqref{E:SS} are stable outside the class of radially-symmetric solutions. 
A second important open problem, even in the class of radial solutions, is to prove the existence and describe melting rates in the context of the physically important {\em two-phase} Stefan problem with Gibbs-Thomson correction.\\

\noindent{\it 3. Nonradial dynamics}.
Formal asymptotics for finite time melting is presented in~\cite{AnHeVe} in addition to a wealth of other possible singularity formation scenarios. There the problem is formulated via the Baiocchi transform and the authors identify ellipses in 2D and ellipsoids in 3D as the asymptotic attractors for the melting dynamics in suitably rescaled variables. We also mention the formal asymptotics derived in~\cite{Mc2003,Mc2005}, wherein the same melting phenomenology as in~\cite{AnHeVe} is identified for the Stefan problem confined to a compact domain, melting inwards. \\

Theorem \ref{T:MAIN} lies in the continuation of the methodology developed for the construction of type II blow up bubbles for both parabolic and dispersive problems \cite{MRR,RaphRod,RS1,RS2,RSc,collot}. The strategy consists of two steps: construction of a high order approximate solution based on the expansion of the   blow up/melting solution    with respect to a well chosen small parameter, and development of an energy like method to control the remaining infinite dimensional part. The main novelty of the proof of Theorem \ref{T:MAIN} with respect to these previous works however is the derivation of a {\it new} and sharp functional setting for the construction of type II,  i.e. non self similar blow up bubbles, here applied to a melting problem, which we expect is universal and robust both with respect to the extension to the non radial case and the full classification of type II scenarios. Our analysis relies on new weighted energy bounds with a degenerate Gaussian weight based on the spectral decomposition of the leading order linear operator after a suitable renormalisation, see the strategy of the proof below.
This is conceptually a continuation of the Giga-Kohn approach \cite{GK} to the type I blow up in the energy subcritical range. Our new set of estimates simplifies both the derivation of the approximate solution and the closure of the nonlinear energy bounds by using in an optimal way the dissipative structure of the problem, see in particular \eqref{estemeergy}. 
The existence of a degenerate resonance leading the blow up rate is reminiscent of the derivation of the celebrated "log-log law" for the mass critical nonlinear Schr\"odinger equation, 
see \cite{LPSS, P, MR1, MR4}. 
A recent series of works by Merle and Zaag \cite{MZ1,MZ2,MZ3} suggest that this approach may be of great interest for dispersive wave-like problems as well.\\

Moreover, our approach applies equally well to the construction of   freezing   solutions  (i.e. $\dot\l>0$)   that asymptotically converge 
to a   steady state   $(u=0,\l=\text{const}>0)$. 
 The proof exploits a deep underlying duality between the derivation of the melting and the freezing rates.  

\begin{theorem}
\label{T:MAINbis}[Freezing dynamics]
There exists a set of data $(u_0,\l_0)$ in $H^2\times \RR^*_+,$ arbitrarily small in $\dot{H}^1,$ such that the corresponding solution $(u,\lambda)\in \mathcal{C}([0,T), H^2)\times\mathcal{C}^1([0,T),\RR_+)$ to the exterior Stefan problem~(\ref{E:STEFANPOLAR}) exists globally-in-time. 
 Furthermore 
the solution freezes asymptotically
$$
\lim_{t\to +\infty}\l(t)=\l_\infty>0,
$$ 
where
\be\label{E:FORMULA}
\l_\infty= \sqrt{\l_0^2 - \frac 1\pi \int_{\Omega_0}u_0(x)\,dx},
\ee
and it has the following asymptotic behaviour:  \\
 \noindent\underline{{\em 1. Stable regime}}: the fundamental freezing rate is given by 
  \be\label{E:MELTINGRATEbis}
\l_\infty - \l(t)= \frac{c(u_0,\l_0)(1+o_{t\to +\infty}(1))}{\log t}
 \ee
  for some $c(u_0,\l_0)>0$;
it is {\em stable} with respect to small well localised smooth radial perturbations; this regime contains negative data $u_0<0$.\\
 \noindent\underline{{\em 2. Excited regimes}}: excited melting rates are given by
 \be
 \label{E:MELTINGRATEbisk}
\l_\infty - \l(t) = \frac{c_k(u_0,\l_0)(1+o_{t\to +\infty}(1))}{t^{k}(\log t)^2}, \ \ k\in\mathbb N^*,
 \ee
 for some $c_k(u_0,\l_0)>0$ and correspond to superheated well localised initial data lying on a locally Lipchitz manifold of codimension $k$ in some well localised norm.\\
\noindent\underline{{\em 3. Energy asymptotics}}: in all cases, the Dirichlet energy dissipates at the rate
\be \label{E:MELTINGRATEbis2}
\|\nabla u(t)\|_{L^2(\Omega(t))}=\frac{d_k(u_0,\l_0)(1+o_{t\to +\infty}(1))}{t^{k+1}\log t}, \ \ k\in\mathbb{N},
\ee 
for some $d_k(u_0,\l_0)>0.$
 \end{theorem}

\noindent{\it Comments on the result.}\\

\noindent{\it 1. More melting regimes}. In the setting of the Stefan and Hele-Shaw problems, the authors consider in ~\cite{QuVa2000} a melting scenario for the one-phase Stefan problem outside a fixed domain containing the origin and kept at a pre-fixed non-negative and nontrivial temperature, acting as an effective heat source. The liquid thus expands to infinity for positive initial data and an asymptotic rate of expansion for the free-boundary radius is obtained. Note that this situation is quite different from our setting as there is no such heat source in our case, and the freezing/melting process is driven entirely by the choice of initial conditions.\\

\noindent{\it 2. Solitary wave regimes}. A non trivial global-in-time dynamics with convergence to the solitary wave similar to that described by theorem~\ref{T:MAINbis} has been derived in other critical settings, see for example  \cite{NT2}, \cite{MMR}, \cite{MMN}. The quantised convergence rates with logarithmic corrections are reminiscent of some classical nonlinear dynamical systems scenarios, see for example \cite{galaywayne}.\\

The first main open problem following this work is the understanding of the full non radial stability of the free boundary problem in the stable melting regime $k=0$ which should be amenable to our approach. Let us mention that a related problem 
in the context of evaporating drops was recently studied and solved in the setting of  a {\em self-similar} collapse in the very nice work~\cite{FontHongHwang}. The second main open problem is to give a complete description of the flow for small initial data, and here we expect that the constructions and underlying functional framework of Theorem \ref{T:MAIN} and Theorem \ref{T:MAINbis} are essential steps.

\subsection*{Acknowledgements} The authors thank the anonymous referees for their helpful comments. P.R is supported by the Institut Universitaire de France and the ERC-2014-CoG 646650 SingWave. M.H. kindly acknowledges the hospitality of  the Laboratoire J.A. Dieudonn\'e, Universit\'e de Nice Sophia-Antipolis, where part of this work has been carried out.

\subsection*{Notations} We denote the ball of radius $K$ by
$$
B_K(\Bbb R^{d}) : =\{x\in \Bbb R^{d}, |x|\leq K\}
$$ 
and set
$$
\Lambda : =y\pa_y.
$$ 
For any $\alpha\ge0$ we denote the external domain:
$$\Omega_{\alpha} : = \{x\in \RR^2\big| \ |x|\ge\alpha\}.
$$
When $\alpha = 1,$ we shall simply write $\Omega_1=\Omega.$ We define the weight 
$$
\rho_{\pm}(z) : =e^{\pm\frac{|z|^2}{2}}
$$ and the scalar product on $\Omega_{\sqrt{b}}$ by 
$$
\la f,g\ra_{b,\pm}=\int_{\sqrt{b}}^\infty fg\rho_\pm zdz
$$ 
and the associated norms 
$$
\|u\|_{L^2_{\rho,b,\pm}}=\left(\int_{\sqrt{b}}^\infty u^2\rho_\pm z dz\right)^{\frac 12}, \ \ \|u\|_{H^1_{\rho,b,\pm}}=\left(\|\pa_zu\|_{L^2_{\rho,b,\pm}}^2+\|u\|_{L^2_{\rho,b,\pm}}^2\right)^{\frac 12}
$$ 
and $$\|u\|_{H^2_{\rho,b,\pm}}=\left(\|\Delta u\|_{L^2_{\rho,b,\pm}}+\|\pa_zu\|_{L^2_{\rho,b,\pm}}^2+\|u\|_{L^2_{\rho,b,\pm}}^2\right)^{\frac 12}.$$
We define for $b>0$ the Hilbert space 
\be
\label{defhonrhob}
H^1_{\rho,b,\pm}=\{u:\Omega_{\sqrt b}\to\mathbb{R},\ \ \mbox{$u$ radial with}\ \ \|u\|_{H^1_{\rho,b,\pm}}<+\infty\ \ \mbox{and}\ \ u(\sqrt{b})=0\}
\ee 
equipped with the scalar product $\la \cdot,\cdot\ra_{b,\pm}$, and for $b=0$: 
$$
H^1_{\rho,0,\pm}=\{u:\mathbb{R}^2\to\mathbb R,\ \ \mbox{$u$ radial with}\ \ \|u\|_{H^1_{\rho,0,\pm}}<+\infty\}
$$ 
equipped\footnote{Observe that $u(\sqrt{b})=0$ ensures that $H^1_{\rho,b}$ can be isometrically embedded into $H^1_{\rho,0}$ by considering 
the map $u \mapsto u{\bf 1}_{z\geq \sqrt{b}}$.} with the scalar product $\la \cdot,\cdot\ra_0$. similarly, we define the renormalised quantities:
$$
( f,g)_{b,\pm}=\int_{1}^\infty fg\rho_{b,\pm} ydy, \ \ \rho_{b,\pm}=e^{\pm\frac{b|y|^2}{2}}
$$ 
and the norms 
$$
\|v\|_{b,\pm}=\left(\int_1^\infty v^2\rho_{b,\pm} y dy\right)^{\frac 12}, \ \ \|v\|_{H^1_{b,\pm}}=\|v\|_{b,\pm}+\|\pa_y v\|_{b,\pm}, \ \ \|v\|_{H^2_{b,\pm}}=\|v\|_{H^1_{b,\pm}}+\|\Delta v\|_{b,\pm},
$$ 
and the Hilbert space 
\be\label{E:HBDEF}
H^1_{b,\pm}=\{v:\Omega \to\mathbb{R}, \ \ \mbox{$v$ radial with}\ \ \|v\|_{H^1_{b,\pm}}<+\infty\ \ \mbox{and}\ \ v(1)=0\}.
\ee
We define the sequence of numbers:
\be
\label{edfalphaj}
\alpha_0:=0, \ \ \alpha_j:=\sum_{i=0}^{j-1}\frac{1}{j-i} \ \ \mbox{for} \ \ j\geq 1.
\ee
Throughout the paper, summations over $0\leq j\leq k-1$ are empty for $k=0$.


\subsection{Strategy of the proof}


Problem~\eqref{E:STEFANPOLAR} is invariant under an energy critical  scaling: if $(u,\lambda)$ solves~\eqref{E:STEFANPOLAR}, then so does
\begin{align}\label{E:SELFSIMILAR0}
u_\mu(t,r):=  u(\mu^2t, \mu r), \ \ \lambda_\mu(t) : = \frac{\lambda(t)}{\mu},
\end{align}
and the scaling leaves the Dirichlet energy\footnote{which is dissipated from \eqref{decaydirichlet} in the melting regime.} unchanged. We therefore renormalise the flow
\be
u(t,r)=v(s,y),\ \ \ \frac{ds}{dt}=\frac{1}{\lambda^2(t)}, \ \ \ y=\frac{x}{\lambda(t)}, \label{E:SELFSIMILAR}
\ee
and obtain the renormalised equation with a {\em fixed} boundary:
\be
\label{E:STEFANSS}
\left\{\begin{array}{lll}\pa_sv-\frac{\lambda_s}{\lambda}\Lambda v -\Delta v=0, \ \ y>1;\\
v(s,1)=0 \\
v_y(s,1)=-\frac{\lambda_s}{\lambda}
\end{array}\right.
\ee

\noindent{\bf step 1} Perturbative spectral analysis. We start with the description of melting regimes. Let in a first approximation 
$$b=-\lsl,\ \ 0<b\ll1,
$$ then   for any given $b$   the linear operator driving \eqref{E:STEFANSS} 
is 
$$
\H_b =-\Delta +b\Lambda, \ \ v(1)=0.
$$ 
Our first new input is to diagonalise this operator in a suitable Hilbert space. Indeed, $\H_b$ is self adjoint with respect to the measure $e^{-\frac{b|y|^2}{2}}dy$, and has up to a shift compact resolvent 
in the Hilbert space $H^1_{b,-}$  (see~\eqref{E:HBDEF})   
and hence discrete spectrum. However, the limit $b\to 0$ is singular in the sense that the limiting operator is the Laplacian with resonant eigenmodes and continuous spectrum. After a renormalisation 
 (i.e. rescaling
the space variable by a multiple of $\sqrt b$)  , 
we may equivalently consider 
$$
H_b=-\Delta +\Lambda, \ \ v(\sqrt{b})=0,
$$ 
which formally is a deformation of the standard harmonic oscillator, but with a   Dirichlet boundary   condition. 
We claim that in these renormalised variables, the operator $H_b$ can be diagonalised using a perturbative Lyapounov-Schmidt type argument in the weighted Hilbert space $H^1_{\rho,b,-}.$ 
The first $k$ eigenvalues are given for $0<b<b^*(k)$ by 
\be\label{E:EIGENVALUERATES}
\l_{b,k}=2k+\frac{2}{|\ln b|}+o\left(\frac{1}{|\ln b|}\right), \ \ k\in \Bbb N
\ee
with a corresponding asymptotic expansion of the eigenmodes, see Proposition \ref{lemmadiag}. 
 Unwinding the above renormalisation, we obtain a family of eigenvectors for the operator $\H_b$:  
\be\label{E:ER2}
\H_b\eta_{b,k}=b\l_{b,k}\eta_{b,k}, \ \ \eta_{b,k}(1)=0.
\ee

\noindent{\bf step 2} Approximate solution and modulation equations. 

\medskip
\noindent
\underline{The fundamental melting rate $k=0$.}
For pedagogical purposes, we first summarise our approach specialised to the the case $k=0$ - the derivation of the fundamental melting rate~\eqref{E:MELTINGRATE}. 
Let provisionally
$$
 b=-\lsl, \ \ 0<b\ll1.
$$ 
Following~\cite{RaphRod,RS1,MRR} we look for an approximate solution to \eqref{E:STEFANSS} in the form of a slowly modulated profile 
$$
v(s,y)=v_{b(s)}(y)
$$
 and hence~\eqref{E:STEFANSS} reduces to
$$
b_s\pa_bv_b+\H_bv_b=0, \ \ v_b(1)=0, \ \ \pa_yv_b(1)=b,
$$ 
where $\H_b$ is given above.
The basic observation is that for $b=0$, $\H_0$ has a bound state $\eta_0=\log y$ and a continuous spectrum, but for any $b>0$, 
by~\eqref{E:EIGENVALUERATES}--\eqref{E:ER2} with $k=0$, we are able to compute the bound state for a sufficiently small $0<b<b^*$: 
$$
\H_b\eta_b=b\l_{b,0}\eta_b, \ \ \l_{b,0}=\frac{2}{|\log b|}\left[1+o_{b\to0}(1)\right], \ \ \eta_b(1)=0, 
$$
This furnishes an approximate solution to \eqref{E:STEFANSS}: 
$$
v_b=b\eta_b
$$ 
with the accompanying  law 
$$
b_s+b^2\l_{b,0}=0.
$$ 
The integration of the leading order dynamical system 
\be
\label{dynsmiacal}
\left\{\begin{array}{ll}b_s+\frac{2b^2}{|\log b|}=0\\ \lsl+b=0, \ \ \frac{ds}{dt}=\frac{1}{\l^2}\end{array}\right.
\ee 
leads to the finite time melting with the asymptotics \eqref{E:MELTINGRATE}.

\medskip
\noindent
\underline{The general case $k\in\mathbb N$.}
The computation of the excited melting  regimes is technically more challenging. Due to the new degrees of freedom, 
it is too much to ask for the slowly varying variable $b$ to be approximately equal to $-\frac{\l_s}{\l}$. Instead, we now  fix   
\be
\label{eqbbb}
-\lsl=a, \ \ 0<a\ll 1
\ee 
and rewrite the renormalised flow~\eqref{E:STEFANSS} in the form 
$$
\pa_sv+\H_bv+(a-b)\Lambda v=0, \ \ v(s,1)=0, \ \ \pa_yv(s,1)= a
$$ 
for a parameter $b$ which will be chosen later. 
We look for an approximate solution of the form 
$$
Q(s,y)=\sum_{j=0}^kb_j(s)\eta_{b(s),j}(y),
$$ 
where we recall the definition of $\eta_{b,j}$~\eqref{E:EIGENVALUERATES}.
After projecting onto each eigenmode, we obtain the leading order dynamical system: 
\be\label{bsyshte}
\left\{\begin{array}{ll} (b_j)_s+bb_j\l_{b,k}+\frac{2(a-b)b_j}{|\log b|}+\frac{jb_j}{b}\Phi=0, \ \  0\leq j\leq k\\ \Phi=b_s+2b(a-b).\end{array}\right.
\ee 
This system is complemented by the renormalised nonlinear free boundary condition $\pa_yv(1)=a$ which forces the leading order relationship 
$$
a=\sum_{j=0}^kb_j\left(1+\frac{2\alpha_j}{|\log b|}\right)+{\rm lower \ order \  terms}
$$ 
with $(\alpha_j)_{0\leq j\leq k}$ given by \eqref{edfalphaj}. It remains to chose $b(a)$ which is done through the choice 
\be
\label{choosephi}
\Phi=0
\ee 
which will be motivated below. If $b_k$ dominates over the remaining $b_j$-s, $j=0,\dots,k-1$, the integration of the dynamical system 
\be
\label{estdyna}
\left\{\begin{array}{llll}(b_k)_s+\left(2k+\frac{2}{|\log b|}\right)bb_k+\frac{2(a-b)b_k}{|\log b|}=0\\ b_s+2b(a-b)=0\\ a=b_k\left(1+\frac{2\alpha_k}{|\log b|}\right)\\ \frac{ds}{dt}=\frac{1}{\l^2}, \ \ -\lsl=a\end{array}\right.
\ee 
leads to finite time melting with the rate \eqref{E:MELTINGRATE} for $k=0$ and \eqref{meltingk} for $k\geq 1$, which as solutions to \eqref{estdyna}  possess $k$ unstable   directions.\\

\noindent{\bf step 3} Energy estimate. We now construct a solution of the form 
$$
v(s,y)=\sum_{j=0}^kb_j(s)\eta_{b(s),j}(y)+\e(s,y)
$$ 
with 
\be
\label{orthointro}
 (\e,\eta_{b,j})=0, \ \ 1\leq j\leq k
 \ee 
 and close an energy estimate for the remainder $\e$. Here we note that for a solution to the linear problem
 $$
 \pa_s \e+\H_b\e+(a-b)\Lambda \e=0,
 $$ 
 the time dependance of $b(s)$ yields a modified energy identity 
 $$
 \frac 12\frac{d}{ds}\int \e^2e^{-\frac{b|y|^2}{2}}dy=-(\H_b\e,\e)+\underbrace{(b_s+2b(a-b))}_{=\Phi}\int |y|^2\e^2e^{-\frac{b|y|^2}{2}}dy
 $$  
 and hence the choice of $b$ \eqref{choosephi} to cancel the second term in the energy identity\footnote{which involves a different type of norm for which the spectral gap constant is not explicit and would thus lead to severe difficulties.}. 
To the leading order, thanks to the orthogonality conditions \eqref{orthointro} and the spectral gap estimate in weighted spaces associated to the knowledge of the kernel of $\H_b,$ we obtain the fundamental energy estimate: 
 $$
 \frac 12\frac{d}{ds}\int \e^2e^{-\frac{b|y|^2}{2}}dy=-(\H_b\e,\e)\leq -(2k+2)b\int \e^2e^{-\frac{b|y|^2}{2}}dy.
 $$ 
An integration-in-time will produce the necessary decay to close the bound on the dissipative part of the solution, i.e. $\e$. The situation is however more complicated since the problem cannot close 
at the level of $H^1$ Sobolev regularity, and instead forces us to take one more derivative. However, at the $H^2$ level the corresponding energy identity produces dangerous boundary terms. 
These  come with a particular structure and may beautifully enough be handled through time integration\footnote{This is reminiscent of similar essential issues in \cite{MR1}.}, see Proposition \ref{propenergy}. This part of the analysis is a replacement for the polynomially weighted estimates in \cite{RS1}, and uses in an optimal way the dissipative structure of the equation\footnote{whereas the energy method in~\cite{RS1,RaphRod,MRR} works in both the dissipative and dispersive settings, but barely uses the dissipative terms in the energy estimates.} and the nonlinear algebra induced by the free boundary.\\
 The construction of the manifold of initial data to avoid the codimension k instabilities of the system of ODE's \eqref{bsyshte} is finally performed using a now classical Brouwer type argument as in \cite{cote,collot,RSc,MRRsur}.\\
 
 \noindent{\bf step 4} Freezing. These regimes correspond to an expansion of the circular ice block, reflected in a change of sign in \eqref{E:STEFANSS}, \eqref{eqbbb}: 
 $$
 \lsl=A>0.
 $$ 
This causes a modification in the spectrum of the operator 
 $$
 H_B=-\Delta -B\Lamdba,\ \ v(\sqrt{B})=0, \ \ B>0,
 $$ 
 which admits the eigenvalues:
 $$
 \l_{B,k}=2k+2+\frac{2}{|\ln B|}+o\left(\frac{1}{|\ln B|}\right), \ \ k\in \Bbb N.
 $$ 
Computations parallel to the melting case lead to the dynamical system
 $$
 \left\{\begin{array}{llll}
 (B_k)_s+\left(2k+\frac{2}{|\log b|}\right)BB_k+\frac{2(B-A)B_k}{|\log B|}=0\\ B_s+2B(B-A)=0\\ A=B_k\left(1+\frac{2\alpha_k}{|\log B|}\right)\\ \frac{ds}{dt}=\frac{1}{\l^2}, \ \ \lsl=A
 \end{array}\right.
$$
which after  integration-in-time   produces the global-in-time freezing regimes \eqref{E:MELTINGRATEbis}. 
The energy method is run along similar lines for very well localised initial data,  since the energy spaces are naturally weighted with   the confining measure $e^{\frac{By^2}{2}}\,dy$, $B>0$. 
The analysis is slightly simpler thanks to the better decay of the $B_k$ mode which induces a stronger decoupling from the remaining modes.

\subsection{Plan of the paper}
 In section \ref{sectiongap}, we use a Lyapounov-Schmidt like argument to compute the bound state of $\H_b$ and the associated spectral gap in weighted norms in both the melting regime 
$\frac{\l_s}{\l}<0$, Lemma \ref{lemmarenormalized}, and the freezing regime $\frac{\l_s}{\l}>0$, Lemma \ref{lemmarenormalizedbis}. 
 In section \ref{sectionnonlin}, we construct the melting regimes. We introduce the nonlinear decomposition of the flow, section \ref{sectinogoingeo}, compute the modulation equations using the free boundary  geometry, sections \ref{beivbevbeo} and  \ref{S:MODULATION}, and close the key energy bound, Proposition \ref{pr:bootstrap}. The proof of Theorem \ref{T:MAIN} now follows from a classical shooting argument \`a la Brouwer detailed in section \ref{S:PROOFMAINTHEOREM}. In section \ref{sectionnonlinbis}, we {\em deliberately} follow a parallel plan for the construction of the global-in-time freezing regimes.


\section{Spectral analysis in weighted spaces}
\label{sectiongap}

We compute in this section the $k$ first eigenvalues of the linear operator $$
\H_{b,\pm}=-\Delta \mp  b\Lambda\ \ \mbox{with boundary condition}\ \ u(1)=0
$$ 
and the associated spectral gap estimate in the perturbative regime $0<b<b^*(k)$,  $b^*(k)\ll1$  . The proof relies on a Lyapunov-Schmidt type bifurcation argument at $b=0$ performed in weighted Sobolev spaces.\\

To ease notations, we fix 
\be\label{E:CONVENTION}
\pm =-, \ \ \H_{b}=\H_{b,-}, \ \ \rho=\rho_-=e^{-\frac{|z|^2}{2}}, \ \ b>0,  \ \ \mbox{in sections \ref{sectiontwo}, \ref{sectionthree}, \ref{sectionfour}}
\ee
and we omit the $-$ subscript for the sake of simplicity. The case $b<0$ with the $\rho_+$ weight and the operator $\H_{b,+}$ is addressed in section \ref{sectionfive}.


\subsection{Coercivity for the harmonic oscillator}
\label{sectionone}

We recall in this section without proof the classical estimates for the harmonic oscillator.\\

\noindent{\bf Melting case}: consider $-\Delta +\Lambda$ on $(H^1_{\rho_-,0},\la \cdot,\cdot\ra_0)$. This operator is self adjoint for the $\la \cdot,\cdot \ra_{0,-}$ scalar product as is easily seen by writing
\begin{align}\label{E:HBFORMULA}
-\Delta +\Lambda= -\frac{1}{\rho_- z}\pa_z\left(\rho_- z\pa_z\right).
\end{align}
The normalised Laguerre polynomials~\cite{Sz}
\be
\label{estaimti}
L_k(x)=\frac{e^x}{k!}\frac{d^k}{dx^k}(e^{-x}x^k),\ \ k\in \Bbb N,
\ee 
solve $$XL_k''+(1-X)L'_k+kL_k=0$$ and hence 
\be
\label{cneneoe}P_k(r)=L_k\left(\frac{r^2}{2}\right)
\ee diagonalises the harmonic oscillator:
\be
\label{orthooncf}
(-\Delta +\Lambda)P_k=2k P_k, \ \ \la P_n,P_m\ra_{\rho,0}=1.
\ee
Moreover, they satisfy the double normaliaation condition:
$$
\int_{0}^{+\infty}L_n L_m e^{-x}dx=\delta_{nm}, \ \ L_n(0)=1
$$
or equivalently 
\be
\label{normalization}
\la P_j,P_k\ra_{0,-}=\delta_{jk}, \ \ P_k(0)=1,
\ee
and the classical induction formula 
\be
\label{induction}
\Lambda P_k=2k(P_k-P_{k-1}),  \ \ k\geq 1.
\ee
An extensive overview of Laguerre polynomials can be found in~\cite{Sz}.
We recall the standard sharp Poincar\'e inequality for the harmonic oscillator: $\forall u\in H^1_{\rho_-,0}$ with 
$$
\la u,P_j\ra_{0,-}=0, \ \ 0\leq j\leq k,
$$ 
there holds:
\be
\label{estimationcoreor} 
\|\pa_zu\|_{L^2_{\rho_-,0}}\geq (2k+2)\|u\|_{L^2_{\rho_-,0}}^2.
\ee

\noindent{\bf Freezing case}: Consider $-\Delta -\Lambda$ on $(H^1_{\rho_+,0},\la \cdot,\cdot\ra_0)$. Then the map
$$
\begin{array}{ll} L^2_{0,+}\to L^2_{0,-}\\ v\mapsto w=e^{\frac{|z|^2}{2}}v\end{array}
$$ is an isometry and integrating by parts: 
\be
\label{comptuationintergation}
\int |\nabla v|^2e^{\frac{|z|^2}{2}}\rho_{+}dz=\int |\nabla w|^2\rho_{-}dz+2B\int |w|^2\rho_{B,-}dz
\ee
or equivalently:
\be
\label{commutationrelation}
(-\Delta -\Lambda)v=\left(-\Delta w+\Lambda w+2 w\right) e^{-\frac{|z|^2}{2}}.
\ee
Hence the family of eigenvectors $$\hat{P}_j=P_je^{-\frac{|z|^2}{2}}$$ diagonalises the operator, and there holds the spectral gap estimate:
 $\forall u\in H^1_{\rho_+,0}$ with 
$$
\la u,\hat{P}_j\ra_{0,+}=0, \ \ 0\leq j\leq k,
$$ 
there holds:
\be
\label{estimationcoreorbis} 
\|\pa_zu\|_{L^2_{\rho_+,0}}\geq (2k+4)\|u\|_{L^2_{\rho_+,0}}^2.
\ee

\begin{remark} The shift $2=d$ in \eqref{estimationcoreorbis}, where $d$ stands for the dimension of the ambient space,  will be crucial for the computation of the freezing rates.
\end{remark}


\subsection{Almost invertibility of the renormalised operator}
\label{sectiontwo}

Recall the notational convention~\eqref{E:CONVENTION}. We consider the renormalised operator 
$$
H_b=-\Delta+\Lambda\ \ \mbox{with boundary condition}\ \ u(\sqrt{b})=0
$$ 
in the radial sector and for $0<b<b^*$ small enough. Thanks to the boundary condition $u(\sqrt{b})=0$ and \eqref{E:HBFORMULA}, $H_b$ is self adjoint for the scalar product $\la \cdot,\cdot \ra_b$ on the Hilbert space $H^1_{\rho,b}$ given by \eqref{defhonrhob}. We claim a near invertibility property of $H_b$ which is the starting point of the Lyapunov Schmidt argument.

Before stating the lemma, we introduce some notations. We first fix a frequency size 
$$
K\in \Bbb N
$$ 
and a sufficiently small parameter $$0<b<b^*(K)\ll1 .
$$ 
Universal constants in the sequel may depend on $K$, but are uniform in $b\in(0,b^*(K))$.
 We define the Gramm matrix 
 \be
\label{grammmatric}
M_{b,k}=(\la P_i,P_j\ra_b)_{0\leq i,j\leq k}, \ \ k\le K.
\ee 
Observe from \eqref{orthooncf} that 
\be
\label{almostorgho}
\la P_i,P_j\ra_b=\la P_i,P_j\ra_0+O(\int_{|z|\leq \sqrt{b}}zdz)=\delta_{ij}+O(b)
\ee and hence 
\be
\label{inviennoe}
M_{b,k}=Id+O(b)\ \ \mbox{is invertible}
\ee for $0\leq b<b^*(k)$ small enough. We introduce the vector: 
$$
\matchal P_k=(P_j)_{0\leq j\leq k}
$$ 
and consider the function 
\be
\label{cnekvnenvo}
m_k(b,z)=(M_{b,k}^{-1}\mathcal P_k(\sqrt{b}),\mathcal P_k(z)),
\ee 
which by~\eqref{estaimti} and~\eqref{inviennoe} satisfies:
\be
\label{estmnbk}
m_k(b,z)=\sum_{j=0}^k\left[1+O(b)\right]P_j(z).
\ee
We now claim:

\begin{lemma}[Near inversion of $H_b-2k$]
\label{lemmanerinversion}
Let $k\in \Bbb N$ and $0<b<b^*(k)$ small enough. Then for all $f\in L^2_{\rho,b}$ with 
 \be
 \label{cneneonevno}
 \la f,P_j\ra_b=0, \ \  0\leq j\leq k,
 \ee there is a unique solution $u\in H^1_{\rho,b}$ to:
\be
\label{cneoneoenveo}
\left\{\begin{array}{ll}\tilde{H}_{b,k}u=f \ \ \mbox{where}\ \ \tilde{H}_{b,k}u=(H_b-2k)u-\sqrt{b}m_k(b,z)\pa_zu(\sqrt{b})\\
\la u,P_j\ra_b=0, \ \ 0\leq j\leq k.
\end{array}\right.
\ee 
Moreover, 
\be
\label{inversionestimate}
\|\Delta u\|_{L^2_{\rho,b}}+\|\pa_zu\|_{L^2_{\rho,b}}+\|\Lambda u\|_{L^2_{\rho,b}}+\|u\|_{L^2_{\rho,b}} + |\ln b| \big|\sqrt{b}\pa_z u(\sqrt{b})\big|
\lesssim \|f\|_{L^2_{\rho,b}}.
\ee
\end{lemma}

\begin{proof}[Proof of Lemma \ref{lemmanerinversion}] We use a Lax Milgram type argument in $H^1_{\rho,b}$. Let $k\in \Bbb N$ and define the constraint set
\[
\mathcal{C} : = \{u \in H^1_{\rho,b} \, \big| \la u,P_j\ra_b=0, \ \ 0\leq j\leq k  \}.
\]
We consider the problem of minimising the functional $\mathcal{F}:H^1_{\rho,b}\to\mathbb{R}:$
$$\mathcal{F}(u)=\int_{z\geq \sqrt{b}}|\pa_zu|^2\rho zdz-2k\int_{z\geq \sqrt{b}}u^2\rho zdz-\la f,u\ra_b$$
over the constraint set $u\in \mathcal{C}.$  Let 
$$
I_b = \inf_{u\in \mathcal{C}} \mathcal{F}(u).
$$
We recall from the standard Poincar\'e inequality for the harmonic oscillator \eqref{estimationcoreor} and the compactness estimate \eqref{weightedesimate} that for spherically symmetric $v\in H^1_{\rho,0}$ with $\la v,P_j\ra_0=0$, $0\leq j\leq k$, there holds:
\be
\label{weghtedpoinacre}
 \int |\pa_zv|^2\rho z dz-2k\int v^2\rho zdz\gtrsim \int(1+|z|^2)|v|^2\rho zdz.
\ee
 Applying this to $v(z)=u{\bf 1}_{z\geq \sqrt{b}}\in H^1_{\rho,0}$, we conclude that $I_b>-\infty$ and that any minimising sequence $u_n$ is uniformly bounded in $H^1_{\rho,b}.$ 
Therefore, up to a subsequence, using the compact Sobolev embedding $H^1_{\rho,b}\hookrightarrow L^2_{\rho,b}$ and \eqref{weghtedpoinacre}, we conclude: 
$$u_n\rightharpoonup u \ \ \mbox{in}\ \ H^1_{\rho,b}, \ \ u_n\to u\ \ \mbox{in}\ \ L^2_{\rho,b}.$$ 
In particular, using the local compactness of the embedding $H^1(\Bbb R)\subset L^\infty(\Bbb R)$, this implies 
$$u(\sqrt{b})=0, \ \ \la u,P_j\ra _b=0, \ \ 0\leq j\leq k$$ 
and $u$ is a minimiser of $\mathcal{F}$ over $\mathcal{C}.$ 
By a standard variational argument, there exist  Lagrange multipliers $\l_j\in\mathbb{R}$ such that 
\be
\label{cneokneonove}
H_bu=f-\sum_{j=0}^k \l_jP_j.
\ee
Hence from standard regularity argument, $u\in H^2_{\rm loc}(r\geq \sqrt{b})$. We may then take the scalar product with $P_i$ and compute: 
\bee
\la H_bu,P_i\ra_b&=&-\int_{z\geq \sqrt{b}}\frac{1}{z \rho}\pa_z(z\rho \pa_zu)P_i\rho zdz=\sqrt{b}e^{-b/2}P_j(\sqrt{b})\pa_zu(\sqrt{b})+\la u,H_bP_i\ra_b\\
& = & \sqrt{b}e^{-b/2}P_i(\sqrt{b})\pa_zu(\sqrt{b})+2i\la u,P_i\ra_b=\sqrt{b}e^{-b/2}P_i(\sqrt{b})\pa_zu(\sqrt{b}).
\eee
We conclude from \eqref{cneokneonove}, \eqref{cneneonevno}, \eqref{grammmatric} that
$$\sqrt{b}e^{-b/2}\pa_zu(\sqrt{b})\left(P_i(\sqrt{b})\right)_{0\leq i\leq k}=-M_{b,k}(\l_i)_{0\leq i\le k}$$ or equivalently $$(\l_i)_{0\leq i\le k}=-\sqrt{b}e^{-b/2}\pa_zu(\sqrt{b})M_{b,k}^{-1}\mathcal P_k(\sqrt{b})$$ and hence $u$ solves \eqref{cneoneoenveo} from the definition \eqref{cnekvnenvo} and \eqref{cneokneonove}. We now observe that $u\in \mathcal C$ ensures $$\la m_{k}(b,\cdot),u\ra_b=0$$ and hence 
taking the scalar product of~\eqref{cneoneoenveo} with $u$, using~\eqref{weghtedpoinacre} with $v=u{\bf 1}_{z\geq \sqrt{b}}$, and the identity $\la H_b u, u\ra_b = \|\pa_z u\|_{L^2_{\rho,b}}^2$ yields: 
$$
 \|u\|_{L^2_{\rho,b}}^2\lesssim \|\pa_zu\|_{L^2_{\rho,b}}^2-2k\|u\|_{L^2_{\rho,b}}^2  =\la f,u\ra_b$$
 and hence
\be
\label{poeuetepot}
 \|u\|_{L^2_{\rho,b}} +\|\pa_zu\|_{L^2_{\rho,b}}  \lesssim \|f\|_{L^2_{\rho,b}}.
 \ee
 We now integrate by parts to compute:
 \be
 \label{pofeueutoiurte}
 \la H_b u,\ln z\ra_b = \la u,1 \ra_b  -\frac 12 |\ln b| \sqrt{b}e^{-b/2}\pa_zu(\sqrt{b}).
 \ee
 We estimate from \eqref{estmnbk}
 \be\label{E:MBKBOUND}
 \|m_{k}(b,\cdot)\|_{L^2_{\rho,b}}\lesssim 1
 \ee
and hence \eqref{pofeueutoiurte}, \eqref{cneoneoenveo} ensure $$
 \big|\sqrt{b}\pa_zu(\sqrt{b}) \big|  \lesssim \frac 1{|\ln b|} \left(\|H_b u\|_{L^2_{\rho,b}} + \|u\|_{L^2_{\rho,b}}\right)  \lesssim  \frac 1{ |\ln b|} \left(\|f\|_{L^2_{\rho,b}} + \big|\sqrt{b}\pa_z u(\sqrt{b})|\right)
 $$ 
which together with \eqref{poeuetepot} yields 
$$
\|\pa_zu\|_{L^2_{\rho,b}}+\|u\|_{L^2_{\rho,b}} + |\ln b| \big|\sqrt{b}\pa_z u(\sqrt{b})\big|
\lesssim \|f\|_{L^2_{\rho,b}}.
$$ 
We now use the equation \eqref{cneoneoenveo} again and teh bound~\eqref{E:MBKBOUND} which yield
$$
\|H_bu\|_{L^2}\lesssim \|f\|_{L^2_{\rho,b}}.
$$ 
We then use a Pohozhaev type integration by parts to compute:
\bee
\int_{z\geq \sqrt{b}}z\pa_z u(H_bu)z\rho dz&=&-\int_{z\geq \sqrt{b}}\pa_z(z\rho\pa_zu)z\rho\pa_zu\frac{dz}{\rho}\\
& = & -\left[\frac12 z^2\rho^2(\pa_zu)^2\frac{1}\rho\right]_{\sqrt{b}}^{+\infty}-\frac 12\int_{z\geq \sqrt{b}} z^2\rho^2(\pa_zu)^2\frac{\pa_z\rho}{\rho^2}dz\\
&= & \frac12be^{-b/2}|\pa_zu(\sqrt{b})|^2 +\frac 12\int_{z\geq \sqrt{b}}(\Lambda u)^2z\rho dz\\ 
& = & O\left(\|f\|^2_{L^2_{\rho,b}}\right)+ \frac 12\int_{z\geq \sqrt{b}}(\Lambda u)^2z\rho dz.
\eee
which after a simple application of the Cauchy-Schwarz inequality yields $\|\Lambda u\|_{L^2_{\rho,b}} \lesssim \|f\|_{L^2_{\rho,b}}$ 
and \eqref{inversionestimate} is proven.
\end{proof}


\subsection{Partial diagonalisation of $H_b$}
\label{sectionthree}


We are now in position to diagonalise $H_b$ for frequencies $0\leq k\leq K$ under the smallness $0<b<b^*(K)$. 

\begin{proposition}[Eigenvalues for $H_b$]
\label{lemmadiag}
Let $K\in \Bbb N$. Then for all $0<b<b^*(K)$ small enough, $H_b$ admits a sequence of eigenvalues
\be\label{E:EIGENVALUEEQUATION}
H_b\psi_{b,k}=\l_{b,k}\psi_{b,k},\ \ \psi_{b,k}\in H^1_{\rho,b}, \ \ 0\leq k \leq K,
\ee
such that for each $0\le k\le K,$ the following properties hold:\\
\noindent{\em (i) Expansion of eigenvalues}: there holds the expansion of the eigenvalue 
\be
\label{estimatelowe}
\l_{b,k}=2k+\frac{2}{|\ln b|}+\tilde{\l}_{b,k}, \ \ \tilde{\l}_{b,k}=O\left(\frac{1}{|\ln b|^2}\right), \ \ |\pa_b\l_{b,k}|\lesssim \frac{1}{b|\ln b|^2}.
\ee
\noindent{\em (ii) Expansion of eigenvectors}: there holds the expansion
\be
\label{deftk}
\left\{\begin{array}{ll}\psi_{b,k}=T_{b,k}(z)+\psit_{b,k}(z)\\
T_{b,k}(z)=P_k(z)\left[\ln z-\ln(\sqrt{b})\right]+\sum_{j=0}^{k-1}\mu_{b,jk}P_j(z)\left[\ln z-\ln(\sqrt{b})\right]
\end{array}\right.
\ee
with
\be
\label{estmuj}
\mu_{b,jk}= \frac{2}{(k-j)|\ln b|}+\tilde{\mu}_{b,jk}, \ \ \tilde{\mu}_{b,jk}=O\left(\frac{1}{|\ln b|^2}\right), \ \ b\pa_b\tilde{\mu}_{b,jk}=O\left(\frac{1}{|\ln b|^3}\right)
\ee
and
\begin{align}
& \|\Delta \psit_{b,k}\|_{L^2_{\rho,b}}+\|z^2\psit_{b,k}\|_{L^2_{\rho,b}}+\|\tilde{\psi}_{b,k}\|_{H^1_{\rho,b}}+\|\Lambda \psit_{b,k}\|_{L^2_{\rho,b}}+\|b\pa_b\tilde{\psi}_{b,k}\|_{L^2_{\rho,b}} \notag \\
 & \ \ \ \ +|\ln b| \big|\sqrt{b}\pa_z\tilde{\psi}_{b,k}(\sqrt{b})\big|  \lesssim \frac{1}{|\ln b|}, \label{calculeignvector}
\end{align}
\be
\label{estimatedervaitveinb}
\|\pa_b\psi_{b,k}\|_{H^1_{\rho,b}}+|\ln b||\sqrt{b}\pa_b\pa_z\psi_{b,k}(\sqrt{b})|\lesssim \frac{1}{b},
\ee
and
\be
\label{estimatemomentum}
\|\Lambda \psi_{b,k}\|_{L^2_{\rho,b}}+ \|z^2\psi_{b,k}\|_{L^2_{\rho,b}}\lesssim  |\ln b|.
\ee
\noindent{\em (iii) Spectral gap estimate}: let $u\in H^1_{\rho,b}$ with $$\la u,\psi_{b,j}\ra_b=0, \ \ 0\leq j\leq k,$$ then
\be
\label{spcetralgap}
\|\pa_zu\|^2_{L^2_{\rho,b}}\geq \left[2k+2+O\left(\frac1{|\ln b|}\right)\right]\|u\|^2_{L^2_{\rho,b}}.
\ee 
Moreover, 
\be
\label{neneone}
\l_{b,0}=\inf_{u\in H^1_{\rho,b}}\frac{\|\pa_zu\|^2_{L^2_{\rho,b}}}{\|u\|_{L^2_{\rho,b}}^2}
\ee and 
\be
\label{postiivitypsib}
\psi_{b,0}(z)>0\ \ \mbox{for}\ \ z>\sqrt{b}.
\ee
\noindent{\em (iv) Further identities}: there hold the algebraic identities for $0\leq j\leq k$:
\be
\label{scalirpeoduc}
\la \psi_{b,k},\psi_{b,k}\ra_b=\frac{|\ln b|^2}{4}\left[1+O\left(\frac1{|\ln b|}\right)\right].
\ee
\be
\label{idenitnirninebis}
\frac{\la b\pa_b\psi_{b,j},\psi_{b,k}\ra_b}{\la \psi_{b,k},\psi_{b,k}\ra_b}=-\frac{1}{|\ln b|}\delta_{jk}+O\left(\frac{1}{|\ln b|^2}\right),
\ee
Moreover,
\bea
\label{exactcomputation}
&&\|\Lambda \psi_{b,k}-2k(\psi_{b,k}-\psi_{b,k-1})\|_{H^2_{\rho,b}}\lesssim 1,\\
&& \|b\pa_b\psi_{b,k}+\frac{1}{|\ln b|}\psi_{b,k}\|_{H^2_{\rho,b}}\lesssim \frac{1}{|\ln b|},\label{vnbevbnenonev}\\
&&\|z^2 \psi_{b,k}+(2k+2)\psi_{b,k+1}-(4k+2)\psi_{b,k}+2k\psi_{b,k-1}\|_{H^2_{\rho,b}}\lesssim 1.
\label{vnbevbnenonevbis}
\eea
\end{proposition}

\begin{remark} From standard Sturm-Liouville oscillation argument, $\psi_{b,k}$ vanishes $k$ times on $z>\sqrt{b}$, and hence only the ground state $\psi_{b,0}$ is nonnegative.
\end{remark}

\begin{remark} Constants in Lemma \ref{lemmadiag} depend on the frequency $K$ but are uniform in $0<b<b^*(K)$.
\end{remark}

\begin{proof}[Proof of Proposition \ref{lemmadiag}] The argument is of Lyapunov-Schmidt type. We remove the $b$ subscript as much as possible to simplify notations.\\

\noindent {\bf step 1} The Lyapunov-Schmidt argument. Let $T_k(z)\in H^1_{\rho,b}$ be given by \eqref{deftk},
and introduce the universal profile
\[
\mathcal{T}_b(z) : = \ln z-\ln(\sqrt{b}).
\]
Then
\be
\label{copmjeouegov}
\left\{\begin{array}{ll}(H_b-2k)T_k(z)=Q_k(z)+\sum_{j=0}^{k-1}\mu_{jk}\left[2(j-k)\mathcal{T}_b(z)P_j(z)+Q_j\right]\\
Q_j=-2\frac{P'_j(z)}{z}+P_j(z).
\end{array}\right.
\ee
Observe from \eqref{cneneoe} that $Q_j=\tilde{Q}_j(r^2)$ with ${\rm deg} \, \tilde{Q}_j\leq j-1$ and hence 
\be
\label{spanrealtion}
\forall 0\leq j\leq k, \ \ Q_j\in \mbox{\rm Span}(P_j)_{0\leq j\leq k}.
\ee
We solve the eigenvalue problem 
\be\label{E:EIGENVALUEPROBLEM}
(H_b-2k)\psi=\mu_k\psi,
\ee
by representing $\psi$ in the form
\be\label{E:PSIDECOMPOSITION}
\psi=T_k+\tilde{\psi}.
\ee
and hence \eqref{E:EIGENVALUEPROBLEM}, \eqref{E:PSIDECOMPOSITION} give an equation for $\psit:$ 
\bea
\label{E:PSI1EQUATION}
\nonumber (H_b-2k)\tilde{\psi}  &=&-(H_b-2k)T_k+\mu_k(T_k+\psit)\\
\nonumber& =& -Q_k(z)-\sum_{j=0}^{k-1}\mu_{jk}\left[2(j-k)\mathcal{T}_b(z)P_j(z)+Q_j\right]+\mu_k(T_k+\tilde{\psi})\\
&= :&   f(\psit).
\eea
We define $(\mu_j(\psit))_{0\leq j\leq k}$ by imposing the relations:
\be\label{E:DEFINITIONLAMBDA}
\frac{\la f(\psit),P_i\ra_b}{\la P_i,P_i\ra_b}=\sqrt{b}\pa_z\psit(\sqrt{b})\left(M_{b,k}^{-1}\mathcal P_k(\sqrt{b})\right)_i,  \ \ i=0,\dots,k,
\ee 
which is proved below to correspond to an invertible linear system on $((\mu_{j,k})_{0\leq j\leq k-1},\mu_k)$.
Observe that \eqref{E:DEFINITIONLAMBDA} allows us to rewrite \eqref{E:PSI1EQUATION} as:
\be
\label{defhtildie}
\tilde{H}_{b,k}\psit=f-\sum_{j=0}^k\frac{\la f,P_j\ra_b}{\|P_j\|_{L^2_{\rho,b}}^2}P_j=F(\psit)
\ee
with $\tilde{H}_{b,k}$ given by \eqref{cneoneoenveo}.
Thus to find $\psit,$ by Lemma~\ref{lemmanerinversion}
we are left with solving the fixed point equation 
\[
\psit=\tilde{H}_{b,k}^{-1}F(\psit).
\] 
and we indeed claim that the operator $\tilde{H}_{b,k}^{-1}\circ F$ is a strict contraction from the closed ball 
\begin{align}
B_\alpha : =  \Big\{\psit\in H^1_{\rho,b}\Big| & \ \|\Delta \psit\|_{L^2_{\rho,b}}+\sqrt b|\pa_z\psit(\sqrt b)|+\|\psit\|_{H^1_{\rho,b}}\leq \frac{\alpha}{|\ln b|}, \notag\\
& \ \ \psit(\sqrt b)=0,\ \ \la \psit,P_j\ra_b=0, \ \ 0\leq j\leq k\Big\}, \label{deballbalpha}
\end{align} 
to itself for $\alpha$ universal large enough.\\
\noindent\underline{\it Computation of $\mu_{jk},\mu_k$}: We invert \eqref{E:DEFINITIONLAMBDA}. Indeed, we rewrite 
\bee
f(\psit)&=&-Q_k+\frac{|\ln b|}{2}\mu_kP_k\left[1+\frac{2\ln z}{|\ln b|}\right]+\mu_k\tilde{\psi}\\
& - & \sum_{j=0}^{k-1}\frac{|\ln b|}{2}\mu_{jk}\left\{\left[2(j-k)-\mu_k\right]P_j\left[1+\frac{2\ln z}{|\ln b|}\right]+\frac{2Q_j}{|\ln b|}\right\}.
\eee
Let 
$$
\vec{\mu} = (\mu_{0k},\dots,\mu_{(k-1)k}, \mu_k)  \ \text{ and } \ |\vec \mu|_\infty : =\max \{|\mu_{jk}|_{0\leq j\leq k-1},|\mu_k|\}.
$$
Using the almost orthogonality~\eqref{almostorgho} and the $\la\cdot,\cdot\ra_b$-orthogonality of $\psit$ and $P_j,$ $j=0,\dots k,$ we obtain:
\begin{align}
\frac{\la f(\psit),P_j \ra_b}{\la P_j,P_j\ra_b}  & = (k-j)|\ln b|\mu_{jk}-\la Q_k,P_j\ra_0+ (C\vec{\mu})_j 
+ \mu_k\sum_{\ell=0}^{k-1}\mu_{\ell k}\la P_j,P_\ell\mathcal T_b\ra_b + O(|b|), \label{E:MUJK}\\
\frac{\la f(\psit),P_k\ra_b}{\la P_k,P_k\ra_b}&=\mu_k\frac{|\ln b|}{2}-\la Q_k,P_k\ra_0
+ \mu_k\sum_{\ell=0}^{k-1}\mu_{\ell k}\la P_k,P_\ell\mathcal T_b\ra_b+ (C\vec{\mu})_k+  O(|b|), \label{E:MU}
\end{align}
where 
$
C = (C_{ij})_{i,j=0,\dots k} = O(1),
$
is a $(k+1)\times(k+1)$-matrix with bounded entries in the regime where $b$ is small and
$
(C\vec{\mu})_i = \sum_{\ell=0}^{k-1} C_{i\ell}\mu_{\ell k} + C_{ik} \mu_k
$
is the $i$-th entry of the vector $C\vec{\mu},$ $i=0,\dots,k$.
The above system contains quadratic terms
and it can be solved for $\vec{\mu}$ by the following simple iteration argument.
Given $\vec{\tilde \mu}= (\tilde \mu _{0k},\dots.\tilde{\mu}_{(k-1)k},\tilde \mu_k)$, consider the system
\begin{align*}
\frac{\la f(\psit),P_j \ra_b}{\la P_j,P_j\ra_b}  & = (k-j)|\ln b|\mu_{j,k}-\la Q_k,P_j\ra_0+ (\tilde C\vec{\mu})_j 
+ O(|b|), \\
\frac{\la f(\psit),P_k\ra_b}{\la P_k,P_k\ra_b}&=\mu_k\frac{|\ln b|}{2}-\la Q_k,P_k\ra_0
+ (\tilde C\vec{\mu})_k+  O(|b|),
\end{align*}
where $\tilde{C}_{ij} = C_{ij} +\delta_{ik}\sum_{\ell=0}^{k-1}\tilde{\mu}_{ik}\la P_j,P_\ell\mathcal T_b\ra_b,$ $j=0,\dots k-1.$
Assuming that $|\vec{\tilde \mu}|_\infty \le 1,$ we have $\tilde C=O(1)$ and we can invert the above system 
to obtain
\bea
\label{compmujk}
&&\mu_{jk}=\frac{\la Q_k,P_j\ra_0}{(k-j)|\ln b|}+O\left(\frac{1}{|\ln b|^2}+\frac{\sqrt{b}|\pa_z\psit(\sqrt{b})|}{|\ln b|}\right), 0\leq j\leq k-1,\\
\label{vcompmuk}
&& \mu_k=\frac{2\la Q_k,P_k\ra_0}{|\ln b|}+O\left(\frac{1}{|\ln b|^2}+\frac{\sqrt{b}|\pa_z\psit(\sqrt{b})|}{|\ln b|}\right)
\eea
and therefore $|\vec \mu|_\infty \lesssim \frac 1{|\ln b|} \le 1$ for $b<b^*$ sufficienty small, where we used~\eqref{deballbalpha}. 
Contractive property follows easily in a similar manner and for a given $\psit\in B_\alpha$ we obtain the unique solution
$\vec \mu.$ 
We may moreover integrate by parts and use \eqref{normalization} to compute:
\bee
\la Q_k,P_k\ra_0=\la -2\frac{P_k'}{z}+P_k,P_k\ra_0=P_k^2(0)=1.
\eee
similarly for $j\leq k-1$:
$$
\la Q_k,P_j\ra_0=-2\la \frac{P'_k}{z},P_j\ra_0=2+\la\frac{P'_j}{z},P_k\ra_0=2
$$
since $\frac{P'_j}{z}$ is a polynomial of $r^2$ of degree $\leq j-1<k$. Note that by~\eqref{compmujk} and~\eqref{vcompmuk} we additionally have the bound
\be\label{E:MUINFINITYBOUND}
|\vec \mu|_\infty \lesssim \frac 1 {|\ln b|}+\frac{\sqrt{b}|\pa_z\psit(\sqrt{b})|}{|\ln b|}.
\ee

\noindent\underline{\it Estimating $F(\psit)$}: 
 Let us introduce the approximate projection operator:
$$
\Bbb P_kf : =\sum_{j=0}^k\frac{\la f,P_j\ra_b}{\|P_j\|_{L^2_{\rho,b}}^2}P_j.
$$
Observe that $(\text{I} - \Bbb P_k) P_j = O(b)$ and by~\eqref{spanrealtion}
$(\text{I} - \Bbb P_k) Q_j = O(b)$  for any $j\in\{0,1,\dots, k\}$ and the almost orthogonality relation~\eqref{almostorgho}. 
From~\eqref{deballbalpha} we have that $(\text{I} - \Bbb P_k) \tilde \psi = \tilde\psi.$
Thus, from~\eqref{E:PSI1EQUATION} and~\eqref{defhtildie} we obtain that
\begin{align*}
F(\psit)  & =  \mu_k (\text{I} -\Bbb P_k)\left[P_k\ln z + P_k\frac{|\ln b|}{2}\right]+\mu_k\psit \\
& \ \ \ +\sum_{j=0}^{k-1}\mu_{j,k}\left[2(k-j)+\mu_k\right](\text{I}-\Bbb P_k)\left[P_j\ln z + P_j\frac{|\ln b|}2\right] + O(b(1+|\vec \mu|_\infty))\\
& =\mu_k (\text{I} -\Bbb P_k)\left[P_k\ln z\right]+\mu_k\psit 
 +\sum_{j=0}^{k-1}\mu_{j,k}\left[2(k-j)+\mu_k\right](\text{I}-\Bbb P_k)\left[P_j\ln z\right] \\
 & \ \ \ + O\left(b(1+|\vec \mu|_\infty) + b|\ln b||\vec \mu|_\infty \right).
\end{align*}
Therefore
\begin{align}
\|F(\psit)\|_{L^2_{\rho,b}} & \lesssim |\vec \mu|_\infty \left(1 +\|\psit\|_{L^2_{\rho,b}} \right)  + O\left(b(1+|\vec \mu|_\infty) + b|\ln b||\vec \mu|_\infty \right) \notag \\
& \le C \frac {1+ \sqrt{b}|\pa_z\psit(\sqrt{b})|}{|\ln b|} \left(1 + \frac{\alpha}{|\ln b|}\right) \label{E:FBOUND},
\end{align}
for $0<b<b^*$ with $b^*$ sufficiently small and a universal constant $C>0.$
Applying Lemma~\ref{lemmanerinversion} and using~\eqref{E:FBOUND} we conclude from~\eqref{defhtildie}
\begin{align*}
 & \|\Delta \psit\|_{L^2_{\rho,b}}+\|\psit\|_{H^1_{\rho,b}}+|\ln b||\sqrt{b}\pa_z\psit(\sqrt{b})|\lesssim \|F(\psit)\|_{L^2_{\rho,b}} \\
 & \le C \frac{1+\sqrt{b}|\pa_z\psit(\sqrt{b})|}{|\ln b|}\left(1+\frac\alpha {|\ln b|}\right)
\end{align*}
and hence 
$$
\|\Delta \psit\|_{L^2_{\rho,b}}+\|\psit\|_{H^1_{\rho,b}}+ |\ln b||\sqrt{b}\pa_z\psit(\sqrt{b})|+\|F(\psit)\|_{L^2_{\rho,b}}\le \frac\alpha{|\ln b|}
$$ 
and $\psit\in B_\alpha$ for $\alpha>0$ universal large enough and $0<b<b^*(k)$ small enough.
Therefore $\tilde{\psi}\in B_\alpha.$
 
\noindent\underline{\it Contractive property}:
To show the contractive property, 
note that for any $\phi_1,\phi_2\in B_\alpha$ by~\eqref{E:PSI1EQUATION} we have that 
\begin{align*}
&f(\phi_1)-f(\phi_2) \\
& =\frac{|\ln b|}{2}\left(\mu_k(\phi_1)-\mu_k(\phi_2)\right) P_k\left[1+\frac{2\ln z}{|\ln b|}\right]+\left(\mu_k(\phi_1)-\mu_k(\phi_2)\right)\phi_1
   +\mu_k(\phi_2)\left(\phi_1-\phi_2\right)\\
& \ \ \ - \sum_{j=0}^{k-1}\frac{|\ln b|}{2}\left(\mu_{jk}(\phi_1)-\mu_{jk}(\phi_2)\right)\left\{\left[2(j-k)-\mu_k\right]P_j\left[1+\frac{2\ln z}{|\ln b|}\right]+\frac{2Q_j}{|\ln b|}\right\}.
\end{align*}
By a calculation analogous to~\eqref{E:MUJK} and~\eqref{E:MU}, we can evaluate the $\la\cdot,\cdot\ra_b$-inner product of $f(\phi_1)-f(\phi_2)$
with $P_j,$ $j=0,1,\dots,k$ and thereby estimate $|\mu_{jk}(\phi_1) - \mu_{jk}(\phi_2)|,$ $|\mu_{k}(\phi_1) - \mu_{k}(\phi_2)|.$ Using~\eqref{deballbalpha}, we arrive at 
\[
\|f(\phi_1)-f(\phi_2)\|_{L^2_{\rho,b}} \lesssim \frac{1}{|\ln b|} \|\phi_1-\phi_2\|_{L^2_{\rho,b}},
\]
which together with $F = (\text{Id} - \Bbb P_k) f $ gives 
$
\|F(\phi_1)-F(\phi_2)\|_{L^2_{\rho,b}} \lesssim \frac{1}{|\ln b|} \|\phi_1-\phi_2\|_{L^2_{\rho,b}}.
$
The operator $\tilde{H}_b^{-1}$ is continuous by Lemma~\ref{lemmanerinversion} and therefore for a sufficiently small $0<b<b_*$ the operator $\tilde{H}_b^{-1}\circ F$ is a strict contraction. 
By the Banach fixed point theorem, there exists a unique
$\psi_{b,k}\in H^1_{\rho,b}$ satisfying~\eqref{E:EIGENVALUEEQUATION}. The Fr\'echet differentiability of $\psi_b$ with respect to $b$ at any fixed $b>0$ follows similarly, the classical details are left to the reader.\\

\noindent{\bf step 2} Proof of~\eqref{estimatemomentum}. 
We estimate from \eqref{estmuj}, \eqref{deftk}: 
$$
\|\Lambda T_{b,k}\|_{L^2_{\rho,b}}\lesssim |\ln b|
$$  
and \eqref{calculeignvector} now implies 
\be\label{E:LAMBDAPSIBK}
\|\Lambda \psi_{b,k}\|_{L^2_{\rho,b}}\lesssim |\ln b|.
\ee
We may now apply \eqref{weightedesimate} to $z\psi_b$ and \eqref{estimatemomentum} follows from~\eqref{E:LAMBDAPSIBK}.\\

\noindent{\bf step 3} Spectral gap estimate. Let $u\in H^1_{\rho,b}$ satisfy
$$
\la u,\psi_{b,j}\ra_b=0, \ \ 0\leq j\leq k.
$$ 
Let 
$$
v=u{\bf 1}_{z\geq \sqrt{b}}-\sum_{j=0}^k\frac{\la u,P_j\ra_b}{\|P_j\|_{L^2_{\rho,0}}}P_j\in H^1_{\rho,0}
$$ 
then by the Poincar\'e inequality \eqref{estimationcoreor}:
\be
\label{cbebebebie}
\|\pa_zv\|_{L^2_{\rho,b}}^2\geq (2k+2)\|v\|_{L^2_{\rho,b}}^2.
\ee
We now compute from \eqref{deftk} for $0\leq j\leq k$:
\bee
0=\la u,\psi_j\ra_b=\left\la u,P_j(z)\mathcal T_b(z)+\sum_{i=0}^{j-1}\mu_i\left[P_i(z)-P_i(\sqrt{b})\right]+\psit_{b,k}\right\ra_b
\eee
and hence using \eqref{estmuj}, \eqref{calculeignvector}:
$$|\la u,P_j\ra_b|\lesssim \frac{1}{|\ln b|}\|u\|_{L^2_{\rho,b}}, \ \ 0\leq j\leq k$$ from which $$\|u-v\|_{H^1_{\rho,b}}\lesssim \frac{\|u\|_{L^2_{\rho,b}}}{|\ln b|}.$$ Injecting this into \eqref{cbebebebie} implies 
\be
\label{cneonneopernio}
\|\pa_zu\|_{L^2_{\rho,b}}^2\geq \left[2k+2+O\left(\frac{1}{|\ln b|}\right)\right]\|u\|_{L^2_{\rho,b}}^2.
\ee 
To prove \eqref{neneone}, \eqref{postiivitypsib}, let  
$$
\mu_b=\inf_{u\in\mathcal C_b}\frac{\|\pa_zu\|^2_{L^2_{\rho,b}}}{\|u\|_{L^2_{\rho,b}}^2}.
$$
Then any minimizing sequence normalised by $\|u_n\|_{L^2_{\rho,b}}=1$ is bounded in $H^1_{\rm loc}(r\geq \sqrt{b}).$ By the compactness of radial Sobolev and trace embeddings and \eqref{weightedesimate} it strongly converges in $L^2_{\rho,b}$. Hence the infimum is attained and from Lagrange multiplier theory: $$H_b\phi_b=\mu_b\phi_b.$$ Moreover, since $|\phi_b|$ is also an infimum, we may assume $\phi_b\geq 0$. If $\mu_b\neq \l_b$, then $\la\phi_b,\psi_{b,0}\ra_b=0$ and hence from \eqref{spcetralgap}: 
$$
\mu_b=\|\pa_z\phi_b\|_{L^2_{\rho,b}}^2\gtrsim 1. 
$$ 
Note that 
$
\l_{b,0} = \frac{\|\pa_z\psi_b\|_{L^2_{\rho,b}}^2}{\|\psi_b\|_{L^2_{\rho,b}}^2} \ge \mu_b
$
by the definition of $\mu_b.$ 
Together with \eqref{estimatelowe} this contradicts the definition of $\mu_b$ for $0<b<b^*$ small enough. 
Hence $\l_{b,0}=\mu_b$ is the bound state. The simplicity of the first eigenvalue follows from a classical argument~\cite{Evans}- 
 and hence $\psi_b\equiv \phi_b\geq 0.$ 
 Note that $\psi_b>0$ for $z>\sqrt{b}$ by the strong maximum principle.\\

\noindent {\bf step 4} Estimate for $\pa_b\l_{b,k}.$ Applying $\pa_b$ to~\eqref{E:EIGENVALUEEQUATION}, we obtain
\be\label{E:PABPSIBEQUATION}
H_b\pa_b\psi_{b,k} = \pa_b\l_{b,k} \psi_{b,k} + \l_{b,k}\pa_b\psi_{b,k}.
\ee
Evaluating the $\la\cdot,\cdot\ra_b$ inner product of~\eqref{E:PABPSIBEQUATION} with $\psi_{b,k}$ and integrating by parts we obtain
\begin{align*}
& \pa_b\l_{b,k} \|\psi_{b,k}\|_{L^2_{\rho,b}}^2 + \l_{b,k}\la\pa_b\psi_{b,k},\psi_{b,k}\ra_b  = \la H_b\pa_b\psi_{b,k},\psi_{b,k}\ra_b \\
& \ \ \  = \la \pa_b\psi_{b,k}, H_b \psi_{b,k}\ra_b - \pa_b\psi_{b,k}(\sqrt b) \pa_z\psi_{b,k}(\sqrt b) \rho(\sqrt b)\sqrt b \\
& \ \ \ = \l_{b,k}\la\pa_b\psi_{b,k},\psi_{b,k}\ra_b - \pa_b\psi_{b,k}(\sqrt b) \pa_z\psi_{b,k}(\sqrt b) \rho(\sqrt b)\sqrt b.
\end{align*}
Therefore
\be\label{E:PABLAMBDAB}
\pa_b\l_{b,k} = - \frac {\pa_b\psi_{b,k}(\sqrt b) \pa_z\psi_{b,k}(\sqrt b) \rho(\sqrt b)\sqrt b}{\|\psi_{b,k}\|_{L^2_{\rho,b}}^2}.
\ee
From $\psi_{b,k}(\sqrt b)=0$ it follows that 
$
\pa_b\psi_{b,k}(\sqrt b) = - \frac 1{2\sqrt b} \pa_z\psi_{b,k}(\sqrt b) 
$
and therefore from~\eqref{E:PABLAMBDAB}
\be\label{E:PABLAMBDABKBOUND}
\pa_b\l_{b,k} = \frac{ |\pa_z\psi_{b,k}(\sqrt b)|^2\rho(\sqrt b)}{2\|\psi_{b,k}\|_{L^2_{\rho,b}}^2}.
\ee
In particular, since $|\pa_z\psi_{b,k}(\sqrt b)| = O(\frac 1{\sqrt b})$ by~\eqref{deftk} and~\eqref{calculeignvector}, and $\|\psi_{b,k}\|_{L^2_{\rho,b}}^2 \gtrsim |\ln b|^2$  by~\eqref{estmuj} -~\eqref{calculeignvector} 
it follows that 
\be\label{E:PABLAMBDABESTIMATE}
|\pa_b\l_{b,k}| \lesssim \frac 1{b |\ln b|^2}.
\ee
which is the last claim of~\eqref{estimatelowe}. \\

\noindent {\bf step 5} Estimate for $|\pa_b\mu_{b,ik}|,$ $i=0,\dots,k-1.$
From~\eqref{deftk} we obtain
\begin{align}\label{E:PABTBK}
\pa_bT_{b,k} = -\frac 1{2b} P_k + \sum_{j=0}^{k-1}\pa_b\mu_{b,jk} P_j\mathcal{T}_b - \frac 1{2b}\sum_{j=0}^{k-1}\mu_{b,jk} P_j.
\end{align}
From~\eqref{deftk} and~\eqref{E:PABPSIBEQUATION} it follows that 
\be\label{E:PABPSIB1EQUATION}
H_b \pa_b \psit_{b,k} = F_k + \l_{b,k}\pa_b\psit_{b,k},
\ee
where 
\be\label{E:FKDEFINITION}
F_k =  - H_b \pa_bT_{b,k}+\pa_b(\l_{b,k} T_{b,k}) + \pa_b\l_{b,k}\tilde{\psi}_{b,k}.
\ee
Rewriting~\eqref{E:PABPSIB1EQUATION} in the form $H_b\pa_b\psi_{b,k} = (\pa_b\l_{b,k}\psit_{b,k}+\l_{b,k}\pa_b\psit_{b,k}) + \l_{b,k}\pa_bT_{b,k}$
and evaluating the $\la\cdot,\cdot\ra_b$-inner product with $P_j,$ $j=0,\dots,k-1,$ we obtain
\begin{align}\label{E:INT0}
\la H_b\pa_b\psi_{b,k}, P_j\ra_b = \l_{b,k}\la\pa_bT_{b,k},P_j\ra_b
\end{align}
since $\la\psit_{b,k},P_j\ra_b=\la\pa_b\psit_{b,k},P_j\ra_b=0.$
On the other hand, from~\eqref{E:PABTBK} we have that 
\begin{align}
\la\pa_bT_{b,k},P_j\ra_b & = -\frac 1{2b} M_{jk} 
+\sum_{i=0}^{k-1} \pa_b\mu_{b,ik}\left( M_{ji} \frac{|\ln b|}2 + \la P_i\ln z,P_j\ra_b\right) - \frac 1{2b}\sum_{i=0}^{k-1} \mu_{b,ik} M_{ji} \notag \\
&  = \frac{|\ln b|}2 \pa_b\mu_{b,ik} \delta_{ji} + \sum_{i=0}^{k-1}c_{ji}\pa_b\mu_{b,ik} + O(\frac 1{b|\ln b|}), \label{E:INT1}
\end{align}
where 
$
(c_{ij})_{i,j=0,\dots,k-1} = O(1)
$
and we used $M_{ji}=\delta_{ji}+O(b)$ for $j=0,\dots,k-1$ and the first two claims of~\eqref{estmuj} which have already been proven above.
On the other hand, observe that by integration-by-parts and the orthogonality $\la\pa_b\psit_{b,k},P_j\ra_b=0$ we have
\begin{align}\label{E:INT}
\la H_b\pa_bT_{b,k}, P_j\ra_b
=\rho(\sqrt b)\sqrt b\left(\pa_z\pa_bT_{b,k}(\sqrt b)P_j(\sqrt b) - \pa_b T_{b,k}(\sqrt b)\pa_zP_j(\sqrt b)\right).
\end{align}
By~\eqref{E:PABTBK} $|\pa_bT_{b,k}(\sqrt b)| \lesssim 1/b$ and
\[
\pa_z\pa_bT_{b,k} = -\frac 1{2b}\left(\pa_zP_k(\sqrt b) + \sum_{j=0}^{k-1}\mu_{b,jk}\pa_zP_j(\sqrt b)\right) 
+ \frac 1{\sqrt b}\sum_{j=0}^{k-1}\pa_b\mu_{b,jk} P_j(\sqrt b),
\]
which together with~\eqref{E:INT} leads to
\be\label{E:INT2}
\la H_b\pa_bT_{b,k}, P_j\ra_b
= \rho(\sqrt b)\sum_{j=0}^{k-1}\pa_b\mu_{b,jk} P_j(\sqrt b) + O(1),
\ee
where we note that $\pa_zP_k(\sqrt b) =\sqrt b L_k'(\frac b2) = O(\sqrt b)$ implying 
$\frac 1{\sqrt b} \max_{j=0,\dots,k}|\pa_zP_k(\sqrt b)| = O(1).$ To see that $L_k'(\frac b2)=O(1)$ observe that by the definition~\eqref{estaimti} of the $k$-th Laguerre polynomial,
it follows that $L_k'$ is a polynomial of degree $k-1.$
From~\eqref{E:INT0},~\eqref{E:INT1}, and~\eqref{E:INT2}  we conclude that 
\begin{align}\label{E:INVERT}
\frac{|\ln b|}2  \pa_b\mu_{b, ik} \delta_{ji} + c^*_{ji}\pa_b\mu_{b,i k} = O(\frac 1{b|\ln b|}), \ \ c^*_{ji} = O(1), \ i,j=0,\dots,k-1.
\end{align}
The system~\eqref{E:INVERT} is invertible for $0\le b\le b^*$ sufficiently small, and as a consequence 
\be\label{E:PABMUBIKBOUND}
\sup_{i=0,\dots,k-1}|\pa_b\mu_{b,ik}| \lesssim \frac 1{b|\ln b|^2},
\ee
which completes the proof of~\eqref{estmuj}.\\

\noindent {\bf step 5} Estimate for $\|\pa_b\psit_{b,k}\|_b.$
Recalling that $\la\psit_{b,k},P_j\ra_b = 0,$ $j=0,1,\dots k,$ by the construction of $\psit_{b,k},$ we conclude that $\la\pa_b\psit_{b,k},P_j\ra_b=0$ since $\psit_{b,k}(\sqrt b)=0.$
Moreover, the spectral gap estimate~\eqref{estimationcoreor} with $u=\pa_b\psit_{b,k}{\bf 1}_{z\ge\sqrt b}+\pa_b\psit_{b,k}(\sqrt{b}){\bf 1}_{0\le z<\sqrt b}$ together with the bound
$\l_{b,k}=2k+\tilde{\l}_{b,k}$ imply
\be\label{E:POINCARE2}
\l_{b,k}\|\pa_b\psi_{b,k}\|_{L^2_{\rho,b}}^2  \lesssim\frac{2k+O(\frac 1{|\ln b|})}{2k+2+O(\frac 1{|\ln b|})}  \|\pa_z\pa_b\psi_{b,k}\|_{L^2_{\rho,b}}^2 + C |\pa_b\psit_{b,k}(\sqrt b)|^2.
\ee
Evaluating the $\la\cdot,\cdot\ra_b$ inner product of~\eqref{E:PABPSIB1EQUATION} with $\pa_b\psit_{b,k},$ integrating-by-parts, using Cauchy-Schwarz, and the Poincar\'e-type  inequality~\eqref{E:POINCARE2} 
we obtain
\begin{align}\label{E:BOUNDB0}
\|\pa_z\pa_b\psit_{b,k}\|_{L^2_{\rho,b}}^2 \lesssim \|F_k\|_{L^2_{\rho,b}}^2 + \sqrt b |\pa_z\pa_b\psit_{b,k}(\sqrt b)\pa_b\psit_{b,k}(\sqrt b)| +  |\pa_b\psit_{b,k}(\sqrt b)|^2.
\end{align}
From $\psit_{b,k}(\sqrt b)=0$ it follows after differentiating with respect to $b$ that 
\be\label{E:BOUNDB1}
|\pa_b\psit_{b,k}(\sqrt b)| \lesssim \frac 1{\sqrt b}|\pa_z\psit_{b,k}(\sqrt b)| \lesssim  \frac 1{b|\ln b|^2}.
\ee
To estimate $\sqrt b |\pa_z\pa_b\psit_{b,k}(\sqrt b)|$ we 
note that 
\begin{align*}
0 & = \la \pa_b\psit_{b,k},1\ra_b  = \la \pa_b\psit_{b,k}, H_b \ln z\ra_b \\
& = \la H_b\pa_b\psit_{b,k},  \ln z \ra_b + e^{-b/2}\pa_b\psit_{b,k}(\sqrt{b}) + \frac 12 |\ln b| \sqrt{b}\pa_z\pa_b\psit_{b,k}(\sqrt{b})e^{-b/2}.
\end{align*}
Therefore
\begin{align*}
& \big|\sqrt{b}\pa_z\pa_b\psit_{b,k}(\sqrt{b})\big|  \lesssim \frac 1{|\ln b|}\left(\|F_k\|_{L^2_{\rho,b}} + |\l_{b,k}|\|\pa_b\psit_{b,k}\|_{L^2_{\rho,b}} + \frac 1{b|\ln b|^2} \right) \\
& \ \ \ \  
\lesssim \frac 1{|\ln b|}\|F_k\|_{L^2_{\rho,b}}+ \frac 1{|\ln b|}\|\pa_z\pa_b\psit_{b,k}\|_{L^2_{\rho,b}}+ \frac 1{b|\ln b|^3} ,
\end{align*}
where we used~\eqref{E:PABPSIB1EQUATION} in the first estimate and the Poincar\'e inequality~\eqref{E:POINCARE2} in the last. 
Using the Cauchy-Schwarz inequality and plugging this bound and~\eqref{E:BOUNDB1} into~\eqref{E:BOUNDB0}
we obtain
\begin{align}\label{E:PAZPABBOUND}
\|\pa_z\pa_b\psit_{b,k}\|_{L^2_{\rho,b}} \lesssim  \|F_k\|_{L^2_{\rho,b}}+\frac 1{b|\ln b|^2}.
\end{align}
From~\eqref{E:PABTBK} and~\eqref{E:PABMUBIKBOUND} we conclude that $\|T_{b,k}\|_{L^2_{\rho,b}} \lesssim \frac 1b.$
To estimate $\|F_k\|_{L^2_{\rho,b}}$ note that from~\eqref{copmjeouegov} we have the identity:
\begin{align*}
 H_b\pa_bT_{b,k}= 2k\pa_b T_{b,k} + \sum_{j=0}^{k-1}\pa_b\mu_{b,jk}\left[2(j-k)\mathcal{T}_bP_j+Q_j\right] +\frac 1b\sum_{j=0}^{k-1}\mu_{b,jk}(k-j)P_j
\end{align*}
and therefore using~\eqref{E:PABTBK} and~\eqref{estimatelowe} we have
\begin{align*}
- H_b\pa_bT_{b,k} + \l_{b,k}\pa_bT_{b,k} & = O(\frac 1{|\ln b|})\pa_bT_{b,k}  \\
& \ \ \ + \sum_{j=0}^{k-1}\pa_b\mu_{b,jk}\left[2(j-k)\mathcal{T}_bP_j+Q_j\right] +\frac 1b\sum_{j=0}^{k-1}\mu_{b,jk}(k-j)P_j.
\end{align*}
Hence, using~\eqref{E:PABMUBIKBOUND} and~\eqref{estmuj}:
$$
\|- H_b\pa_bT_{b,k} + \l_{b,k}\pa_bT_{b,k} \|_{L^2_{\rho,b}}\lesssim  \frac 1{b|\ln b|}.
$$
Thus, using the definition~\eqref{E:FKDEFINITION} of $F_k,$ bounds~\eqref{E:PABLAMBDABESTIMATE},~\eqref{calculeignvector}, and the previous bound
we obtain 
\begin{align*}
\|F_k\|_{L^2_{\rho,b}} & \lesssim \frac 1{b|\ln b|} +|\pa_b\l_{b,k}|\left( \|T_{b,k}\|_{L^2_{\rho,b}} + \|\psit_{b,k}\|_{L^2_{\rho,b}}\right) \\
& \lesssim \frac 1b + \frac 1{b|\ln b|^2}\left(|\ln b| + \frac 1{|\ln b|}\right) \lesssim  \frac 1{b|\ln b|}.
\end{align*}
Plugging this back into~\eqref{E:PAZPABBOUND} we get $$\|\pa_z\pa_b\psit_{b,k}\|_{L^2_{\rho,b}}\lesssim  \frac 1{b|\ln b|}$$ and therefore, by the 
spectral gap estimate~\eqref{estimationcoreor}, just like in~\eqref{E:POINCARE2}, we obtain
$$
\|\pa_b\psit_{k,b}\|_{L^2_{\rho,b}} \lesssim \frac 1{b|\ln b|}
$$
and the proof of~\eqref{calculeignvector} is completed.\\
 
\noindent{\bf step 6}  Proof of~\eqref{scalirpeoduc} -~\eqref{vnbevbnenonevbis}. \\
\noindent{\it Proof of~\eqref{scalirpeoduc}.} 
We estimate from \eqref{deftk}, \eqref{calculeignvector}: 
\be
\label{esfyveuvuev}
\la \psi_{b,k},\psi_{b,k}\ra_b=\frac{|\ln b|^2}{4}\left[\la P_k, P_k\ra_0+O\left(\frac1{|\ln b|}\right)\right]
\ee
and \eqref{scalirpeoduc} follows from the normalisation \eqref{normalization}.\\
\noindent{\it Proof of \eqref{idenitnirninebis}.} 
We compute from \eqref{E:PABTBK}, \eqref{estmuj}, \eqref{calculeignvector}: 
\be\label{enivbneneonve}
\left\|b\pa_b\psi_{b,k}+\frac12P_k\right\|_{H^2_{\rho,b}}\lesssim \frac{1}{|\ln b|}
\ee
and hence using \eqref{scalirpeoduc}:
\bee
\frac{\la b\pa_b \psi_{b,j},\psi_{b,k}\ra_b}{\la \psi_{b,k},\psi_{b,k}\ra_b}&=&\frac{4}{(\ln b)^2}\left[1+O\left(\frac{1}{\ln b}\right)\right]\left[\la -\frac{P_k}{2},\psi_{b,j}\ra_b+O(1)\right]\\
& = & \frac{4}{(\ln b)^2}\left[-\frac{1}4|\ln b|\la P_j,P_k\ra_b+O(1)\right]=-\frac{1}{|\ln b|}+O\left(\frac{1}{|\ln b|^2}\right),
\eee 
this is \eqref{idenitnirninebis}.\\
\noindent{\it Proof of \eqref{exactcomputation}}: We compute from \eqref{deftk}, \eqref{induction}:
\bee
\Lambda \psi_{b,k}&=&\Lambda P_k\mathcal{T}_b+P_k+\sum_{j=0}^{k-1}\mu_{b,jk}\Lambda ( P_j
\mathcal T_b)+\Lambda \psit_{b,k}\\
& = & 2k[P_k-P_{k-1}]\mathcal T_b+P_k+\sum_{j=0}^{k-1}\mu_{b,jk}\Lambda (P_j
\mathcal T_b)+\Lambda \psit_{b,k}\\
& = & 2k(\psi_{b,k}-\psi_{b,k-1})+\mathcal E_{b,k}
\eee
where the remainder estimate
$$
\|\mathcal E_{b,k}\|_{H^2_{\rho,b}}\lesssim 1
$$ 
holds due to~\eqref{estmuj},~\eqref{calculeignvector} and \eqref{exactcomputation} is proved. \\
\noindent{\it Proof of~ \eqref{vnbevbnenonev}.} 
Note that 
$$
P_k=\frac{2}{|\ln b|}\left\{\psi_{b,k}-\tilde{\psi}_{b,k}-P_k\ln z-\sum_{i=0}^{k-1}\mu_{b,jk}P_j\mathcal T_b\right\}
$$ 
and therefore 
\bee
&&\|b\pa_b\psi_{b,k} + \frac 1{|\ln b|}\psi_{b,k}\|_{H^2_{\rho,b}}\\
 & \lesssim &
\|b\pa_b\psi_{b,k} + \frac 12 P_k \|_{H^2_{\rho,b}} + 
\frac 1{|\ln b|}\|\tilde{\psi}_{b,k}+P_k\ln z+\sum_{i=0}^{k-1}\mu_{b,jk}P_j\mathcal T_b\|_{H^2_{\rho,b}} \\
& \lesssim &\frac 1{|\ln b|},
\eee
where we used~\eqref{enivbneneonve} and~\eqref{estmuj},~\eqref{calculeignvector}. This  concludes the proof of \eqref{vnbevbnenonev}.\\
\noindent{\it Proof of~ \eqref{vnbevbnenonevbis}.} We have from \eqref{deftk}, \eqref{calculeignvector}:
$$
\|z^2\psi_{b,k}-\frac{|\log b|}{2}z^2P_k\|_{H^1_b}\lesssim 1.
$$
Let $\Phi_k=z^2P_k$, then 
$$
(-\Delta + \Lambda)\Phi_k=(2k+2)\Phi_k-4P_k-4\Lambda P_k=(2k+2)\Phi_k-4P_k-8k(P_k-P_{k-1})
$$ 
and hence the relation 
\bee
&&2k\la \Phi_k,P_k\ra_0=(2k+2)\la \Phi_k,P_k\ra_0-4-8k\ \ \mbox{ie}\ \ \la \Phi_k,P_k\ra_0=2+4k\\
&& (2k-2)\la \Phi_j,P_{k-1}\ra_0=(2k+2)\la \Phi_K,P_{k-1}\ra_0+8k\ \ \mbox{ie}\ \ \la \Phi_k,P_k\ra_0=-2k\\
&& \la \Phi_k,P_{j}\ra_0=0, \ \ 0\le j \le k-2.
\eee
Since $P_k$ is a polynomial we conclude that there exists a $c_k\in\RR$ such that 
$$
z^2P_k=c_kP_{k+1}+(4k+2)P_k-2kP_{k-1}.
$$ 
Since $P_k(0)=1$, we obtain by plugging in $z=0$ into the above relationship: 
$$
z^2P_k=-2(k+1)P_{k+1}+(4k+2)P_k-2kP_{k-1},
$$ which yields \eqref{vnbevbnenonevbis}.
\end{proof}


\subsection{Diagonalisation of $\H_b$}
\label{sectionfour}

We are now position to derive the bound state and the spectral gap estimate for the operator $\H_b$.

\begin{lemma}[Renormalised eigenfunction]
\label{lemmarenormalized}
Let $K\in \Bbb N$. Then for all $0<b<b^*(K)$ small enough, the renormalised operator 
$$
\H_b=-\Delta+ b\Lambda \ \ \mbox{with boundary value}\ \ u(1)=0
$$ 
has a family of eigenstates $\eta_{b,k}$ satisfying:\\
\be
\label{eignevalure}
\mathcal H_b\eta_{b,k}=b\l_{b,k}\eta_{b,k}, \  \ 0\leq k\leq K,
\ee
with the following properties:\\
\noindent{\em (i) Structure of the eigenmodes}: there holds the expansion
\be
\label{deftketa}
\left\{\begin{array}{ll}\eta_{b,k}=S_{b,k}(y)+\etat_{b,k}(y)\\
S_{b,k}(y)=P_k(\sqrt{b}y)\ln y+\sum_{j=0}^{k-1}\mu_{b,jk}P_j(\sqrt{b} y)\ln y
\end{array}\right.
\ee
with
\be
\label{calculeignvectoreta}
\frac{\|\Delta \etat_{b,k}\|_{L^2_{b}}}{\sqrt{b}}+\|\pa_y\etat_{b,k}\|_b+\sqrt{b}\|\tilde{\eta}_{b,k}\|_{b}+\sqrt{b}\|\Lambda \etat_{b,k}\|_{b}+|\ln b||\pa_y\tilde{\eta}_{b,k}(1)|\lesssim \frac{1}{|\ln b|},
\ee
\noindent{\em (ii) Further estimates on the eigenvector}: there holds
\bea
\label{cebivbfenoe}
&& \|y\eta_{b,k}\|_{b}\lesssim \frac{|\ln b|}{b}, \ \ \|y^2\eta_{b,k}\|_{b}\lesssim \frac{|\ln b|}{b\sqrt{b}}\\
&& \label{estderivative} \|b\pa_b\eta_{b,k}\|_b\lesssim \frac{|\ln b|}{\sqrt{b}}.
\eea
Moreover:
\bea
\label{exactcomputationeta}
\nonumber&& \|\Lambda \eta_{b,k}-2k(\eta_{b,k}-\eta_{b,k-1})\|_{b}+\frac{1}{\sqrt{b}}\|\pa_y[\Lambda \eta_{b,k}-2k(\eta_{b,k}-\eta_{b,k-1})]\|_{b}\\
&& + \frac{1}{b}\|\mathcal H_b\left[\Lambda \eta_{b,k}-2k(\eta_{b,k}-\eta_{b,k-1})\right]\|_b\lesssim  \frac{1}{\sqrt{b}}
\eea
and 
\bea
\label{foihonononfe}
\nonumber &&\|2b\pa_b\eta_{b,j}-\Lambda \eta_{b,j}+\frac{2}{|\ln b|}\eta_{b,j}\|_b+\frac{1}{\sqrt{b}}\|\pa_y[2b\pa_b\eta_{b,j}-\Lambda \eta_{b,j}+\frac{2}{|\ln b|}\eta_{b,j}\|_b\\
&& + \frac{1}{b}\|\mathcal H_b\left[2b\pa_b\eta_{b,j}-\Lambda \eta_{b,j}+\frac{2}{|\ln b|}\eta_{b,j}\right]\|_b\lesssim \frac{1}{\sqrt{b}|\ln b|}.
\eea

\noindent{\em (iii) Normalisation:}
\be
\label{scalirpeoduceta}
( \eta_{b,k},\eta_{b,k})_b=\frac{|\ln b|^2}{4b}\left[1+O\left(\frac1{|\ln b|}\right)\right].
\ee
\end{lemma}

\begin{remark} Observe that \eqref{exactcomputationeta}, \eqref{foihonononfe}, \eqref{scalirpeoduceta} imply the bound:
\bea
\label{exactcomputationetabis}
\nonumber&& \|b\pa_b\eta_{b,k} -k(\eta_{b,k}-\eta_{b,k-1})\|_{b}+\frac{1}{\sqrt{b}}\|\pa_y[b\pa_b\eta_{b,k} -k(\eta_{b,k}-\eta_{b,k-1})]\|_{b}\\
&& + \frac{1}{b}\|\mathcal H_b\left[b\pa_b\eta_{b,k}-k(\eta_{b,k}-\eta_{b,k-1})\right]\|_b\lesssim  \frac{1}{\sqrt{b}}.
\eea
\end{remark}

\begin{proof}[Proof of Lemma \ref{lemmarenormalized}]
Given $u:\Omega\to\mathbb{R}$, let $v(y)=u(\sqrt{b}y)$, it is straightforward to check that
\be
\label{formuelerenorer}
\H_b v=b(H_bu)(\sqrt{b}y), \ \  \|\pa_y^\ell v\|_{b}=b^{\frac{\ell-1}{2}}\|\pa_z^\ell u\|_{L^2_{\rho,b}}, \ \ \ell \in \mathbb{N}.
\ee
We therefore let
\be
\label{defphibnk}
\eta_{b,k}(y)=\psi_{b,k}(z), \ \ z=y\sqrt{b}.
\ee
and ~\eqref{eignevalure} follows. The estimate \eqref{exactcomputationeta} follows by rescaling \eqref{exactcomputation}, and \eqref{scalirpeoduceta} by rescaling \eqref{scalirpeoduc}. The decomposition \eqref{deftketa} follows from \eqref{deftk}, and \eqref{calculeignvector} implies \eqref{calculeignvectoreta}. The bounds \eqref{cebivbfenoe} follow by rescaling the bounds: $$\ \ \|z\psi_{b,k}\|_{L^2_{\rho,b}}+\|z^2\psi_{b,k}\|_{L^2_{\rho,b}}\lesssim |\ln b|.$$
Directly from the definition \eqref{defphibnk}, we compute:
\be
\label{deftrrscaac}
b\pa_b\eta_{b,k}=\left[\frac{\Lambda \psi_{b,k}}{2}+b\pa_b\psi_{b,k}\right](\sqrt{b}y),
\ee 
which together with \eqref{estimatedervaitveinb}, \eqref{estimatemomentum} yields: 
$$
\|b\pa_b\eta_{b,k}\|_b\lesssim \frac{|\ln b|}{\sqrt{b}},
$$ 
thus proving~\eqref{estderivative}. From \eqref{deftrrscaac}:
$$
2b\pa_b\eta_{b,k}-\Lambda \eta_{b,k}=2\left[b\pa_b\psi_{b,k}\right](\sqrt{b}y) 
$$ 
and hence from \eqref{vnbevbnenonev}:
$$
\left\|2b\pa_b\eta_{b,k}-\Lambda \eta_{b,k}+\frac{2}{|\ln b|}\eta_{b,k}\right\|_b\lesssim \frac{1}{\sqrt{b}}\|b\pa_b\psi_{b,k}+\frac{1}{|\ln b|}\psi_{b,k}\|_{L^2_{\rho_b}}\lesssim \frac{1}{|\ln b|\sqrt{b}}
$$ 
and similarly for higher derivatives.
\end{proof}


\subsection{Diagonalisation of $\mathcal H_B$}
\label{sectionfive}

We now change sign and consider the operator for $B>0$
\be
\label{defHB}
\H_B=-\Delta- B\Lambda \ \ \mbox{with boundary value}\ \ u(1)=0,
\ee
which is a self adjoint operator on $H^1_{B,+}$ given by \eqref{E:HBDEF}. 

\begin{lemma}[Renormalised eigenfunction]
\label{lemmarenormalizedbis}
Let $K\in \Bbb N$. Then for all $0<B<B^*(K)$ small enough, the renormalised operator 
$$
\H_B=-\Delta- B\Lambda \ \ \mbox{with boundary value}\ \ u(1)=0
$$ 
has a family of eigenstates 
\be
\label{beiebibeiebis}
\etah_{B,k}=e^{-\frac{B|y|^2}{2}}\eta_{B,k}, \  \mathcal H_B\etah_{B,k}=B\lh_{B,K}\etah_{B,k}, \ \ \lh_{B,k}=\l_{B,k}+2, \ \ 0\le k \le K,
\ee 
with $\eta_{B,k},\l_{B,k}$ given by Lemma \ref{lemmarenormalized}. Furthermore, there hold the following properties:\\
\noindent{\em (i) Structure of the eigenmodes}: there holds the expansion
\be
\label{deftketabis}
\left\{\begin{array}{ll}\etah_{B,k}=S_{B,k}(y)e^{-\frac{B|y|^2}{2}}+\etaht\\
S_{B,k}(y)=P_k(\sqrt{B}y)\ln y+\sum_{j=0}^{k-1}\mu_{B,jk}P_j(\sqrt{B} y)\ln y
\end{array}\right.
\ee
with
\bea
\label{calculeignvectoretabis}
\nonumber &&\frac{\|\Delta \etaht_{B,k}\|_{L^2_{B}}}{\sqrt{B}}+\|\pa_y\etaht_{B,k}\|_B+\sqrt{B}\|\tilde{\etah}_{B,k}\|_{B}+\sqrt{B}\|\Lambda \etaht_{B,k}\|_{B}+|\log B||\pa_y\tilde{\etah}_{B,k}(1)|\\
&&\lesssim  \frac{1}{|\ln B|},
\eea
\noindent{\em (ii) Further estimates on the eigenvector}: there holds
\bea
\label{cebivbfenoebis}
&& \|y\etah_{B,k}\|_{b}\lesssim \frac{|\ln B|}{B}, \ \ \|y^2\etah_{B,k}\|_{B}\lesssim \frac{|\ln B|}{B\sqrt{B}}\\
&& \label{estderivativebis} \|B\pa_B\etah_{B,k}\|_b\lesssim \frac{|\ln B|}{\sqrt{B}}.
\eea
Moreover:
\bea
\label{exactcomputationetabisbis}
\nonumber&& \|B\pa_B \etah_{B,k}-(k+1)[\etah_{B,k+1}-\etah_{B,k}]\|_{B}\\
\nonumber 
&+& \frac{1}{\sqrt{B}}\|\pa_y[B\pa_B\etah_{B,k}-(k+1)[\etah_{B,k+1}-\etah_{B,k}]\|_{B}\\
& +& \frac{1}{B}\|\mathcal H_B\left[B\pa_B\etah_{B,k}-(k+1)[\etah_{B,k+1}-\etah_{B,k}]\right]\|_B\lesssim  \frac{1}{\sqrt{B}}
\eea
and 
\bea
\label{foihonononfebisbis}
\nonumber &&\|2B\pa_B\etah_{B,j}-\Lambda \etah_{B,j}+\frac{2}{|\ln B|}\etah_{B,j}\|_b+\frac{1}{\sqrt{B}}\|\pa_y[2B\pa_B\etah_{B,j}-\Lambda \etah_{B,j}+\frac{2}{|\ln B|}\etah_{B,j}\|_B\\
&& + \frac{1}{B}\|\mathcal H_B\left[2B\pa_B\etah_{B,j}-\Lambda \eta_{B,j}+\frac{2}{|\ln B|}\etah_{B,j}\right]\|_B\lesssim \frac{1}{\sqrt{B}|\ln B|}.
\eea

\noindent{\em (iii) Normalisation:}
\be
\label{scalirpeoducetabis}
( \etah_{B,k},\etah_{B,k})_B=\frac{|\log B|^2}{4B}\left[1+O\left(\frac1{|\log B|}\right)\right].
\ee
\end{lemma}

\begin{proof}[Proof of Lemma \ref{isometryhb}] This is a direct consequence of Lemma \ref{lemmarenormalized}. Indeed, the map 
\be
\label{isometryhb}
\begin{array}{ll} L^2_{B,+}\to L^2_{B,-}\\ v\mapsto w=e^{\frac{B|y|^2}{2}}v\end{array}\ \ \mbox{is an isometry}
\ee  and integrating by parts: 
\be\label{E:ISOMETRY}
\int |\nabla v|^2e^{\frac{B|y|^2}{2}}\rho_{B,+}dy=\int |\nabla w|^2\rho_{B,-}dy+2B\int |w|^2\rho_{B,-}dy
\ee
or equivalently:
$$
\H_Bv=\left(-\Delta w+B\Lambda w+2B w\right) e^{-\frac{B|y|^2}{2}}.
$$
Together with Lemma~\ref{lemmarenormalized}, this yields \eqref{beiebibeiebis}. We now renormalise \eqref{weightedesimate} which yields:
\be
\label{estwiheihegohe}
B^{\frac{k+1}{2}}\|y^kw\|_{B,-}\lesssim \|\pa_yw\|_{B,-}+\sqrt{B}\|w\|_{B,-}.
\ee 
The isometric relation~\eqref{E:ISOMETRY} and~\eqref{estwiheihegohe} imply the following comparison of norms: 
\bea
\label{normzerp}
&&\|\pa_yv\|_{B,+}\lesssim \|\pa_yw\|_{B,-}+\sqrt{B}\|w\|_{B,+}\\
\label{normone}
&&\|\Lambda v\|_{B,+}=\|\Lambda w-By^2w\|_{B,-}\lesssim \|\Lambda w\|_{B,-}+\frac{1}{\sqrt{B}}\|\pa_yw\|_{B,-}+\|w\|_{B,-}\\
&&\|y^kv\|_{B,+}\lesssim \|y^kw\|_{B,-}.
\eea
Using~\eqref{estwiheihegohe}:
\bea
\label{normtwo}
\nonumber \|\Delta v\|_{B,+}&\lesssim &\|\Delta w\|_{B,-}+B\|w\|_{B,-}+B\|\Lambda w\|_{B,-}+B^2\|y^2w\|_{B,-}\\
& \lesssim & \|\Delta w\|_{B,-}+B\|w\|_{B,-}+B\|\Lambda w\|_{B,-}+\sqrt{B}\|\pa_y w\|_{B,-}.
 \eea
 We also observe that 
 $$
 |\pa_yw(1)|\lesssim |\pa_yv(1)|\ \ \mbox{if}\ \ w(1)=0.
 $$ 
Using the above bounds together with~\eqref{beiebibeiebis},~\eqref{deftketabis},~\eqref{calculeignvectoreta},~\eqref{cebivbfenoe}, and~\eqref{scalirpeoduceta} 
 yields \eqref{calculeignvectoretabis}, \eqref{cebivbfenoebis}, and \eqref{scalirpeoducetabis}.
 Moreover, 
 $$
 \pa_B\etah_{B,k}=\left(\pa_B\eta_{B,k}-\frac{|y|^2}{2}\eta_{B,k}\right)e^{-\frac{By^2}{2}}.
 $$ 
 Hence \eqref{estderivativebis} follows from \eqref{estderivative}, \eqref{cebivbfenoe}. 
 We now observe the fundamental conjugation 
 $$
 \Lambda \etah_{B,j}-2B\pa_B\etah_{B,j}=\left(\Lambda \eta_{B,j}-2B\pa_B\eta_{B,j}\right)e^{-\frac{B|y|^2}{2}},
 $$ and hence \eqref{foihonononfebisbis} follows from \eqref{foihonononfe}. Moreover from \eqref{vnbevbnenonevbis}:
 $$B|y|^2\eta_{B,k}=z^2\psi_{B,k}=-(2k+2)\eta_{B,k+1}+(4k+2)\eta_{B,k}-2k\eta_{B,k-1}+F_k, \ \ \|F_k\|_B\lesssim\frac{1}{\sqrt{B}}$$ and hence using \eqref{exactcomputationetabis}:
 \bee
 B\pa_B\eta_{B,k}-B\frac{|y|^2}{2}\eta_{B,k}&=&k(\eta_{B,k}-\eta_{B,k-1})-\frac 12\left[-(2k+2)\eta_{B,k+1}+(4k+2)\eta_{B,k}-2k\eta_{B,k-1}\right]+\tilde{F}_k\\
 & = & (k+1)\left[\eta_{B,k+1}-\eta_{B,k}\right]+\tilde{F}_k\ \ \mbox{with}\ \ \|\tilde{F}_k\|_B\lesssim \frac{1}{\sqrt{B}}\\
 \eee
 and similarly for higher derivatives, and \eqref{exactcomputationetabisbis} is proved.
  \end{proof}


\section{Finite time melting regimes}
\label{sectionnonlin}

This section is devoted to the existence and stability of the melting process. In all the section, we let $$\pm=-, \ \ \rho=\rho_-, \ \ b>0.$$


\subsection{Renormalised equations and initialisation}
\label{sectinogoingeo}

We start with the classical modulated nonlinear decomposition of the flow. We let 
\be
\label{normalizationinitial}
u(t,x)=v(s,y), \ \ y=\frac{r}{\l(t)}, \ \ \l(0)=1,
\ee 
where $\l(0)=1$ is assumed without loss of generality thanks to the scaling symmetry \eqref{E:SELFSIMILAR0}. We define the renormalised time 
\be
\label{renormalizedtime}
s(t)=s_0+\int_0^{t}\frac{d\tau}{\l^2(\tau)}, \ \ s_0\gg 1,
\ee
and obtain the renormalised flow: 
\be\label{E:KDEFINITION}
\left\{\begin{array}{ll}\pa_sv+\H_a v=0, \ \ a=-\lsl,\\ v(s,1)=0, \ \ \pa_yv(s,1)=a.
\end{array}\right.
\ee
We now prepare our initial data in the following way:\\

\noindent\underline{case $k=0$}: We first claim that given $0<b^*\ll 1$ and $\e^*$ with $$\|\e^*\|^2_{b^*}\lesssim \frac{b^*}{|\log b ^*|^2},$$ there exists a locally unique decomposition 
\be
\label{decopmsbstar}
b^*\eta_{b^*,0}+\e^*=b\eta_{b,0}+\e\ \ \mbox{with}\ \ (\e,\eta_{b,0})_b=0, \ \ |b-b^*|\lesssim \frac{b^*}{|\log b^*|^2}.
\ee 
Indeed, we define the map $$F(b,\e^*)=( b^*\eta_{b^*,0}-b\eta_{b,0}+\e^*,\eta_{b,0})_b$$ which satisfies $F(b^*,0)=0$ and 
$$
\pa_bF(b^*,0)=
- ( \eta_{b^*,0}+b^*(\pa_b\eta_{b,0})|_{b=b^*},\eta_{b^*,0})_{b^*}=- \frac{|\log b^*|^2}{4b^*}\left[1+O\left(\frac1{|\log b^*|}\right)\right]<0
$$ 
from \eqref{scalirpeoduceta} and the degeneracy \eqref{exactcomputationetabis} for $k=0$. The claim then follows from a standard application of the implicit function theorem. A Taylor expansion of $F$ about $(b,\e^*)\equiv(b^*,0)$ yields the bound
$$
|b-b^*|\frac{|\log b^*|^2}{4b^*}\left[1+O\left(\frac1{|\log b^*|}\right)\right]\lesssim \|\e^*\|_{b^*}\|\eta_{b^*,0}\|_{b^*}\lesssim \|\e^*\|_{b^*}\frac{|\log b^*|}{2\sqrt{b^*}}
$$ 
and hence $$|b-b^*|\lesssim \frac{\sqrt{b^*}}{|\log b|^*}\frac{\sqrt{b^*}}{|\log b|^*}$$ which concludes the proof of \eqref{decopmsbstar}. 

We therefore pick an initial datum $$v_0^*=b_0^*\eta_{b_0^*}+\e_0^*, \ \ \|\e_0^*\|^2_{b^*}\lesssim \frac{b^*}{|\log b ^*|^2}$$ and decompose the solution 
\be
\label{vejbebejbie}
v(s,y)=b_0(s)\eta_{b_0(s),0}+\e(s,y) \ \ \mbox{with} \ \ b(s)=b_0(s)\ \ \mbox{and}\ \  (\e,\eta_{b,0})_b=0
\ee
which makes sense as long as $\e(s,y)$ is small enough in the $L^2$ weighted sense. We let 
\be
\label{defetwo}
\e_2=\H_b\e
\ee 
and define the energy
$$
\calE : = \|\H_b \e\|_b^2,
$$
which is a coercive norm thanks to the orthogonality condition \eqref{vejbebejbie}, see Appendix \ref{appendcoerc}. We assume the initial smallness
\be\label{E:ESMALLzero}
 \calE(0)\leq\ \frac{b^3(0)}{|\ln b(0)|^2}.
 \ee

\noindent\underline{case $k\geq 1$}: Let
\be
\label{linearcj}
c_{k,1}=-\frac{k+1}{2k^2}, \ \ c_{k,2}=c_{k,1}-\frac{(k+1)\alpha_k}{k}.
\ee 
Then we freeze the explicit value 
\be
\label{definitionb}
b(s) : =\frac{1}{2ks}+\frac{c_{k,1}}{s\ln s}
\ee
and the sequence 
\be
\label{defbkj}
\left\{\begin{array}{ll} b_k^e=\frac{k+1}{2k s}+\frac{c_{k,2}}{s\ln s},\\
b_{j}^e=0, \ \ 0\leq j\leq k-1,
\end{array}\right.
\ee
where the index ``e" stands for exact. 

\begin{remark}
\label{remarkexacte}
 Letting 
 \be
 \label{valueaexact}
 a^e=\frac{k+1}{2ks}+\frac{c_{k,1}}{s\ln s},
 \ee then $(a^e,b,b_k^e)$ satisfy 
\be
\label{approxode}
\left\{\begin{array}{lll}(b_k^e)_s+bb_k^e\left[2k+\frac{2}{|\log b|}\right]+\frac{2(a^e-b)b_k^e}{|\log b|}=O\left(\frac{1}{s^2(\ln s)^2}\right)\\
 b_s+2b(a^e-b)=O\left(\frac{1}{s^2(\ln s)^2}\right)
 \\ a^e-b^e_k\left(1+\frac{2\alpha_k}{|\log b|}\right)=O\left(\frac{1}{s(\log s)^2}\right).
\end{array}\right.
\ee
\end{remark}

We define
\be
\label{defqbeta}
Q_\beta (y)=\sum_{j=0}^k b_j\eta_{b,j}(y)
\ee
and introduce the dynamical decomposition 
\be
\label{orthoe}
v(s,y)=Q_{\beta(s)}+\e(s,y), \ \ (\eta_{b,j}(s),\e)_{b}=0, \ \ 0\leq j\leq k.
\ee 
We let again $$
\e_2=\H_b\e, \ \ \calE : = \|\H_b \e\|_b^2,
$$
which due to the orthogonality conditions \eqref{orthoe} is a coercive norm, see Appendix \ref{appendcoerc}. We assume the initial smallness
\be\label{E:ESMALL}
 \calE(0)\leq
  \frac{b^3(0)}{|\ln b(0)|}.
  \ee
For $k\geq 1$, the set of initial data will be built as a codimension $k$ manifold. To this end and in order to prepare the data, we 
 consider the decomposition 
\be
\label{bveuibvibi}
b_j(s)=b_j^e(s)+\bt_j(s), \ \ \bt_j(s)=\frac{V_j(s)}{s(\ln s)^{\frac 32}}, \ \  j=0,\dots, k.
\ee 
Let the $(2\times2)$-matrix $A_k$ be given by 
\be\label{E:AKMATRIX}
A_k: =  \left( \begin{array}{cc}
-1 & -1  \\
1 &  1 +d_k   
\end{array} \right), \ \ d_k: = \frac 1{k(k+1)}.
\ee
The matrix $A_k$ is diagonalisable with one strictly positive $\mu^k_1>0$ and one strictly negative eigenvalue $\mu^k_2<0.$ Let $P_k$ be an orthogonal matrix diagonalizing $A_k,$ i.e.
\[
A_k = P_k^{-1} \Lambda_k P_k, \ \ 
\Lambda_k : =  \left( \begin{array}{cc} 
\mu^k_1 & 0  \\
0 &  \mu^k_2.   
\end{array} \right)
\]
We define the new unknowns $W_k,W_{k-1}$ by setting 
\be\label{E:WKDEFINITION}
\left( \begin{array}{c}
 W_k  \\
W_{k-1}    
\end{array} \right) : = 
P_k \left( \begin{array}{c}
 V_k  \\
V_{k-1}    
\end{array} \right)
\ee
and now assume the initial bound 
\bea
\label{estinitialbound}
&&|W_k(0)|\leq  1 \\
\label{contorljv}
&&|W_{k-1}(0)|^2+\sum_{j=0}^{k-2}\left|\frac{V_j(0)}{\delta}\right|^2\leq  K^2 
\eea
for some universal constants $K>0$, $0<\delta(k)\ll1$  to be chosen later. We then consider the bootstrap bounds 
\be
\label{E:BOOTBOUND}
\mathcal E \le \left\{\begin{array}{ll}  \frac{Db^3}{|\ln b|^2}\ \ \mbox{for}\ \ k=0,\\ \frac{Db^3}{|\ln b|}\ \ \mbox{for} \ \ k\ge 1
\end{array}\right.
\ee
for some large enough universal $D=D(k)$ to be chosen later, and 
\begin{itemize}
\item for $k=0$
\be
\label{positivitybzero}
0<b_0(s)<b^*
\ee
\item for $k\ge 1$, 
\be
\label{boundstablemode}
|W_k(s)|\leq K
\ee
and 
\be
\label{aprioriboundbjbis} 
|W_{k-1}(s)|^2+\sum_{j=0}^{k-2}\left|\frac{V_j(s)}{\delta}\right|^2\leq K^2
\ee
\end{itemize}
and define 
$$
s^*=\left\{\begin{array}{ll}\sup_{s\geq s_0}\{\eqref{E:BOOTBOUND}, \eqref{positivitybzero}\ \ \mbox{hold on}\ \ [s_0,s]\}\ \ \mbox{for}\ \ k=0,\\\sup_{s\geq s_0}\{\eqref{E:BOOTBOUND}, \eqref{boundstablemode}, \eqref{aprioriboundbj}  \ \ \mbox{hold on}\ \ [s_0,s]\}\ \ \mbox{for}\ \ k\ge 1.\end{array}\right.
$$
The main ingredient of the proof of theorem~\ref{T:MAIN} is the following:

\begin{proposition}[Bootstrap estimates on $b$ and $\varepsilon$]
\label{pr:bootstrap}
The following statements hold:\\
{\em 1. Stable regime}: for $k=0$, $s^*=+\infty$.\\
{\em 2. Unstable regime}: for $k\geq 1$, there exist constants $K,$ $\delta = \delta(K)\ll1$ and $(V_0(0),\dots,V_{k-2}(0),W_{k-1}(0))$ depending on $\e(0)$ 
satisfying~\eqref{E:ESMALL},~\eqref{estinitialbound} and~\eqref{contorljv}, such that  $s^*=+\infty$.
\end{proposition}

\begin{remark} Let us observe that our set of initial data is non empty and contains compactly supported arbitrarily small data in $\dot{H}^1$, see Appendix \ref{data}.
\end{remark}

\begin{remark}
The proof of the Proposition \ref{pr:bootstrap} is presented in section~\ref{S:PROOFPROPOSITION}.
\end{remark}

From now on and for the rest of this section, we study the flow in the bootstrap regime $s\in[s_0,s^*)$. Note in particular the rough bounds 
\be
\label{roughboundbj}
|b_k|\lesssim b, \ \ |b_j|\lesssim \frac{b}{|\ln b|}, \ \ 0\leq j\leq k-1\ \ \mbox{and}\ \ \mathcal E\leq \frac{b^3}{\sqrt{|\ln b|}}
\ee
for $s_0\geq s_0(K)$ large enough.


\subsection{Extraction of the leading order ODE's driving the melting}
\label{beivbevbeo}

We derive in this section the main dynamical constraint on the parameters $(a,b,(b_j)_{0\leq j\leq k})$ which lead to the leading order modulation equations, and are a combination of the linear diagonalisation of the $\H_b$ operator and the nonlinear boundary conditions.\\

We start with the constraint induced by the boundary conditions.

\begin{lemma}[Boundary conditions]
\label{lemmaindeitnitues}
There holds:
\bea
\label{E:PARTIALYEPSILONBOUNDARY}
&&a=\sum_{j=0}^{k}b_j\left[1+\frac{2\alpha_j}{|\ln b|}\right]+O\left(\frac{b}{|\ln b|^2}+\frac{\sqrt{\mathcal E}}{|\ln b|\sqrt{b}}\right),\\
&&\e_2(1)=-a(a-b), \label{etwoone}\\
\label{etwoonebis}
& & \pa_y\e_2(1)=-a_s-\sum_{j=0}^k\l_{b,j}bb_j\left[1+\frac{2\alpha_j}{|\ln b|}\right]+O\left(\frac{b^2}{|\ln b|^2}\right).
\eea
\end{lemma}

\begin{proof}[Proof of Lemma \ref{lemmaindeitnitues}]  We compute from \eqref{deftketa}, \eqref{calculeignvectoreta}, \eqref{estmuj} and recalling the definition \eqref{edfalphaj}:
\be
\label{bejvieivbebie}
\pa_{y}\eta_{b,j}(1)=1+\sum_{i=0}^{j-1}\frac{2}{(j-i)|\ln b|}+O\left(\frac{1}{|\ln b|^2}\right)=1+\frac{2\alpha_j}{|\ln b|}+O\left(\frac{1}{|\ln b|^2}\right).
\ee
This implies that
$$
\pa_yQ_\beta(1)=\sum_{j=0}^{k}b_j\pa_y\eta_{b,j}(1)=\sum_{j=0}^{k}b_j\left[1+\frac{2\alpha_j}{|\ln b|}\right]+O\left(\frac{b}{|\ln b|^2}\right).
$$
Since $v = Q_\beta + \e,$ it follows that 
\bee
\e_y\big|_{y=1} & = &v_y\big|_{y=1} - \pa_yQ_\beta\big|_{y=1}=  -\lsl- \pa_yQ_\beta(1)\\
& = & a- \sum_{j=0}^{k}b_j\left[1+\frac{2\alpha_j}{|\ln b|}\right]+O\left(\frac{b}{|\ln b|^2}\right),
\eee
which together with \eqref{coercunubis} yields \eqref{E:PARTIALYEPSILONBOUNDARY}. From \eqref{E:KDEFINITION}, $v(s,1)=0$ and $\pa_yv(s,1)=-\lsl=a$: $$0=\H_av(1)=(\H_bv+(a-b)\Lambda v)(1)=\e_2(1)+a(a-b),$$ this is \eqref{etwoone}. Now from $\pa_yv(s,1)=a$, we have $$\pa_s\pa_yv(s,1)=a_s.$$ On the other hand, taking $\pa_y$ of \eqref{E:KDEFINITION}, we have:
$$
0=\pa_s\pa_yv+\pa_y(\H_bv+(a-b)\Lambda v)=\pa_s\pa_yv+\pa_y\e_2+\pa_y\H_bQ_\beta+(a-b)y\Delta v.
$$
We evaluate the above identity at $y=1$. From \eqref{E:KDEFINITION} and $\pa_sv(s,1)=0$, $\pa_yv(s,1)=a$: $$\Delta v(1)=a\Lambda v(1)=a^2.$$ By construction, $$\pa_y \H_bQ_\beta=\sum_{j=0}^k\l_{b,j}bb_j\pa_y\eta_{b,j},$$ and hence:
$$a_s+\pa_y\e_2(1)+\sum_{j=0}^k\l_{b,j}bb_j\left[1+\frac{2\alpha_j}{|\ln b|}\right]+a^2(a-b)=O\left(\frac{b^2}{|\ln b|^2}\right).$$ We inject into the estimate the rough bound 
\be
\label{controla}
|a|\lesssim b
\ee
which follows from \eqref{roughboundbj}, \eqref{E:PARTIALYEPSILONBOUNDARY} and  \eqref{etwoonebis} is proved.
\end{proof}

We now show how $Q_\beta$ is prepared to generate an approximate solution to \eqref{E:KDEFINITION} with the suitable leading order dynamical system for $(\beta,
\l)$ induced by the spectral diagonalisation of $\H_b$.

\begin{proposition}[Leading order modulation equations]
\label{approxsolution}
Under the a priori bounds of Proposition \ref{pr:bootstrap}, there holds \be
\label{equationqbeta}
\pa_sQ_{\beta}+\H_ aQ_\beta={\rm Mod} +\Psi
\ee
where we defined the modulation vector
\bea
\label{equationmodulation}
\Mod& : =& \left[(b_k)_s+bb_k\l_{b,k}+\frac{2(a-b)b_k}{|\ln b|}+\frac{kb_k}{b}\Phi\right]\eta_{b,k}\\
\nonumber & + & \sum_{j=0}^{k-1} \left[(b_j)_s+bb_j\l_{b,j}+\frac{2(a-b)b_j}{|\ln b|}+\frac{jb_j-(j+1)b_{j+1}}{b}\Phi\right]\eta_{b,j},
\eea
the deviation 
\be
\label{defphi}
\Phi : =b_s+2b(a-b),
\ee
and the remaining error satisfies the bound:
\be
\label{esterreur}
\|\Psi\|_b+\frac{1}{\sqrt{b}}\|\pa_y\Psi\|_b+\frac{1}{b}\|\mathcal H_b\Psi\|_b\lesssim \frac{b^{\frac 32}}{|\ln b|}+\frac{|\Phi|}{\sqrt{b}}.
\ee
\end{proposition}

\begin{proof}[Proof of Proposition \ref{approxsolution}]
By definition $$\H_a=\H_b+(a-b)\Lambda $$ and we therefore compute from \eqref{defqbeta}:
\bee
&&\pa_sQ_{\beta(s)}(y)+\H_aQ_\beta =\sum_{j=0}^{k}\left[(b_j)_s\eta_{b,j}+b_s\frac{b_j}{b}b\pa_b\eta_{b,j}+bb_j\l_{b,j}\eta_{b,j}+(a-b)b_j\Lambda \eta_{b,j}\right]\\
& = & \sum_{j=0}^{k}\left\{[(b_j)_s+bb_j\l_{b,j}]\eta_{b,j}+(a-b)b_j[\Lambda \eta_{b,j}-2b\pa_b\eta_{b,j}]+\frac{b_j}{b}b\pa_b\eta_{b,j}\Phi\right\}\\
& = & \sum_{j=0}^{k}\left\{\left[(b_j)_s+bb_j\l_{b,j}+\frac{2(a-b)b_j}{|\ln b|}\right]\eta_{b,j}+\frac{jb_j\Phi}{b}[\eta_{b,j}-\eta_{b,j-1}]\right.\\
&+ & \left. (a-b)b_j[\Lambda \eta_{b,j}-2b\pa_b\eta_{b,j}-\frac2{|\ln b|}\eta_{b,j}]+\frac{b_j}{b}\Phi\left[b\pa_b\eta_{b,j}-j(\eta_{b,j}-\eta_{b,j-1})\right]\right\}.
\eee 
The bounds \eqref{foihonononfe}, \eqref{exactcomputationetabis}, \eqref{roughboundbj}, \eqref{controla} now yield \eqref{esterreur}.
\end{proof}

\begin{remark}
The presence of the $|\ln b|$ in the denominator on the right-hand side of~\eqref{esterreur} makes it 
a true error term with respect to our bootstrap regime, and this term is one of the leading order errors when closing the energy estimates in sections~\ref{S:MODULATION} and~\ref{S:ENERGYBOUND}. 
\end{remark}


\subsection{Modulation equations}\label{S:MODULATION}


From~\eqref{E:KDEFINITION},~\eqref{equationqbeta} we obtain the equation satisfied by the perturbation $\e:$
\be
\label{E:EQUATION}
\pa_s\e+\H_a\e=\matchal F
\ee
where 
\be
\label{E:FDEFINITION}
\mathcal F=-\Mod-\Psi.
\ee
The nonlinear decomposition and the orthogonality conditions~\eqref{orthoe} generate a differential equation for the modulation vector
$\beta=(b_j)_{0\leq j\leq k}$  in the setting of the bootstrap lemma \ref{pr:bootstrap} which we now compute exactly.


\begin{lemma}[Modulation equations for $b_j$]
\label{lemmamodulation}
\noindent{\em 1. $k=0$}: the $b$ law is given by:
\be
\label{E:KBOUND}
\left|b_s+\frac{2b^2}{|\ln b|}\right|\lesssim\frac{b^2}{|\ln b|^2}
\ee
\noindent{\em 2. $k\geq1$}:  the modulation dynamical system for the vector $(\bt_j)_{0\leq j\leq k}$ is given by 
\bea
\label{eqbtk}
\nonumber
&& \left| (\bt_k)_s+\frac 1s\left[\bt_k+(k+1)\sum_{j=0}^k\bt_j\right]\right|+ \left| (\bt_{k-1})_s+\frac1s\left[\frac{k-1}{k}\bt_{k-1}-(k+1)\sum_{j=0}^k\bt_j\right]\right|\\
 & + & \sum_{j=0}^{k-2}\left|(\bt_j)_s+\frac{j}{ks}\bt_j\right| \lesssim \frac{b^2}{|\ln b|^2}+\frac{\sqrt{b}\sqrt{\mathcal E}}{|\log b|}.
\eea
\end{lemma}

\begin{remark} The constants in Lemma \ref{lemmamodulation} are independent of $D,K$, see also Remark~\ref{R:SQUAREROOT}. We need to keep track of the coupling between the modes
in \eqref{eqbtk} in order to study the linearised system close to $b_e^j$ and close the shooting argument, see \eqref{controlmodes}.
\end{remark}

\begin{proof}[Proof of Lemma \ref{lemmamodulation}] This lemma follows from the orthogonality conditions~\eqref{orthoe} and the boundary conditions of lemma~\ref{lemmaindeitnitues}.\\

\noindent{\bf step 1} Computation of $\Mod$. The $\Mod$ estimate follows from the sharp choice of orthogonality conditions \eqref{orthoe}. Indeed, for any $0\le j\le k$, we take the scalar product of \eqref{E:EQUATION} with $\eta_{b,j}$ and use the orthogonality condition \eqref{orthoe} to compute:
$$
-(\e,\pa_s\eta_{b,j})_b+\frac{b_s}{2}( \e,|y|^2\eta_{b,j})_b =  (\mathcal F,\eta_{b,j})-(a-b) (\Lambda \e,\eta_{b,j}).
$$ 
We now integrate by parts and use \eqref{orthoe} again to compute:
\bee
&&-(\e,\pa_s\eta_{b,j})_b+\frac{b_s}{2}( \e,|y|^2\eta_{b,j})_b+(a-b)(\Lamdba \e,\eta_{b,j})\\
& = & (\e,-\frac{b_s}{b}b\pa_b\eta_{b,j}+ \frac{b_s}{2}|y|^2\eta_{b,j}+(a-b)[-\Lamdba \eta_{b,j}+by^2\eta_{b,j}])\\
& = & (\e, \frac{\Phi}{b}\left[-b\pa_b\eta_{b,j}+\frac{by^2}{2}\eta_{b,j}\right]+(a-b)[2b\pa_b\eta_{b,j}-\Lamdba \eta_{b,j}])
\eee
We evaluate all terms in the above expression. From \eqref{cebivbfenoe}, \eqref{estderivative}, \eqref{coercunubis}:
$$
|(\e, \frac{\Phi}{b}\left[-b\pa_b\eta_{b,j}+\frac{by^2}{2}\eta_{b,j}\right])|\lesssim \frac{|\Phi|}{b}\|\e\|_b\frac{|\log b|}{\sqrt{b}}\lesssim \frac{|\Phi||\log b|}{b^2}\frac{\sqrt{\mathcal E}}{\sqrt{b}}.
$$ 
Similarly from \eqref{foihonononfe}, \eqref{scalirpeoduceta}, and \eqref{controla}:
$$
|(\e,(a-b)[2b\pa_b\eta_{b,j}-\Lamdba \eta_{b,j}])|\lesssim \frac{|b-a|}{\sqrt{b}}\|\e\|_b\lesssim \frac{\sqrt{\mathcal E}}{\sqrt{b}}.
$$
We now estimate the $\matchal F$ terms given by \eqref{E:FDEFINITION}. From \eqref{esterreur}, \eqref{scalirpeoduceta}:
$$|(\Psi,\eta_{b,j})_b|\lesssim \|\Psi\|_b\|\eta_{b,j}\|_b\lesssim \left[\frac{b^{\frac 32}}{|\ln b|}+\frac{|\Phi|}{\sqrt{b}}\right]\frac{|\ln b|}{\sqrt{b}}\lesssim b+\frac{|\Phi||\ln b|}{b}.$$  The collection of above bounds together 
with \eqref{scalirpeoduceta} and~\eqref{E:FDEFINITION} yields 
\be
\label{estiamtiemod}
\frac{|(\Mod, \eta_{b,j})_b|}{(\eta_{b,j},\eta_{b,j})_b}\lesssim\frac{b}{|\ln b|^2} \left[ b+\frac{|\Phi||\ln b|}{b}+\left[1+\frac{|\Phi||\log b|}{b^2}\right]\frac{\sqrt{\mathcal E}}{\sqrt{b}}\right].
\ee
We now argue differently depending on $k$.\\

\noindent{\bf step 2} Case $k=0$. In this case, $b=b_0$ and thus from \eqref{E:PARTIALYEPSILONBOUNDARY}:
$$\Phi=(b_0)_s+2b_0(a-b_0)=b_s+O\left(\frac{b^2}{|\log b|^2}+\frac{\sqrt{b}\sqrt{\mathcal E}}{|\log b|}\right).$$ 
This together with \eqref{roughboundbj} implies the bound:
\be
\label{reoughboundphi}
|\Phi|\lesssim \frac{b^2}{|\log b|}+|b_s|
\ee
and thus \eqref{estiamtiemod}, \eqref{equationmodulation}, and~\eqref{E:PARTIALYEPSILONBOUNDARY} imply:
\be
\label{cebeneeon}
|b_s+b^2\l_{b,0}|\lesssim \frac{b^2}{|\ln b|^2}+\frac{|b_s|}{|\ln b|}+\left[1+\frac{|b_s||\log b|}{b^2}\right]\frac{\sqrt{b}\sqrt{\mathcal E}}{|\ln b|^2}.
\ee 
Using the rough bound \eqref{roughboundbj}, this gives: 
\be
\label{estbspoitnwise}
\left|b_s\left[1+O\left(\frac{1}{|\ln b|}\right)\right]+\frac{2b^2}{|\ln b|}\left[1+O\left(\frac{1}{|\ln b|}\right)\right]\right|\lesssim \frac{b^2}{|\ln b|^2}
\ee 
and \eqref{E:KBOUND} follows.\\

\noindent{\bf step 3} Case $k\geq 1$. In this case, we have from \eqref{definitionb}: $$|b_s|\lesssim b^2$$ and we therefore need to estimate $\Phi$:
\bea
\label{estimatephi}
\nonumber \Phi&=&b_s+2b(a-b)=b_s+2b\left[a-\sum_{j=0}^kb_j\left(1+\frac{2\alpha_j}{|\ln b|}\right)\right]\\
\nonumber & +& 2b\sum_{j=0}^k(b_j-b_j^e)\left(1+\frac{2\alpha_j}{|\ln b|}\right)+2b\left[b_k^e\left(1+\frac{2\alpha_j}{|\ln b|}\right)-b\right]\\
& = & 2b\sum_{j=0}^k\bt_j+O\left(\frac{b^2}{|\ln b|^2}+\frac{\sqrt{b}\sqrt{\mathcal E}}{|\log b|}\right)
\eea
where we used \eqref{E:PARTIALYEPSILONBOUNDARY}, \eqref{approxode} and the bootstrap bounds \eqref{boundstablemode}, \eqref{aprioriboundbj} in the last step. This implies in particular the rough bound 
\be
\label{roughocntorlphi}
|\Phi|\leq \frac{b^2}{|\log b|}
\ee
From \eqref{estiamtiemod} it follows that: 
\be
\label{estmodet}
\frac{|(\Mod, \eta_{b,j})_b|}{(\eta_{b,j},\eta_{b,j})_b}\lesssim \frac{b^2}{|\ln b|^2}+\frac{\sqrt{b}\sqrt{\mathcal E}}{|\ln b|^2}\lesssim \frac{b^2}{|\log b|^2}.
\ee 
We now recall \eqref{valueaexact} and compute from \eqref{E:PARTIALYEPSILONBOUNDARY}, \eqref{approxode}:
\bea
\label{mainterma}
\nonumber a&=&\sum_{j=0}^k(b_j^e+\bt_j)\left[1+\frac{2\alpha_j}{|\log b|}\right]+O\left(\frac{b}{|\log b|^2}+\frac{\sqrt{\mathcal E}}{|\log b|\sqrt{b}}\right)\\
& = & a^e+\sum_{j=0}^k\bt_j+O\left(\frac{b}{|\log b|^2}+\frac{\sqrt{\mathcal E}}{|\log b|\sqrt{b}}\right).
\eea
We now use \eqref{estimatephi}, \eqref{approxode}, \eqref{mainterma} to compute explicitly:
\bee
&&(b_k)_s+bb_k\l_{b,k}+\frac{2(a-b)b_k}{|\log b|}+\frac{kb_k}{b}\Phi \notag \\
& = & (b_k^e+\bt_k)_s+b(b_k^e+\bt_k)\left[2k+\frac{2}{|\log b|}+O\left(\frac{1}{|\log b|^2}\right)\right]+\frac{2(a^e-b)(b^e_k+\bt_k)}{|\log b|}\\
&& + \frac{2(a-a^e)b_k}{|\log b|}+k(b^e_k+\bt_k)\left[2\sum_{j=0}^k\bt_j+O\left(\frac{b}{|\ln b|^2}+\frac{\sqrt{\mathcal E}}{\sqrt{b}|\log b|}\right)\right]\\
&=&  (\bt_k)_s+\frac 1s\left[\bt_k+(k+1)\sum_{j=0}^k\bt_j\right]+O\left(\frac{b^2}{|\ln b|^2}+\frac{\sqrt{b}\sqrt{\mathcal E}}{|\ln b|}\right). \label{E:MODULATION1}
\eee
similarly for $j=k-1$:
\bee
&&(b_{k-1})_s+bb_{k-1}\l_{b,{k-1}}+\frac{2(a-b)b_{k-1}}{|\ln b|}+\frac{(k-1)b_{k-1}-kb_{k}}{b}\Phi\\
& = &(\bt_{k-1})_s+b\bt_{k-1}\left[2(k-1)+\frac{2}{|\log b|}+O\left(\frac{1}{|\log b|^2}\right)\right]+\frac{2(a^e-b)\bt_{k-1}}{|\ln b|}\\
& +& \frac{2(a-a^e)\bt_{k-1}}{|\ln b|}+\left[(k-1)\bt_{k-1}-k(b_k^e+\bt_k)\right]\left[2\sum_{j=0}^k\bt_j+O\left(\frac{b}{|\ln b|^2}+\frac{\sqrt{\mathcal E}}{\sqrt{b} |\log b|}\right)\right]\\
& =&(\bt_{k-1})_s+\frac1s\left[\frac{k-1}{k}\bt_{k-1}-(k+1)\sum_{j=0}^k\bt_j\right]+ O\left(\frac{b^2}{|\ln b|^2}+\frac{\sqrt{b}\sqrt{\mathcal E}}{|\ln b|}\right).
\eee
Finally for $0\le j\leq k-2$:
\bee
&&(b_j)_s+bb_j\l_{b,j}+\frac{2(a-b)b_j}{|\ln b|}+\frac{jb_j-(j+1)b_{j+1}}{b}\Phi\\
& = & (\bt_j)_s+\frac{j}{ks}\bt_j+O\left(\frac{b^2}{|\ln b|^2}+\frac{\sqrt{b}\sqrt{\mathcal E}}{|\ln b|}\right).
\eee
Injecting the above bounds into \eqref{estmodet} yields \eqref{eqbtk}.
\end{proof}


\subsection{Energy bound}\label{S:ENERGYBOUND}

We now arrive at the second main feature of the analysis which is the derivation of suitable energy bounds for $\e$. 
The key here is the dissipation embedded in the problem and its geometry which feeds back into the energy estimates
through the boundary conditions. A careful analysis of this interaction will allow us to close the energy estimates.

\begin{proposition}[Energy bound]
\label{propenergy}
There holds the pointwise control:\\
\noindent{\em1. for $k=0$}:
\be
\label{estemeergyzero}
\frac 12\frac{d}{ds}\left\{\mathcal E+O\left(\frac{b^3}{|\ln b|^{2}}\right)\right\}+c b\matchal E\lesssim\frac{b^4}{|\ln b|^2};
\ee
\noindent{\em 2. for $k\ge 1$}: 
\be
\label{estemeergy}
\frac 12\frac{d}{ds}\left\{\mathcal E+O\left(\frac{b^3}{|\ln b|^{5/4}}\right)\right\}+\left[3k+c\right]b\matchal E\lesssim\frac{Kb^4}{|\ln b|}
\ee
for some universal constant $c>0$. 
\end{proposition}

\begin{remark} 
The sharp coercivity constant $3k$ in \eqref{estemeergy} which follows from the sharp Poincar\'e estimate \eqref{coercdeux} is essential to close the energy bound, see~\eqref{estimateenergy}. 
\end{remark}

\begin{proof}[Proof of Proposition \ref{propenergy}] We compute the energy identity for $\mathcal E$ and estimate all terms.\\

\noindent{\bf step 1} Algebraic energy identity. Recalling from~\eqref{defetwo} that $\e_2=\mathcal H_{b}\e,$ it follows from \eqref{E:EQUATION}:
$$
\pa_s\e_2+\H_a\e_2=[\pa_s,\H_{b}]\e+[\H_a,\H_{b}]\e+\H_{b}\mathcal F.
$$ 
To compute the commutators $[\pa_s,\H_{b}]$, $[\H_a,\H_{b}]$ we use 
\be
\label{commutatorlaplace}
[\Delta,\Lambda]=2\Delta
\ee
 which yields:
\bee
&&[\pa_s,\H_{b}]\e+[\H_a,\H_{b}]\e=b_s\Lambda \e+[\H_b+(a-b)\Lambda,\H_b]\e=b_s\Lambda \e+(a-b)[\Lambda,-\Delta]\\
& = & b_s\Lambda \e+2(a-b)\Delta \e=(b_s+2b(a-b))\Lambda \e-2(a-b)[-\Delta\e+b\Lambda\e]\\
& = & \Phi\Lambda \e+2(b-a)\e_2.
\eee
Hence the $\e_2$ equation: 
\be
\label{eqationaetwo}
\pa_s\e_2+\H_a\e_2=\Phi\Lambda \e+2(b-a)\e_2+\H_b\mathcal F.
\ee
We now compute the modified energy identity:
\bee
\nonumber \frac 12\frac{d}{ds}\mathcal E& = & \frac 12\frac{d}{ds}\int_{y\geq 1} \e_2^2e^{-\frac{by^2}{2}}ydy=  -\frac{b_s}{4}\int_{y\geq 1} y^2|\e_2|^2e^{-\frac{b|y|^2}{2}}ydy+(\pa_s\e_2,\e_2)_b\\
& = &  -\frac{b_s}{4}\|y\e_2\|_b^2+(\Phi\Lambda \e+2(b-a)\e_2+\H_b\mathcal F-\H_a\e_2,\e_2)_b.
\eee
We carefully integrate by parts to compute:
\bee
&&-\int_{y\geq 1}\e_2\H_a\e_2e^{-\frac{by^2}{2}}ydy=-\int_{y\geq 1}\e_2[\H_b\e_2+(a-b)\Lambda \e_2]e^{-\frac{by^2}{2}}ydy\\
& =& \int_{y\geq 1}\pa_y(\rho_b y\pa_y\e_2)\e_2dy+(b-a)\int \e_2y\pa_y\e_2e^{-\frac{by^2}{2}}ydy\\
& = & -\rho_b(1)\e_2(1)\pa_y\e_2(1)-\int_{y\geq 1}|\pa_y\e_2|^2e^{-\frac{by^2}{2}}ydy\\
&&  - \frac{b-a}{2}\rho_b(1)\e_2^2(1)-\frac{b-a}{2}\int_{y\geq 1}\e_2^2\left[2-by^2\right]e^{-\frac{by^2}{2}}ydy\\
& = &  -\|\pa_y\e_2\|_b^2+(a-b)\|\e_2\|_2^2-\frac{b(a-b)}{2}\|y\e_2\|_b^2-\rho_b(1)\e_2(1)\left[\pa_y\e_2+\frac{b-a}{2}\e_2\right](1).
\eee
This yields the algebraic energy identity:
\bea
\label{algebraicneergyidentity}
\nonumber \frac 12\frac{d}{ds}\mathcal E&=& -\|\pa_y\e_2\|_b^2-(a-b)\|\e_2\|_2^2-\frac{\Phi}{4}\|y\e_2\|_b^2-\rho_b\e_2\left[\pa_y\e_2+\frac{b-a}{2}\e_2\right](1)\\
& + & \Phi(\Lambda \e,\e_2)_b+(\H_b\mathcal F,\e_2)_b.
\eea

We now estimate all  terms in the right-hand side of \eqref{algebraicneergyidentity}.\\

\noindent{\bf step 2} Nonlinear estimates. From \eqref{coercunubis}, \eqref{roughocntorlphi}, \eqref{roughboundbj}:
\bee
|\Phi||(\Lambda\e,\e_2)_b|\lesssim \frac{b^2}{|\ln b|}\frac{\sqrt{\mathcal E}}{b}\sqrt{\mathcal E}\lesssim b\frac{\mathcal E}{|\ln b|}.
\eee
Moreover from \eqref{orthoe} $(\H_b\Mod,\e_2)_b=0,$ and from \eqref{esterreur}, \fref{roughocntorlphi}:
$$|(\H_b\Psi,\e_2)_b\lesssim b\sqrt{\mathcal E}\frac{b^{\frac 32}}{|\ln b|}.$$
We now estimate from \eqref{E:YHBEPSILONBOUND}, \eqref{roughocntorlphi}, \eqref{etwoone}, \eqref{controla}:
$$|\Phi\|y\e_2\|_b^2\lesssim \frac{b^2}{|\ln b|}\left[\frac{\|\pa_y\e_2\|^2_b}{b^2}+b^3\right]\lesssim \frac{\|\pa_y\e_2\|^2_b}{|\ln b|}+\frac{b^4}{|\ln b|^2}.
$$
It now remains to treat the boundary term in \eqref{algebraicneergyidentity} and we argue differently depending on $k$.\\

\noindent{\bf step 3} Conclusion for $k=0$. We compute from \eqref{etwoonebis}, \eqref{E:KBOUND}:
\bee
\pa_y\e_2(1)&=&-a_s-\frac{2b^2}{|\ln b|}+O\left(\frac{b^2}{|\ln b|^2}\right)=-(a-b)_s -(b_s+\frac{2b^2}{|\ln b|})+O\left(\frac{b^2}{|\ln b|^2}\right)\\
& = & -(a-b)_s+O\left( \frac{\sqrt{b}\sqrt{\mathcal E}}{|\ln b|}+\frac{b^2}{|\ln b|^2}\right)
\eee
and hence using \eqref{etwoone}:
\bee
&&\rho_b(1)\e_2(1)\left[\pa_y\e_2+\frac{b-a}{2}\e_2\right](1)\\
&=& e^{-\frac{b}{2}}a(a-b)\left[-(a-b)_s+\frac{a(a-b)^2}{2}+O\left( \frac{\sqrt{b}\sqrt{\mathcal E}}{|\ln b|}+\frac{b^2}{|\ln b|^2}\right)\right]\\
& = & e^{-\frac b2}\left[-(a-b)^2(a-b)_s-b(a-b)(a-b)_s\right]+O\left(b\frac{b^{\frac 32}\sqrt{\mathcal E}}{|\ln b|}+\frac{b^4}{|\ln b|^2}\right)\\
& = & -\frac{d}{ds}\left\{e^{-\frac b2}\frac{(a-b)^3}{6}+e^{-\frac b2}\frac{b(a-b)^2}{2}\right\} - \frac{b_se^{-\frac b2}}{2}\frac{(a-b)^3}{6}\\
&& +e^{-\frac b2}\frac{b(a-b)^2}{2}\left[-\frac{b_s}{2}+\frac{b_s} b\right]+O\left(b\frac{b^{\frac 32}\sqrt{\mathcal E}}{|\ln b|}+\frac{b^4}{|\ln b|^2}\right)\\
& = &  -\frac{d}{ds}\left\{e^{-\frac b2}\frac{(a-b)^3}{6}+e^{-\frac b2}\frac{b(a-b)^2}{2}\right\}+O\left(b\frac{b^{\frac 32}\sqrt{\mathcal E}}{|\ln b|}+\frac{b^4}{|\ln b|^2}+\frac{b^2(a-b)^2}{|\ln b|}\right).
\eee
We now observe from \eqref{E:PARTIALYEPSILONBOUNDARY}, \eqref{roughboundbj} that 
\be
\label{cejbevbevi}
|a-b|\lesssim \frac{b}{|\ln b|},
\ee and hence the collection of above bounds yields the control:
\bee
\frac 12\frac{d}{ds}\left\{\mathcal E+O\left(\frac{b^3}{|\ln b|^2}\right)\right\}=-\|\pa_y\e_2\|_b^2+O\left(\frac{\|\pa_y\e_2\|_b^2+\|\e_2\|_b^2}{|\ln b|}+b\frac{b^{\frac 32}\sqrt{\mathcal E}}{|\ln b|}+\frac{b^4}{|\ln b|^2}\right).
\eee
We now inject \eqref{etwoone}, \eqref{coercdeuxbis} with $k=0$ and \eqref{estemeergyzero} follows.\\

\noindent{\bf step 4} Conclusion for $k\ge 1$. 
Let  
$$
\at=a^e+\sum_{j=0}^k\bt_j
$$ 
be the leading order part of $a,$ see~\eqref{mainterma}.
Then from \eqref{eqbtk}, \eqref{aprioriboundbj}, \eqref{approxode}:
\bee
&&\tilde{a}_s+\sum_{j=0}^k\l_{b,j}bb_j\left[1+\frac{2\alpha_j}{|\ln b|}\right]\\
&&= (a^e)_s+\sum_{j=0}^k(\bt_j)_s+\sum_{j=0}^k\l_{b,j}b(b^e_j+\bt_j)\left[1+\frac{2\alpha_j}{|\ln b|}\right]\\
& = & \left(a^e-b^e_k\left[1+\frac{2\alpha_k}{|\log b|}\right]\right)_s+\left(1+\frac{2\alpha_k}{|\log b|}\right)\left[(b^e_k)_s+\l_{b,j}bb^e_k\right]+O\left(\frac{b^2}{|\log b|^2}\right)\\
&+& O\left(\frac{K b^2}{|\ln b|^{3/2}}+\frac{\sqrt{b}\sqrt{\mathcal E}}{|\ln b|}\right)= -2\frac{(a^e-b)b_e^k}{|\log b|}+O\left(\frac{K b^2}{|\ln b|^{3/2}}+\frac{\sqrt{b}\sqrt{\mathcal E}}{|\ln b|}\right) \\
& = & -\frac{k+1}{2k s^2|\log s|}+O\left(\frac{K b^2}{|\ln b|^{3/2}}+\frac{\sqrt{b}\sqrt{\mathcal E}}{|\ln b|}\right).
\eee
We therefore estimate using Lemma \ref{lemmaindeitnitues}:
\bee
&&-\rho_b\e_2\left[\pa_y\e_2+\frac{b-a}{2}\e_2\right](1)\\
&=& -e^{-\frac{b}{2}}a(a-b)\left[a_s+\sum_{j=0}^k\l_{b,j}bb_j\left[1+\frac{2\alpha_j}{|\ln b|}\right]+O\left(\frac{b^2}{|\ln b|^2}\right)\right]\\
& = & -e^{-\frac{b}{2}}a(a-b)\left[(a-\tilde{a})_s-\frac{k+1}{2k s^2\ln s}+O\left(\frac{K b^2}{|\ln b|^{3/2}}+\frac{\sqrt{b}\sqrt{\mathcal E}}{|\ln b|}\right)\right]\\
& = & -e^{-\frac{b}{2}}a(a-b)(a-\tilde{a})_s+ e^{-\frac{b}{2}}a(a-b)\frac{k+1}{2k s^2\ln s} +O\left(\frac{K b^4}{|\ln b|^{3/2}}+\frac{bb^{\frac 32}\sqrt{\mathcal E}}{|\ln b|}\right)\\
& = & -e^{-\frac b2}(a-\at)_s\left[(a-\at)^2+(a-\at)(2\at-b)+\at(\at-b)\right]\\
&& + O(\frac{b^4}{|\ln b|}) +O\left(\frac{K b^4}{|\ln b|^{3/2}}+\frac{bb^{\frac 32}\sqrt{\mathcal E}}{|\ln b|}\right)\\
& = & -\frac{d}{ds}\left\{-e^{-\frac b2}\left[\frac{(a-\at)^3}{3}+\frac{(a-\at)^2}{2}(2\at-b)+(a-\at)\at(\at-b)\right]\right\}\\
&& + O\left(\frac{b^4}{|\ln b|} +\frac{K b^4}{|\ln b|^{3/2}}+\frac{bb^{\frac 32}\sqrt{\mathcal E}}{|\ln b|}+b^3|a-\at|\right).
\eee
We have the bound from \eqref{mainterma}:
\bee
a-\at
& = & O\left(\frac{b}{|\ln b|}+\frac{\sqrt{\mathcal E}}{|\ln b|\sqrt{b}}\right).
\eee
Injecting the above bounds into \eqref{algebraicneergyidentity} and using the rough bound~\eqref{roughboundbj} we obtain the bound:
\begin{align*}
\frac 12\frac{d}{ds}\left\{\mathcal E+O\left(\frac{b^3}{|\ln b|} \right)   \right\}
&= -\|\pa_y\e_2\|_b^2-(a-b)\|\e_2\|_2^2\\
& \ \ \ + O\left(\frac{\|\pa_y\e_2\|_b^2+\|\e_2\|_b^2}{|\ln b|}+b\frac{b^{\frac 32}\sqrt{\mathcal E}}{|\ln b|}+\frac{b^4}{|\ln b|}\right).
\end{align*}
We now estimate from \eqref{mainterma}:
\be\label{E:AMINUSB}
a-b=\frac{1}{2s}+O\left(\frac{b}{|\ln b|}\right)=kb+O\left(\frac{b}{|\ln b|}\right)
\ee 
and \eqref{estemeergy} now follows from \eqref{coercdeuxbis} with \eqref{etwoone}.
\end{proof}


\subsection{Proof of Proposition~\ref{pr:bootstrap}}\label{S:PROOFPROPOSITION}


We are now in position to give a sharp description of the singularity formation for our set of initial data. The key is to close the bootstrap bounds of Proposition \ref{pr:bootstrap}. We distinguish the cases $k=0$ and $k\ge 1$.\\

{\bf step 1} Closing the bootstrap bounds. 
Our goal is to show that the bounds~\eqref{E:BOOTBOUND}, \eqref{positivitybzero} improve in the case $k=0$
and similarly for the bounds \eqref{E:BOOTBOUND}, \eqref{boundstablemode}, \eqref{aprioriboundbj} in the case $k\ge1.$ The 
improvement of the energy bound~\eqref{E:BOOTBOUND} will follow from proposition~\ref{propenergy}, while the bounds \eqref{boundstablemode}, \eqref{aprioriboundbj}
will be improved for a suitable set of initial data constructed via a topological argument.

\vspace{0.1in}
\noindent\underline{\bf $k=0$}. 
 First observe that \eqref{E:KBOUND} ensures 
$$b_s<0\ \ \mbox{and hence}\ \ b(s)<b_0\leq b^*.$$ From \eqref{cejbevbevi}, 
\be
\label{estab}
\lsl+b=b-a=O\left(\frac{b}{|\ln b|}\right)
\ee and hence: 
\be
\label{vninoiioneneo}
\ln \l(s)=-\int_{0}^{s}b\left[1+O\left(\frac1{|\ln b|}\right)\right]d\sigma<+\infty \ \ \mbox{implies}\ \ \l(s)>0.
\ee
We now rewrite  \eqref{estemeergy} (with $k=0$) as 
$$
\frac{d}{ds}\left\{\frac 12\mathcal E+O\left(\frac{b^3}{|\ln b|^{2}}\right)\right\}+bc\left[\mathcal E+O\left(\frac{b^3}{|\ln b|^{2}}\right)\right]\lesssim \frac{b^4}{|\ln b|^2}
$$
with $c>0.$ Using~\eqref{estab} we obtain:
\be
\label{cejebneove}\frac{d}{ds}\left\{\frac{1}{\l^c}\left[\frac 12\mathcal E+O\left(\frac{b^3}{|\ln b|^{2}}\right)\right]\right\}\lesssim 
\frac{b^4}{\l^c|\ln b|^2}.
\ee
We now integrate in time. To evaluate the right hand side, we integrate by parts using \eqref{E:KBOUND}:
\bee
&&\int_0^s\frac{b^4}{\l^c|\ln b|^2}d\sigma =  \int_0^s\left[-\lsl\frac{b^3}{\l^c|\ln b|^2}+O\left(\frac{b^4}{|\ln b|^3}\right)\right]d\sigma\\
&= & \left[\frac 1c\frac{b^3}{\l^c|\ln b|^2}\right]_0^s-\frac 1c\int_0^s\frac{b_s}{\l^c}\left[\frac{3b^2}{|\ln b|^2}+\frac{2b^2}{|\ln b|^3}\right]d\sigma+\int_0^sO\left(\frac{b^4}{\l^c|\ln b|^3}\right)d\sigma.
\eee
Using the smallness of $b$ we get: 
$$
\int_0^s\frac{b^4}{\l^c|\ln b|^2}d\sigma \lesssim \frac{b^3}{\l^c|\ln b|^2}(s).
$$ 
Hence~\eqref{normalizationinitial}, \eqref{E:ESMALL} and the time integration of \eqref{cejebneove} ensure:
\be
\label{neonvenenvoee}
\matchal E(s)\lesssim \frac{b^3}{|\ln b|^2}(s)+\l^c(s)\left[\matchal E(0)+ \frac{b_0^3}{|\ln b_0|^{2}}\right]\lesssim \frac{b^3}{|\ln b|^2}(s)+\l^c(s)\frac{b_0^3}{|\ln b_0|^2}.
\ee
We moreover estimate from \eqref{E:KBOUND}:
\bee
\frac{d}{ds}\left\{\frac{b^3}{\l^c|\ln b|^2}\right\}=\frac{b^3}{\l^c|\ln b|^2}\left[\frac{3b_s}{b}+\frac{2b_s}{b|\ln b|}-c\lsl\right]
=\frac{b^3}{\l^c|\ln b|^2}\left[cb+O\left(\frac{b}{|\ln b|^2}\right)\right]>0
\eee
and hence using  \eqref{normalizationinitial} again:
$$
\l^c(s)\frac{b_0^3}{|\ln b_0|^2}\leq \l^c(s_0)\frac{b^3}{|\ln b|^2}(s)
$$ 
which together with \eqref{vninoiioneneo} implies $b(s)>0$ and closes the bound \eqref{positivitybzero}.
Injecting this into \eqref{neonvenenvoee}  improves  the energy bound \eqref{E:BOOTBOUND} for $D$ universal large enough, which concludes the proof of Proposition \ref{pr:bootstrap} for $k=0$.\\

\noindent{\underline{$k\geq 1$}}. This case requires a shooting argument to build the nonlinear manifold of perturbations $(V_j)_{0\leq j\leq k-1}$. We first rewrite \eqref{estemeergy} using \eqref{definitionb}:
$$
\frac{d}{ds}\left\{\mathcal E+O\left(\frac{b^3}{|\ln b|}\right)  \right\}+\left[3+\frac{c}{k}\right]\frac{1}s\matchal E\lesssim  \frac{K}{s^4|\ln s|}.
$$ 
Using \eqref{E:ESMALL}, an integration-in-time yields:
\bea
\label{estimateenergy}
\nonumber  \mathcal E(s)&\leq& \frac{s_0^{3+\frac ck}}{s^{3+\frac ck}}\left[\mathcal E_0+\frac{b_0^3}{|\log b_0|}\right]+\frac{b^3}{|\log b|}+\frac{1}{s^{3+\frac ck}}\int_{s_0}^s  \frac{K\sigma^{3+\frac ck}}{\sigma^4|\ln \sigma|}d\sigma\\
 & \lesssim&  \frac{K}{s^{3}(\ln s)}\lesssim K\frac{b^3}{|\ln b|}.
\eea
This means that there exists a $\tilde C>0$ universal large enough such that if
$D = \tilde C K,$
the bound \eqref{E:BOOTBOUND} gets improved, and we assume it now.
We inject this relation into \eqref{eqbtk} and conclude:
\bee
&& \left| (\bt_k)_s+\frac 1s\left[\bt_k+(k+1)\sum_{j=0}^k\bt_j\right]\right|+ \left| (\bt_{k-1})_s+\frac1s\left[\frac{k-1}{k}\bt_{k-1}-(k+1)\sum_{j=0}^k\bt_j\right]\right|\\
 & + & \sum_{j=0}^{k-2}\left|(\bt_j)_s+\frac{j}{ks}\bt_j\right| \lesssim 
\frac{\sqrt{K}b^2}{|\ln b|^{3/2}}\lesssim \frac{\sqrt{K}}{s^2(\ln s)^{3/2}}.
\eee
Equivalently using the change of variables \eqref{bveuibvibi}:
\bea
\label{controlmodes}
 \nonumber &&\left| (V_k)_s+\frac{(k+1)\sum_{j=0}^kV_j}{s}\right|+ \left| (V_{k-1})_s+\frac1s\left[-\frac{1}{k}V_{k-1}-(k+1)\sum_{j=0}^kV_j\right]\right|\\
 & + & \sum_{j=0}^{k-2}\left|(V_j)_s+\frac{j-k}{ks}V_j\right|\lesssim \frac{\sqrt{K}}{s}.
\eea
The bootstrap bound~\eqref{aprioriboundbj} implies that $|V_j|\le \delta K,$ $j=0,\dots k-2.$ Therefore, from the first two bounds
in~\eqref{controlmodes} we conclude that for a sufficiently large $K$ the following bound holds
\[
\left| (V_k)_s+\frac{(k+1)}s\left(V_k+V_{k-1}\right)\right|+ \left| (V_{k-1})_s-\frac{k+1}s\left(V_k+ (1+d_k)V_{k-1}\right)\right| \lesssim \frac{\delta K}{s},
\]
where we remind the reader that $d_k= \frac 1{k(k+1)}.$ 
Recalling the definition~\eqref{E:AKMATRIX} of the matrix $A_k,$ the above inequalities can be succinctly rewritten in the form 
\begin{align*}
\pa_s\left( \begin{array}{c}
 V_k  \\
V_{k-1}    
\end{array} \right) 
= 
\frac{k+1}sA_k \left( \begin{array}{c}
 V_k  \\
V_{k-1}    
\end{array} \right) + O\left(\frac{\delta K}s\right),
\end{align*}
which in turn leads to 
\be\label{E:WKBOUND}
\left| (W_k)_s+\frac{(k+1)\mu^k_1}sW_k\right|+ \left| (W_{k-1})_s+\frac{(k+1)\mu^k_2}sW_{k-1}\right| \lesssim \frac{\delta K}{s},
\ee
where $W_k,W_{k-1}$ are defined in~\eqref{E:WKDEFINITION} and $\mu^k_2<0<\mu^k_1$ are the eigenvalues of $A_k.$
This first yields the control of the stable direction $W_k,$ since after integrating-in-time the first bound in~\eqref{E:WKBOUND}, we arrive at  
$$
|W_k(s)|\leq |W_k(0)|\frac{s_0^{(k+1)\mu^k_1}}{s^{(k+1)\mu^k_2}}+\frac{1}{s^{(k+1)\mu_1^k}}\int_{s_0}^s\frac{\delta K\sigma^{(k+1)\mu_1^k}}{\sigma}d\sigma\le 1+ C\delta K,
$$ 
where we used~\eqref{estinitialbound} and the positivity of $\mu^k_1.$
This improves \eqref{boundstablemode} for $K$ sufficiently large and $\delta <\frac 1{2C}.$ 
We now argue by contradiction and assume that for all $(\frac{V_0}{\delta},\dots,\frac{V_{k-2}}{\delta},W_{k-1})\in B_K(\Bbb R^{d-1})$, the bootstrap time $s^*$ is finite, so that from \eqref{aprioriboundbj}:
\be
\label{estbitabvuie}
|W_{k-1}(s^*)|^2+\sum_{j=0}^{k-2}\left|\frac{V_j(s^*)}{\delta}\right|^2=K^2.
\ee
We claim that this contradicts the Brouwer fixed point theorem. Indeed, using~\eqref{controlmodes},~\eqref{E:WKBOUND}, and the strict negativity of $\mu^k_2$:
\bee
&&\frac12\frac{d}{ds}\left\{|W_{k-1}(s)|^2+\sum_{j=0}^{k-2}\left|\frac{V_j(s)}{\delta}\right|^2\right\}(s^*)\\
& = & \frac{1}{s^*}\left[|\mu_2^k|(k+1)W_{k-1}^2(s^*)+\sum_{j=0}^{k-2}\frac{k-j}{\delta^2k}V_j^2(s^*)+O\left(\delta K^2+K^{3/2}\right)\right]\\
& \geq & \frac{c}{s^*}\left[K^2- C\delta K^2\right],
\eee
for some universal constants $c,C>0.$ 
Hence
\be\label{E:SIGN}
\frac{d}{ds}\left\{|W_{k-1}(s^*)|^2+\sum_{j=0}^{k-2}\left|\frac{V_j(s^*)}{\delta}\right|^2\right\}(s^*)>0
\ee 
for $0<\delta\ll1 $ universal small enough in \eqref{aprioriboundbj}. 
Let 
$$
\tilde{V}=(\frac{V_0}{\delta}, \dots,\frac{V_{k-2}}{\delta},W_{k-1}),
$$ 
then this implies from standard from standard argument that the map 
$$
 B_{K}(\RR^{d-1}) \owns \tilde{V}(0) \mapsto s^*\left(\tilde{V}(0)\right)
$$ 
is continuous, and hence the map 
$$ 
\begin{array}{ll}  B_{K}(\RR^{d-1})\to B_{K}(\RR^{d-1})\\
\tilde{V}(0) \mapsto \tilde{V}\left[\tilde{s}^*( \tilde{V}(0))\right]\end{array}
$$
is continuous and the identity on the boundary sphere $\Bbb S_{d-1}(K)$, a contradiction to Brouwer's fixed point theorem. 
This concludes the proof of Proposition \ref{pr:bootstrap} for $k\geq 1$.

\begin{remark} Note that $(V_0(\e_0),\dots.V_{k-2}(\e_0), W_{k-1}(\e_0))$ are by construction lying on a nonlinear codimension $k$ manifold of initial data.  
The fact that the set of such initial data forms a Lipschitz manifold in the $H^2$ topology reduces to a local uniqueness problem in the class of solutions satisfying the a priori bounds of Proposition \ref{pr:bootstrap}. 
Such a uniqueness problem has been recently solved in a related framework in the more complicated case of the wave equation~\cite{collot} and the KdV equation \cite{MMN}, see also~\cite{szeftel, CRS}, and a completely analogous approach can be applied here. 
We therefore omit the details.
\end{remark}

\begin{remark}\label{R:SQUAREROOT}
Note that the presence of $\sqrt K$ on the right-hand side of~\eqref{controlmodes} is essential to the closure of the estimates. 
It originates from the bound~\eqref{eqbtk}, where we carefully tracked the constants  and proved that 
only $\sqrt{\mathcal E}$ appears on the right-hand side of~\eqref{eqbtk}.
\end{remark}

\noindent{\bf step 2} Global $H^2$ control. From 
Proposition~\ref{pr:bootstrap}, the solution remains in the bootstrap regime of Proposition~\ref{pr:bootstrap} as long as it exists in  $H^2$ which requires: $\forall s\geq 0$,
\be
\label{globalbound}
\|u(s)\|_{L^2(|x|\ge \l(s))}+ \|\nabla u(s)\|_{L^2(|x|\geq \l(s))}+\|\Delta u(s)\|_{L^2(|x|\geq \l(s))}<+\infty \ \ \ee
and 
\be
\label{positibity}
\l(s)>0.
\ee
In the case $k=0,$ the positivity of $\l$ follows from the time integration of \eqref{E:KBOUND} which implies 
$$
|\ln \lambda(s)|\lesssim \int_0^s b \left(1+O(\frac{1}{|\ln b|})\right)d\sigma<+\infty,
$$ 

while in the case $k\ge 1$ we use $a=-\l_s/\l$ and the estimate~\eqref{E:AMINUSB}, which implies the above bound again.
The global $L^2$-bound follows from the basic dissipation law satisfied by the solutions of~\eqref{E:STEFAN}:
\[
\frac12\frac{d}{dt} \|u\|_{L^2(\Omega(t))}^2 + \|\nabla u \|_{L^2(\Omega(t))}^2 = 0,
\]
which immediately implies that $\|u(s)\|_{L^2(|x|\ge \l(s))}<\infty.$
The global $\dot{H}^1$-bound follows directly from the dissipation of the Dirichlet energy \eqref{decaydirichlet}. 
For the global $\dot{H}^2$-bound, we take a cut off function $\chi=0$ for $r\leq 1$ and $r=1$ for $r\geq 2$, then the 
weighted control \eqref{E:BOOTBOUND} ensures 
$$
\|\Delta u(s)\|_{L^2(\l(s)\leq r\leq 2)}<C(s)<+\infty\ \ \mbox{for}\ \ s\geq 0
$$ 
since the exponential weight is uniformly bounded from below and above in $\l(s)\leq r\leq 2.$ 
To obtain the bound in the region $r\ge 2$ we compute:
\begin{align}
\frac12 \frac{d}{dt}\int \chi|\Delta u|^2&=\int \chi \Delta \pa_t u\Delta u=\int \chi \Delta^2 u\Delta u \notag \\
& =-\int\chi|\nabla \Delta u|^2+\frac 12\int \Delta \chi |\Delta u|^2 \label{E:CHIESTIMATE}
\end{align}
 and hence 
 \begin{align*}
 \int \chi|\Delta u|^2(s) 
  \lesssim \|\chi\Delta u(0)\|_{L^2}^2+\int_0^s\frac{1}{\l^2(\sigma)}\|\Delta u(s)\|_{L^2(\l(\sigma)\leq r\leq 2)}^2d\sigma<+\infty.
 \end{align*}
Hence $s^*=+\infty$ which concludes the proof of Proposition \ref{pr:bootstrap}.


\subsection{Proof of Theorem~\ref{T:MAIN}}\label{S:PROOFMAINTHEOREM}


We are now in position to conclude the proof of Theorem \ref{T:MAIN}.\\

\noindent{\bf step 1} Finite time melting. We claim that the solution melts in finite time with the law \eqref{E:MELTINGRATE}, \eqref{meltingk} as a consequence of the time integration of the modulation equations.\\
\noindent\underline{case $k=0$:} From~\eqref{E:KBOUND}, \eqref{E:BOOTBOUND}, we obtain the following pointwise differential inequality for $b:$
\begin{align}\label{E:DIFFERENTIALINEQUALITY}
b_s + \frac{2b^2}{|\ln b|} =O\left( \frac{ b^2}{|\ln  b|^{2}}\right).
\end{align}
We now follow \cite{RaphRod,RSc} to derive the melting speed of $\lambda$ and sketch the proof for the sake of clarity. Multiplying~\eqref{E:DIFFERENTIALINEQUALITY} by $\frac{\ln b}{ b^2}$ we obtain
\begin{align} \notag 
\frac{ b_s\ln b}{ b^2} = 2 + O\left(\frac{1}{|\ln  b|}\right).
\end{align}
The primitive of $\frac{\ln u}{u^2}$ is  $-\frac{\ln u}{u}-\frac{1}{u}$ and therefore
\begin{align}\notag 
\frac{\ln b}{ b}+\frac{1}{ b} = -2s + O(\int_0^s\frac{1}{|\ln  b|}\,d\tau)
\end{align}
which implies $$\frac{\ln s}{s}\lesssim b\lesssim \frac{\ln s}{s}.$$
Hence: 
\be
\label{vbeivbvbei}
b=-\frac{1+\ln b}{2s}\left[1+O\left(\frac1s\int_0^s\frac{1}{|\ln  b|}\,d\tau\right)\right]=-\frac{1+\ln b}{2s}\left[1+O\left(\frac{1}{|\ln b|}\right)\right]
\ee 
Taking the $\ln$ yields
$$
\ln b=-\ln 2-\ln s+\ln(-\ln b)+O\left(\frac{1}{|\ln b|}\right)
$$ 
which reinserted into \eqref{vbeivbvbei} ensures:  
\be\label{E:AUX1}
b = \frac{\ln s}{2s}\left[1+O\left(\frac{\ln\ln s}{\ln s}\right)\right],  \ \ \ln b=\ln\ln s-\ln2 -\ln s+O\left(\frac{\ln\ln s}{\ln s}\right).
\ee
Injecting this into \eqref{vbeivbvbei} again yields:
\bea
\label{E:BSOLVED}
\nonumber 
b&=&-\frac{-\ln s+\ln\ln s-\ln 2+1+O\left(\frac{\ln\ln s}{\ln s}\right)}{2s}\left[1+O\left(\frac{1}{\ln s}\right)\right]\\
& = & \frac{\ln s}{2s}-\frac{\ln\ln s}{2s}+O\left(\frac1s\right).
\eea
Recalling from \eqref{E:PARTIALYEPSILONBOUNDARY} that 
\be
\label{cbwbeibeibe}
b = -\frac{\lambda_s}{\lambda} + O\left(\frac{b}{|\ln b|^{2}}\right),
\ee
we conclude that 
\begin{align} \notag 
-(\ln\lambda)_s = \frac{\ln s}{2s} -\frac{\ln\ln s}{2s} +O(\frac{1}{s}),
\end{align}
which gives
\begin{align} \notag 
-\ln\lambda = \frac{1}{4}(\ln s)^2 -\frac{1}{2} \ln s\ln\ln s +O(\ln s),
\end{align}
which in turn gives
\begin{align} \notag 
-2\ln(\lambda^2) = (\ln s)^2\left[1-2\frac{\ln\ln s}{\ln s}+ O\left(\frac{1}{\ln s}\right)\right].
\end{align}
This leads to 
\begin{align}
\sqrt{-2\ln(\lambda^2)} &= \ln s\left[1-\frac{\ln\ln s}{\ln s}+O\left(\frac{1}{\ln s}\right)\right] \notag \\
&= \ln s - \ln\ln s+O(1) \label{E:AUX2}
\end{align}
from which
$$
e^{\sqrt{-2\ln(\lambda^2)}} = \frac{s}{\ln s}e^{O(1)}$$
and hence from \eqref{cbwbeibeibe}, \eqref{E:BSOLVED}:
$$-\l \l_t=-\lsl=b+O\left(\frac{1}{s\ln s}\right)=\frac{\ln s}{2s}e^{O(1)}=e^{-\sqrt{2|\ln \l^2|}+O(1)}.$$ This yields the pointwise ode:
$$-e^{\sqrt{2|\ln \l^2|}+O(1)}(\lambda^2)_t =1
$$
which integration in time yields:
\be\label{E:LSQUARE}
\lambda^2(t) = (T-t)e^{-\sqrt{2|\ln(T-t)|}+O(1)}
\ee
and \eqref{E:MELTINGRATE} is proved.\\

\noindent\underline{case $k\ge 1$}: We estimate from \eqref{E:BOOTBOUND}, \eqref{linearcj}:
\be\label{E:AUX3}
-\lsl=a=\frac{k+1}{2ks}-\frac{k+1}{2k^2s\ln s}+O\left(\frac{1}{s(\ln s)^{3/2}}\right)
\ee
and hence there exists 
$c^*=c^*(u_0)$ such that 
$$
-\ln\lambda (s)=\frac{k+1}{2k}\ln s-\frac{k+1}{2k^2}\ln \ln s+c^*+o_{s\to +\infty}(1)
$$ 
or equivalently: 
\begin{align}
\label{tlaw}
\l(s)& =c(u_0)(1+o(1))\frac{(\ln s)^{\frac{k+1}{2k^2}}}{s^{\frac{k+1}{2k}}}, \ \ c(u_0)>0.
\end{align} 
We conclude that
$$
T =\int_0^{+\infty}\l^2(s)ds<+\infty
$$ 
and 
\begin{align*}
T-t&=\int_s^{+\infty}\l^2(\sigma)\,d\sigma 
=\int_s^{+\infty}(c^2+o(1))\frac{(\ln \sigma)^{\frac{k+1}{k^2}}}{\sigma^{\frac{k+1}{k}}}d\sigma
=(kc^2+o(1))\frac{(\ln s)^{\frac{k+1}{k^2}}}{s^{\frac 1k}}.
\end{align*}
This implies 
\be\label{E:AUX4}
\frac{1}s=\frac{(T-t)^k}{|\ln (T-t)|^{\frac{k+1}{k}}}(c+o(1))
\ee
which together with \eqref{tlaw} yields the melting law:
$$
\l(t)=(c^*(u_0)+o_{t\to T}(1))\frac{(T-t)^{\frac{k+1}{2}}}{|\ln (T-t)|^{\frac{k+1}{2k}}},
$$ 
this is \eqref{meltingk}.\\

\noindent{\bf step 2} Non-concentration of the energy.
Pick $R>0$ and a cut-off function 
$$
\chi_R(x)=\chi\left(\frac{x}{R}\right)=\left\{\begin{array}{ll} 0\ \ \mbox{for}\ \ x\leq R\\ 1\ \ \mbox{for}\ \ x\geq 2R.\end{array}\right.
$$ 
Then for $t$ sufficiently close to the melting time $T$:
\bea
\label{locliazationconsr}
\nonumber &&\frac 12\frac{d}{dt}\int \chi_R|\nabla u|^2\, dx=\int\chi_R\nabla u\cdot\nabla\pa_tu\, dx=\int\chi_R\nabla u\cdot\nabla\Delta u \, dx\\
\nonumber &=& -\int \Delta u\left[\chi_R\Delta u+\nabla \chi_R\cdot\nabla u\right] \, dx\\
&=&  -\int \chi_R|\Delta u|^2+\frac12\int |\nabla u|^2r\frac{\pa}{\pa r}\left(\frac{\chi'_R}{r}\right)\, dx
\eea
and hence the uniform bound on the Dirichet energy ensures: 
$$
\forall R>0, \ \ \int_0^T\chi_R|\Delta u|^2dx<+\infty.
$$ 
Hence for all $0<\tau < T-t$, 
\bee
&&\int\chi_R|\nabla u(t+\tau)-\nabla u(t)|^2dx=\int\chi_R\left|\int_t^{t+\tau}(\pa_t\nabla u)(\sigma,x)d\sigma\right|^2dx\\
&\lesssim & \tau \int_0^T\chi_R|\pa_t \nabla u|^2dx\lesssim \tau \int_0^T\chi_R|\nabla \Delta u|^2dx\leq  C(R)\tau,
\eee
where the last estimate follows by integrating-in-$t$ equation~\eqref{E:CHIESTIMATE} with $\chi = \chi_R$ and using~\eqref{locliazationconsr}.
Hence for all $R>0$, $\nabla u(t,x)$ is a Cauchy sequence in $L^2(|x|\geq 2 R)$ as $t\to T$. 
We conclude from a simple diagonal extraction argument that there exists $u^*\in \dot{H}^1(\RR^2)$ such that 
\be
\label{estnione}
\forall R>0, \ \ \nabla u(t)\to \nabla u^*\ \ \mbox{in}\ \ L^2(|x|\geq 2R) \text{ as } \ t\to T.
\ee 
Moreover, the uniform bound on the Dirichlet energy \eqref{decaydirichlet} ensures 
\be
\label{weakvoncc}
\nabla u^*\in L^2, \ \ \nabla u(t)\rightharpoonup \nabla u^*\ \ \mbox{ weakly in}\ \ L^2\ \ \mbox{as}\ \ t\to T.
\ee
Pick now 
\be
\label{deffnenoe}
R(t)=\left\{\begin{array}{ll}  \l(t)B(t), \ \ B^2(t)b(t)=\frac12 |\ln b(t)| \ \ \mbox{for}\ \ k=0\\ 
\l(t)B(t), \ \ B^2(t)a(t) =  \l(t) |\ln a(t)|\ \ \mbox{for}\ \ k\ge1.
\end{array}\right.
\ee 
Note that in both cases we have $B(t)\gg1.$
Then from \eqref{locliazationconsr}: 
$$
\left|\frac{d}{d\tau}\int \chi_{R(t)}|\nabla u(\tau)|^2dx\right|\lesssim \frac{1}{R^2(t)}\int |\nabla u(\tau)|^2 \, dx+\int\chi_{R(t)}|\Delta u(\tau)|^2 \, dx
$$ 
and hence integrating over $[t,T)$ and using  \eqref{estnione}, \eqref{deffnenoe}, \eqref{decaydirichlet}: 
$$
\left|\int  \chi_{R(t)}|\nabla u(\tau)|^2dx- \int \chi_{R(t)}|\nabla u^*|^2dx\right|\lesssim \frac{T-t}{R^2(t)}+\int_t^T\int_{r\ge \l(t)}|\Delta u(\tau)|^2dxd\tau.
$$ 
If $k=0,$ we use~\eqref{E:MELTINGRATE} to estimate
$$
\frac{T-t}{R^2(t)}=\frac{2b(t)(T-t)}{\l^2(t)}\frac{1}{|\ln b(t)|}\longrightarrow 0\ \ \mbox{as} \ \ t\to T.
$$

The above limit holds since by~\eqref{E:LSQUARE}
\[
\frac{2b(t)(T-t)}{\l^2(t)}\frac{1}{|\ln b(t)|}
\lesssim \frac{b(t)e^{\sqrt{2|\ln (T-t)|}}}{|\ln b(t)|} \lesssim \frac{b(t)e^{2\sqrt{|\l(t)|}}}{|\ln b(t)|} \lesssim \frac 1{|\ln b(t)|} \to 0, \ \ \text{ as} \ t\to T,
\]
where the last bound follows from~\eqref{E:AUX1} and~\eqref{E:AUX2}.
If $k\ge1,$ then
$$
\frac{T-t}{R^2(t)}=\frac{a(t)(T-t)}{\l^2(t)}\frac{1}{c|\ln a(t)|}\longrightarrow 0\ \ \mbox{as} \ \ t\to T.
$$
The above limit holds since by~\eqref{E:LSQUARE}
\[
\frac{a(t)(T-t)}{\l^2(t)}\frac{1}{|\ln a(t)|}
\lesssim \frac{\frac{(T-t)^k}{|\ln (T-t)|^{\frac{k+1}{k}}}(T-t)}{\frac{(T-t)^{k+1}}{|\ln (T-t)|^{\frac{k+1}{k}}}|\ln a(t)|} 
\lesssim \frac 1{|\ln a(t)|} \to 0, \ \ \text{ as} \ t\to T,
\]
where we used~\eqref{E:AUX3} and~\eqref{E:AUX4}.
Letting $t\to T$, we conclude: 
 \be
 \label{cenkonekonveonveo}
 \int |\nabla u^*|^2=\lim_{t\to T} \int  \chi_{R(t)}|\nabla u(\tau)|^2dx.
 \ee 
 We now claim that 
 \be
 \label{toberpovfedd}
 \lim_{t\to T} \int_{\{|x|\ge \l(t)\}}  (1-\chi_{R(t)})|\nabla u(\tau)|^2dx=0
 \ee
 which together with \eqref{cenkonekonveonveo}, \eqref{weakvoncc} concludes the proof of \eqref{stronltwoconvergece}. 
 Indeed, in the case $k=0,$ from \eqref{calculeignvectoreta},~\eqref{coercunubis},~\eqref{E:BOOTBOUND},
 and~\eqref{deffnenoe} we obtain:
 \begin{align*}
 & \int_{\{|x|\ge \l(t)\}}  (1-\chi_{R(t)})|\nabla u(t)|^2dx\leq \int_{\l(t)\le |x|\leq 2R(t)}|\nabla u(t)|^2dx \\
 & =\int_{1\le |y|\leq 2B(t)}|\nabla v(t,y)|^2  dy\\
  & \lesssim  e^{2bB^2(t)}\left[\int_{1\le y\leq 2B(t)}|b\pa_y\eta_{b_0}|^2\rho_bydy+\|\pa_y\e\|^2_b\right] \\
   & =  e^{2bB^2(t)}\left[b^2|\ln b|^2 + \frac{b^2}{|\ln b|}\right]\lesssim b|\ln b|^2\to \ \ 0\ \ \mbox{as} \ \ t\to T,
  \end{align*}
and \eqref{toberpovfedd} is proved. A similar algebra holds for $k\ge1$. This concludes the proof of Theorem~\ref{T:MAIN}.


\section{Infinite time freezing regimes}
\label{sectionnonlinbis}

This section is devoted to the existence and stability of the freezing process emerging from strongly localised initial data. 
Throughout the section, we let $$\pm=+, \ \ \rho=\rho_+, \ \ B>0.$$


\subsection{Renormalised equations and initialisation}

We let
\be
\label{normalizationinitialbis}
u(t,x)=v(s,y), \ \ y=\frac{r}{\l(t)}, \ \ \l(0)=1,
\ee 
with the renormalised time 
\be
\label{renormalizedtimebis}
s(t)=s_0+\int_0^{t}\frac{d\tau}{\l^2(\tau)}, \ \ s_0\gg 1,
\ee
and obtain the renormalised equation: 
\be\label{E:KDEFINITIONbis}
\left\{\begin{array}{ll}\pa_sv+\H_A v=0, \ \ A=\lsl,\\ v(s,1)=0, \ \ \pa_yv(s,1)=-A.
\end{array}\right.
\ee
We now prepare our initial data in the following way. We let 
\be
\label{defbbs}
B(s)=\frac 1{2s}, \ \ B_k^e=\frac{1}{s^{k+1}(\log s)^2}
\ee so that with $A^e=B^e$:
\bea
\label{systapproximabis}
&&(B_k^e)_s+B_k^eB\left(2k+2+\frac{2}{\log s}\right)+\frac{2(B-A^e)B_k^e}{\log s}
=O\left(\frac{1}{s^{k+2}(\log s)^3}\right).
\eea
We define
\be\label{E:QBETACOOLING}
Q_\beta (y) : =-\sum_{j=0}^k B_j\etah_{B,j}(y)
\ee
and introduce the dynamical decomposition 
\be
\label{orthoebis}
v(s,y)=Q_{\beta(s)}+\e(s,y), \ \ (\etah_{B,j}(s),\e)_{B}=0, \ \ 0\leq j\leq k.
\ee 
We let again $$
\e_2=\H_B\e, \ \ \calE : = \|\H_B \e\|_B^2,
$$
which due to the orthogonality conditions \eqref{orthoe} is a coercive norm, see Appendix \ref{appendcoerc}. We assume the initial smallness
\be\label{E:ESMALLbis}
 \calE(0)\leq
  \frac{B(B_k^e)^2(0)}{|\ln B(0)|}
  \ee
  and consider the bootstrap bound 
\be
\label{E:BOOTBOUNDbis}
\mathcal E \le\frac{DB(B_k^e)^2}{|\log B|}
\ee
for some large enough universal $D=D(k)$ to be chosen later. Moreover, we assume initially 
  \be
  \label{initializationbzero}
 B_k(s_0)=B_k^e(s_0), \ \ s_0\gg 1
  \ee
  and bootstrap the bound 
  \be
  \label{bootBseroekr}
  |B_k(s)|\leq 10 B_k^e(s).
  \ee
  For $k\ge 1$, we also let
\be
\label{bveuibvibibis}
B_j(s)=\frac{V_j(s)}{s^{k+1}(\log s)^{\frac 52}}, \ \  j=0,\dots, k-1.
\ee 
and assume 
\be
\label{aprioriboundbj} 
\sum_{j=0}^{k-1}\left|V_j(s)\right|^2\leq K^2.
\ee
We define 
$$
s^*=\left\{\begin{array}{ll}\sup_{s\geq s_0}\{\eqref{E:BOOTBOUNDbis}, \eqref{bootBseroekr} \mbox{hold on}\ \ [s_0,s]\}\ \ \mbox{for}\ \ k=0,\\ \sup_{s\geq s_0}\{\eqref{E:BOOTBOUNDbis}, \eqref{bootBseroekr}, \eqref{aprioriboundbj}\ \ \mbox{hold on}\ \ [s_0,s]\}\ \ \mbox{for}\ \ k\ge 1.
\end{array}\right.
$$
The main ingredient of the proof of Theorem~\ref{T:MAINbis} is the following:

\begin{proposition}[Bootstrap estimates on $B$ and $\varepsilon$]
\label{pr:bootstrapbis}
The following statements hold:\\
{\em 1. Stable regime}: for $k=0$, $s^*=+\infty$.\\
{\em 2. Unstable regime}: for $k\geq 1$, there exist constants $K,$ $\delta = \delta(K)\ll1$ and $(V_0(0),\dots,V_{k-2}(0),V_{k-1}(0))$ depending on $\e(0)$ such that  $s^*=+\infty$.
\end{proposition}

From now on and for the rest of this section, we study the flow in the bootstrap regime $s\in[s_0,s^*)$. Note in particular the rough bounds 
\be
\label{roughboundbjbis}
|B_j|\lesssim |B_k^e|, \ \ 0\leq j\leq k\ \ \mbox{and}\ \ \mathcal E\leq \frac{B(B_k^e)^2}{\sqrt{|\ln B|}}
\ee
for $s_0\geq s_0(K)$ large enough.


\subsection{Extraction of the leading order ODE's driving the freezing}


We start with the constraint induced by the boundary conditions:

\begin{lemma}[Boundary conditions]
\label{lemmaindeitnituesbis}
There holds:
\bea
\label{E:PARTIALYEPSILONBOUNDARYbis}
&&A=\sum_{j=0}^{k}B_j\left[1+\frac{2\alpha_j}{|\ln B|}+O\left(\frac{1}{|\ln B|^2}\right)\right]+O\left(\frac{\sqrt{\mathcal E}}{|\ln B|\sqrt{B}}\right),\\
&&\e_2(1)=A(B-A), \label{etwoonebisbis}\\
\label{etwoonebisbisthird}
& & \pa_y\e_2(1)=A_s-A^2(B-A)+\sum_{j=0}^k\lh_{B,j}BB_j\left[1+\frac{2\alpha_j}{|\ln B|}+O\left(\frac{1}{|\ln B|^2}\right)\right]\eea
\end{lemma}

\begin{remark} Note that \eqref{E:PARTIALYEPSILONBOUNDARYbis}, \eqref{roughboundbjbis} imply
\be
\label{roughboundA}
|A|\lesssim B_k^e+\frac{\sqrt{\mathcal E}}{|\ln B|\sqrt{B}}\lesssim B_k^e.
\ee
\end{remark}

\begin{proof}[Proof of Lemma \ref{lemmaindeitnitues}]  We compute from \eqref{beiebibeiebis} \eqref{bejvieivbebie}: 
\be
\label{bejvieivbebiebis}
\pa_{y}\etah_{B,j}(1)=\pa_y\eta_{B,j}(1)e^{-\frac B2}=1+\frac{2\alpha_j}{|\ln B|}+O\left(\frac{1}{|\ln B|^2}\right).
\ee
This implies that
$$
\pa_yQ_\beta(1)=-\sum_{j=0}^{k}B_j\pa_y\eta_{B,j}(1)=-\sum_{j=0}^{k}B_j\left[1+\frac{2\alpha_j}{|\ln B|}+O\left(\frac{1}{|\ln B|^2}\right)\right].
$$
Since $v = Q_\beta + \e,$ it follows that 
\bee
\e_y\big|_{y=1} & = &v_y\big|_{y=1} - \pa_yQ_\beta\big|_{y=1}=  -\lsl- \pa_yQ_\beta(1)\\
& = & -A+ \sum_{j=0}^{k}B_j\left[1+\frac{2\alpha_j}{|\ln B|}+O\left(\frac{1}{|\ln B|^2}\right)\right],
\eee
which together with \eqref{coercunubisbis} yields \eqref{E:PARTIALYEPSILONBOUNDARYbis}. From \eqref{E:KDEFINITIONbis}, $v(s,1)=0$ and $\pa_yv(s,1)=-\lsl=-A$: 
$$
0=\H_Av(1)=(\H_Bv+(B-A)\Lambda v)(1)=\e_2(1)+A(A-B),
$$ this is \eqref{etwoonebisbis}. Now from $\pa_yv(s,1)=-A$, we have 
$$
\pa_s\pa_yv(s,1)=-A_s.
$$ 
On the other hand, taking $\pa_y$ of \eqref{E:KDEFINITIONbis}, we have:
$$
0=\pa_s\pa_yv+\pa_y(\H_Bv+(B-A)\Lambda v)=\pa_s\pa_yv+\pa_y\e_2+\pa_y\H_BQ_\beta+(B-A)y\Delta v.
$$
We evaluate the above identity at $y=1$. From \eqref{E:KDEFINITIONbis} and $\pa_sv(s,1)=0$, $\pa_yv(s,1)=-A$: $$\Delta v(1)=-A\Lambda v(1)=A^2.$$ By construction, 
$$
\pa_y \H_bQ_\beta=-\sum_{j=0}^k\hat{\l}_{B,j}BB_j\pa_y\etah_{B,j},
$$ 
and hence:
$$-A_s+\pa_y\e_2(1)-\sum_{j=0}^k\hat{\l}_{B,j}BB_j\left[1+\frac{2\alpha_j}{|\ln B|}+O\left(\frac{1}{|\log B|^2}\right)\right]+A^2(B-A)=0.
$$
\end{proof}

We now compute the leading order modulation equations.

\begin{proposition}[Leading order modulation equations]
\label{approxsolutionbis}
Under the a priori bounds of Proposition \ref{pr:bootstrapbis}, there holds \be
\label{equationqbetabis}
\pa_sQ_{\beta}+\H_ AQ_\beta={\rm Mod} +\Psi
\ee
where we defined the modulation vector
\bea
\label{equationmodulationbis}
\Mod& : =& -\sum_{j=0}^{k} \left[(B_j)_s+BB_j\lh_{B,j}+\frac{2(B-A)B_j}{|\ln B|}\right]\etah_{B,j},
\eea
and the remaining error satisfies the bound:
\be
\label{esterreurbis}
\|\Psi\|_b+\frac{1}{\sqrt{B}}\|\pa_y\Psi\|_B+\frac{1}{B}\|\mathcal H_B\Psi\|_B\lesssim \frac{\sqrt{B}B^e_k}{|\log B|}.
\ee
\end{proposition}

\begin{proof}[Proof of Proposition \ref{approxsolution}]
Let the deviation 
$$
\Phi : =B_s+2B(B-A).
$$
By definition $$\H_A=\H_B+(B-A)\Lambda $$ and we therefore compute from~\eqref{E:QBETACOOLING}:
\bee
&&-\pa_sQ_{\beta(s)}(y)-\H_AQ_\beta\\
&=&\sum_{j=0}^{k}\left[(B_j)_s\etah_{B,j}+B_s\frac{B_j}{B}B\pa_B\etah_{B,j}+BB_j\hat{\l}_{B,j}\etah_{B,j}+(B-A)B_j\Lambda \etah_{B,j}\right]\\
& = & \sum_{j=0}^{k}\left\{[(B_j)_s+BB_j\lh_{b,j}]\etah_{B,j}+(B-A)B_j[\Lambda \etah_{B,j}-2B\pa_B\etah_{B,j}]+\frac{B_j}{B}B\pa_B\etah_{B,j}\Phi\right\}\\
& = & \sum_{j=0}^{k}\left\{\left[(B_j)_s+BB_j\lh_{B,j}+\frac{2(B-A)B_j}{|\log B|}\right]\etah_{B,j}\right.\\
&& + \left. (B-A)B_j\left(\Lambda \etah_{B,j}-2B\pa_B\etah_{B,j}-\frac2{|\ln B|}\etah_{B,j}\right)+\frac{B_j}{B}\Phi B\pa_B\etah_{B,j}\right\}.\\
\eee 
We now estimate from \eqref{roughboundA}: 
\be
\label{reoughboundphibis}
|\Phi|\lesssim |AB|\lesssim BB_k^e
\ee
and hence using \eqref{roughboundbjbis}, \eqref{estderivativebis}:
$$\|\frac{B_j}{B}\Phi B\pa_B\etah_{B,j}\|_B\lesssim \frac{|\log B|}{\sqrt{B}}(B_k^e)^2\lesssim \frac{\sqrt{B}B^e_k}{|\log B|}$$ and similarly for higher derivatives. Moreover from \eqref{foihonononfebisbis}, \eqref{roughboundbjbis}:
\bee
\left\|(B-A)B_j\left[\Lambda \etah_{B,j}-2B\pa_B\etah_{B,j}-\frac2{|\ln B|}\etah_{B,j}\right]\right\|_B\lesssim \frac{1}{\sqrt{B}|\ln B|}BB^e_k\lesssim \frac{\sqrt{B}B^e_k}{|\log B|}
\eee
and similarly for higher derivatives, and \eqref{esterreurbis} is proved.
\end{proof}


\subsection{Modulation equations}


The relations \eqref{E:KDEFINITIONbis},~\eqref{equationqbetabis} yield
\be
\label{E:EQUATIONbis}
\pa_s\e+\H_A\e=\matchal F,  \ \ 
\mathcal F=-\Mod-\Psi
\ee
and we now compute the exact modulation equations.

\begin{lemma}[Modulation equations for $B_j$]
\label{lemmamodulationbis}
There holds for $0\leq j\leq k$:
\be
\label{pouetbisbis}
\left|(B_j)_s+B_j B\left(2j+2+\frac{4}{|\log B|}\right)\right|\lesssim \frac{BB_k^e}{|\log B|^2}.
\ee  
\end{lemma}

\begin{proof}[Proof of Lemma \ref{lemmamodulationbis}] Let $0\leq j\le k$ and take the scalar product of \eqref{E:EQUATIONbis} with $\etah_{B,j}$ and use the orthogonality condition \eqref{orthoebis} to compute:
$$
-(\e,\pa_s\etah_{B,j})_B-B_s( \e,\frac{|y|^2}2\etah_{B,j})_B =  (\mathcal F,\etah_{B,j})+(A-B) (\Lambda \e,\etah_{B,j}).
$$ 
We integrate by parts using \eqref{orthoebis}:
\bee
&&-(\e,\pa_s\etah_{B,j})_B-B_s( \e,\frac{|y|^2}2\etah_{B,j})_B+(B-A) (\Lambda \e,\etah_{B,j})_B\\
& = & -(\e,\pa_s\etah_{B,j})_B-B_s( \e,\frac{|y|^2}2\etah_{B,j})_B-(B-A)(\e,B|y|^2\etah_{B,j}+\Lambda \etah_{B,j})_B\\
& = & -(\e,\frac12(B_s+2B(B-A))|y|^2\etah_{B,j})_B-(\e,\frac{B_s+2B(B-A)}{B}B\pa_B\etah_{B,j})_B\\
& & + (B-A)(\e,2B\pa_B\etah_{B,j}-\Lambda\etah_{B,j})_B\\
& = & -(\e,\frac{\Phi}{B}[\frac 12B|y|^2\etah_{B,j}+B\pa_B\etah_{B,j}])_B+(B-A)(\e,2B\pa_B\etah_{B,j}-\Lambda\etah_{B,j})_B.
\eee
We now estimate from \eqref{coercunubisbis}, \eqref{cebivbfenoebis}, \eqref{estderivativebis}, \eqref{reoughboundphibis}:
\bee
\left|-(\e,\frac{\Phi}{B}[\frac 12B|y|^2\etah_{B,j}+B\pa_B\etah_{B,j}])_B\right|\lesssim \|\e\|_B\frac{|\Phi|}{B}\frac{|\log B|}{\sqrt{B}}\lesssim\frac{\sqrt{\mathcal E}}{|\log B|\sqrt{B}}
\eee
and using \eqref{foihonononfebisbis} and \eqref{orthoebis}:
$$|(B-A)(\e,2B\pa_B\etah_{B,j}-\Lambda\etah_{B,j})_B|\lesssim \frac{B}{\sqrt{B}}\|\e\|_B\lesssim \frac{\sqrt{\mathcal E}}{|\log B|\sqrt{B}}.$$
We now estimate the $\matchal F$ terms given by \eqref{E:EQUATIONbis}. From \eqref{esterreurbis}:
$$
|(\Psi,\etah_{B,j})_B|\lesssim \|\Psi\|_B\|\etah_{B,j}\|_B\lesssim\frac{\sqrt{B}B^e_k}{|\log B|}\frac{|\log B|}{\sqrt{B}}\lesssim B_k^e.
$$  The collection of above bounds together 
with \eqref{scalirpeoducetabis} and~\eqref{E:EQUATIONbis} yields 
$$
\frac{|(\Mod, \etah_{B,j})_B|}{(\etah_{B,j},\etah_{B,j})_B}\lesssim\frac{B}{|\ln B|^2} \left[ \frac{\sqrt{\mathcal E}}{|\log B|\sqrt{B}}+B_k^e\right]=\frac{BB_k^e}{|\log B|^2}
$$
or equivalently:
\be
\label{estparameteresbis}
\sum_{j=0}^k\left|(B_j)_s+BB_j\hat{\l}_{B,j}+\frac{2(B-A)B_j}{|\log B|}\right|\lesssim \frac{BB_k^e}{|\log B|^2}.
\ee
We conclude from \eqref{estparameteresbis}, \eqref{roughboundA}, \eqref{beiebibeiebis}:
$$\left|(B_j)_s+B_j B\left[2j+2+\frac{4}{|\log B|}+O\left(\frac{1}{|\log B|^2}\right)\right]\right|\lesssim \frac{BB_k^e}{|\log B|^2},$$ this is \eqref{pouetbisbis}. 
\end{proof}


\subsection{Energy bound}\label{S:ENERGYBOUNDbis}


We now derive the energy estimate in the freezing regime

\begin{proposition}[Energy bound for freezing]
\label{propenergybis}
There holds the pointwise control
\be
\label{cnekovnevbnenvoeneo}
\frac 12\frac{d}{ds}\left\{\mathcal E+O\left(\frac{B(B_k^e)^2}{|\log B|}\right)\right\}+B(2k+4+c)\|\e_2\|_B^2\lesssim B\frac{B(B_k^e)^2}{|\log B|}
\ee
for some universal constant $c>0$.
\end{proposition}

\begin{proof}[Proof of Proposition \ref{propenergybis}] We compute the energy identity for $\mathcal E$ and estimate all terms.\\

\noindent{\bf step 1} Algebraic energy identity. Recall $\e_2=\mathcal H_{B}\e$. We compute from \eqref{E:EQUATIONbis}:
$$
\pa_s\e_2+\H_A\e_2=[\pa_s,\H_{B}]\e+[\H_A,\H_{B}]\e+\H_{B}\mathcal F.
$$ 
and hence using \eqref{commutatorlaplace}:
\begin{align*}
& [\pa_s,\H_{B}]\e+[\H_A,\H_{B}]\e =-B_s\Lambda \e+[\H_B+(B-A)\Lambda,\H_B]\e \\
& =-B_s\Lambda \e+(B-A)[\Lambda,-\Delta]
 =  -B_s\Lambda \e+2(B-A)\Delta \e \\
 &=-(B_s+2B(B-A))\Lambda \e-2(B-A)[-\Delta\e-B\Lambda\e]
 =  -\Phi\Lambda \e-2(B-A)\e_2.
\end{align*}
Hence the $\e_2$ equation: 
\be
\label{eqationaetwobis}
\pa_s\e_2+\H_A\e_2=-\Phi\Lambda \e-2(B-A)\e_2+\H_B\mathcal F.
\ee
We now compute the modified energy identity:
\bee
\nonumber \frac 12\frac{d}{ds}\mathcal E& = & \frac 12\frac{d}{ds}\int_{y\geq 1} \e_2^2e^{\frac{By^2}{2}}ydy= \frac{B_s}{4}\int_{y\geq 1} y^2|\e_2|^2e^{\frac{B|y|^2}{2}}ydy+(\pa_s\e_2,\e_2)_B\\
& = &  \frac{B_s}{4}\|y\e_2\|_B^2+(-\Phi\Lambda \e-2(B-A)\e_2+\H_B\mathcal F-\H_A\e_2,\e_2)_B.
\eee
We carefully integrate by parts to compute:
\bee
&&-\int_{y\geq 1}\e_2\H_A\e_2e^{\frac{By^2}{2}}ydy=-\int_{y\geq 1}\e_2[\H_B\e_2+(B-A)\Lambda \e_2]e^{\frac{By^2}{2}}ydy\\
& =& \int_{y\geq 1}\pa_y(\rho_B y\pa_y\e_2)\e_2dy-(B-A)\int_{y\ge1} \e_2y\pa_y\e_2e^{\frac{By^2}{2}}ydy\\
& = & -\rho_B(1)\e_2(1)\pa_y\e_2(1)-\int_{y\geq 1}|\pa_y\e_2|^2e^{\frac{By^2}{2}}ydy\\
&&  +\frac{B-A}{2}\rho_B(1)\e_2^2(1)+\frac{B-A}{2}\int_{y\geq 1}\e_2^2\left[2+By^2\right]e^{\frac{By^2}{2}}ydy\\
& = &  -\|\pa_y\e_2\|_B^2+(B-A)\|\e_2\|_2^2+\frac{B(B-A)}{2}\|y\e_2\|_B^2-\rho_B(1)\e_2(1)\left[\pa_y\e_2-\frac{B-A}{2}\e_2\right](1).
\eee
This yields the algebraic energy identity:
\bea
\label{algebraicneergyidentitybis}
\nonumber \frac 12\frac{d}{ds}\mathcal E&=& -\|\pa_y\e_2\|_B^2-(B-A)\|\e_2\|_2^2+\frac{\Phi}{4}\|y\e_2\|_B^2-\rho_B\e_2\left[\pa_y\e_2-\frac{B-A}{2}\e_2\right](1)\\
& - & \Phi(\Lambda \e,\e_2)_B+(\H_B\mathcal F,\e_2)_B.
\eea
We now estimate all  terms in the right-hand side of \eqref{algebraicneergyidentitybis}.\\

\noindent{\bf step 2} Nonlinear estimates. From \eqref{coercunubisbisbis}, \eqref{reoughboundphibis}:
\bee
|\Phi||(\Lambda\e,\e_2)_B|\lesssim BB_k^e\frac{\mathcal E}{B}\lesssim B\frac{\mathcal E}{|\ln B|}.
\eee
Moreover from \eqref{orthoebis}, \eqref{equationmodulationbis}, $(\H_B\Mod,\e_2)_b=0,$ and from \eqref{esterreurbis}, \eqref{roughboundbjbis}:
$$|(\H_B\Psi,\e_2)_B\lesssim B\sqrt{\mathcal E}\frac{\sqrt{B}B^e_k}{|\log B|}\lesssim B\frac{B(B^e_k)^2}{|\log B|}.$$
We now estimate from \eqref{E:YHBEPSILONBOUNDbis}, \eqref{reoughboundphibis}, \eqref{etwoonebisbis}, \eqref{roughboundA}:
\bee
|\Phi\|y\e_2\|_B^2&\lesssim& BB_k^e\left[\frac{\|\pa_y\e_2\|^2_B}{B^2}+\frac{\e^2_2(1)}{B}\right]\lesssim \frac{\|\pa_y\e_2\|^2_B}{|\log B|}+\frac{B^3A^2}{|\log B|}\\
& \lesssim & \frac{\|\pa_y\e_2\|^2_B}{|\log B|}+\frac{B^3}{|\log B|}\left[(B_k^e)^2+\frac{\mathcal E}{B|\ln B|^2}\right]\\
& \lesssim & \frac{\|\pa_y\e_2\|^2_B}{|\log B|}+B\frac{B(B_k^e)^2}{|\log B|^2}.
\eee

\noindent{\bf step 3} Boundary term and conclusion. It now remains to treat the boundary term in \eqref{algebraicneergyidentitybis}. First from \eqref{etwoonebisbis}, \eqref{etwoonebisbisthird}, \eqref{roughboundA}:
$$|B-A||\e_2(1)|^2\lesssim |B|(AB)^2\lesssim B^3(B_k^e)^2.$$ Let $$\tilde{A}=\sum_{j=0}^kB_j\left(1+\frac{2\alpha_j}{|\log B|}\right),$$ and observe the bounds from \eqref{E:PARTIALYEPSILONBOUNDARYbis}, \eqref{roughboundbjbis}, \eqref{pouetbisbis}:
\be
\label{estatilde}
|\tilde{A}-A|\lesssim \frac{B^e_k}{|\log B|}, \ \ |\tilde{A}_s|\lesssim BB^e_k
\ee
We now rewrite \eqref{etwoonebisbisthird} using \eqref{pouetbisbis}, \eqref{roughboundA}, \eqref{estparameteresbis}:
\bee
\pa_y\e_2(1)&=&\left(A-\tilde{A}\right)_s+\sum_{j=0}^k\left(1+\frac{2\alpha_j}{|\log B|}\right)\left[(B_j)_s+\hat{\l}_{B,j}BB_j\right]+O\left(\frac{BB_k^e}{|\log B|^2}\right)\\
& = & (A-\tilde{A})_s+O\left(\frac{BB_k^e}{|\log B|}\right)
\eee
from which:
\bee
&&\rho_B(1)\e_2(1)\left[\pa_y\e_2-\frac{B-A}{2}\e_2\right](1)\\
& = & \rho_B(1)A(B-A)\left[(A-\tilde{A})_s+O\left(\frac{BB_k^e}{|\log B|}+ B^2B_k^e\right)\right]\\
& = & O\left(B\left(\frac{B(B^e_k)^2}{|\log B|}\right)\right)+ \rho_B(1)A(B-A)(A-\tilde{A})_s.
\eee 
We now compute using \eqref{estatilde}:
\bee
&&-\rho_B(1)A(B-A)(A-\tilde{A})_s=  \rho_B(1)(A-\tilde{A})_s\left[(A-\tilde{A})^2+(A-\tilde{A})(\tilde{A}-B)-\tilde A(B-A)\right]\\
&&=  -\frac{d}{ds}\left\{\rho_B(1)\left[\frac{(A-\tilde{A})^3}{3}+\frac{(A-\tilde{A})^2(\tilde{A}-B)}{2}-(A-\tilde{A}){red}\tilde A(B-A)\right]\right\}\\
&& - \frac{B_s}{2}\rho_B(1)\left[\frac{(A-\tilde{A})^3}{3}+\frac{(A-\tilde{A})^2(\tilde{A}-B)}{2}-(A-\tilde{A})\tilde A(B-A)\right]\\
&& + \rho_B(1)\left[\frac{(A-\tilde{A})^2}{2}(\tilde{A}_s-B_s)-(A-\tilde{A})(\tilde{A}_s(B-A)+(B_s-A_s)\tilde{A})\right]\\
& = & \frac{d}{ds}\left\{O\left(\frac{B(B_e^k)^2}{|\log B|^2}\right)\right\}+O\left(\frac{B^2(B_k^e)^2}{|\log B|}\right)
\eee
Injecting the collection of above bounds into \eqref{algebraicneergyidentitybis} yields:
\bee
\label{algebraicneergyidentitybisbis}
&& \frac 12\frac{d}{ds}\left\{\mathcal E+O\left(\frac{B(B_k^2)^2}{|\log B|}\right)\right\}\\
& = & -\|\pa_y\e_2\|_B^2-(B-A)\|\e_2\|_B^2+O\left(\frac{\|\pa_y\e_2\|^2_B}{|\log B|}+B\frac{B(B_k^e)^2}{|\log B|}\right)
\eee
and hence using the coercivity \eqref{coercdeuxbisbis},~\eqref{etwoonebisbis}, \eqref{defbbs} and~\eqref{roughboundA}:
$$\frac 12\frac{d}{ds}\left\{\mathcal E+O\left(\frac{B(B_k^2)^2}{|\log B|}\right)\right\}+B\left[2k+5+O\left(\frac{1}{|\log B|}\right)\right]\|\e_2\|_B^2\lesssim \frac{B(B_k^e)^2}{|\log B|},$$ and \eqref{cnekovnevbnenvoeneo} is proved.
\end{proof}


\subsection{Proof of Proposition \ref{pr:bootstrapbis} and Theorem \ref{T:MAINbis}}


We may now close the bootstrap estimates of Proposition \ref{pr:bootstrapbis}.\\

\noindent{\bf step 1} Closing the energy bound.  First observe
$\left|\lsl\right|=|A|\lesssim B$ and thus $|\ln \l(s)|\lesssim C(s)$ implies $\l(s)>0$ on $[0,s^*]$. The control of the $\matchal E$ norm easily implies the $H^2$ control $\|u(s)\|_{H^2}\lesssim C(s)$ 
and hence the solution is well defined from the point of view of the $H^2$ Cauchy theory on $[0,s^*]$. 
We now integrate in time the bound \eqref{cnekovnevbnenvoeneo} and obtain for some $c>0$ 
\bee
\matchal E(s)&\le& \left(\frac{s_0}{s}\right)^{2k+3+c}\mathcal E(0)+ C\frac{B(s)(B_k^e(s))^2}{|\log s|}+C \frac{1}{s^{2k+3+c}}\int_{s_0}^s\frac{B^2(B_k^e)^2}{|\log B|}\sigma^{2k+3+c}\,d\sigma\\
& \le & \left(\frac{s_0}{s}\right)^{2k+3+c}\mathcal E(0)+C \frac{B(s)(B_k^e(s))^2}{|\log s|}\lesssim \frac{B(s)(B_k^e(s))^2}{|\log s|}.
\eee
Hereby we used the  explicit formulas \eqref{defbbs} to infer that $ \frac{1}{s^{2k+3+c}}\int_0^s\frac{B^2(B_k^e)^2}{|\log B|}\sigma^{2k+3+c}\,d\sigma \lesssim \frac{B(s)(B_k^e(s))^2}{|\log s|} $ 
and in the last inequality the initial data assumption \eqref{E:ESMALLbis}. This closes the energy bound \eqref{E:BOOTBOUNDbis}.\\

\noindent{\bf step 2} Control of $B_j$. We  estimate from \eqref{pouetbisbis}:
$$
\left|(B_k)_s+BB_k\left(2k+2+\frac{4}{|\log B|}\right)\right|\lesssim \frac{BB^e_k}{|\log B|^2}
$$ 
from which 
$$
\left|\frac{d}{ds}(s^{k+1}(\log s)^2B_k)\right|\lesssim \frac{1}{s(\log s)^2}.
$$ 
An integration-in-time yields:
\be
\label{controbceovbov}
B_k(s)= \frac{1}{s^{k+1}(\log s)^2}\left[\frac{B_k(s_0)}{B_k^e(s_0)}+O\left(\frac{1}{\log s_0}\right)\right]
\ee
and the initial data assumption \eqref{initializationbzero} now improves \eqref{bootBseroekr}.\\
For $k\ge 1$, we now argue by contradiction and assume that for all $(V_j)_{0\leq j\leq k-1}$ with $\sum_{j=0}^{k-1}|V_j(0)|^2\leq K^2$, there holds 
$s^*<+\infty$ i.e. 
$$
\sum_{j=0}^{k-1}|V_j(s^*)|^2=K^2.
$$ 
We estimate using the variables \eqref{bveuibvibibis} and \eqref{pouetbisbis}:
$$
\left|(V_j)_s-\frac{k-j}{s}V_j\right|\lesssim \frac{1}{s\sqrt{\log s}}, \ \ j = 1,\dots, k-1.
$$ 
Hence at the exit time:
$$
\frac12\frac{d}{ds}\left[\sum_{j=0}^{k-1}|V_j|^2\right](s^*)\gtrsim \frac1{s^*}\sum_{j=0}^{k-1}|V_j(s^*)|^2>0
$$
and a contradiction follows as in the melting case using Brouwer fixed point theorem. This concludes the proof of Proposition \ref{pr:bootstrapbis}.


\subsection{Proof of Theorem~\ref{T:MAINbis}}\label{S:PROOFMAINTHEOREMbis}


The proof of Theorem~\ref{T:MAINbis} now follows from a simple time integration of the modulation equations.\\
We estimate from \eqref{pouetbisbis} 
$$
\left|(B_k)_s+\frac{B_k}s\left(k+1+\frac{2}{\log s}\right)\right|\lesssim \frac{B_k^e}{s(\log s)^2}
$$ 
and hence 
$$
\left|\frac{d}{ds}(s^{k+1}(\log s)^2B_k)\right|\lesssim \frac{1}{s(\log s)^2}
$$ 
which implies for $s$ large enough: 
$$
B_k(s)=\frac{c(u_0)(1+o(1))}{s^{k+1}(\log s)^2}
$$ 
for some universal constant $c=c(u_0)$. 
We conclude from \eqref{E:PARTIALYEPSILONBOUNDARYbis}: 
$$
\lsl=A= \frac{c(u_0)(1+o(1))}{s^{k+1}(\log s)^2}
$$ 
from which there exists $\l_\infty\geq \l_\infty(u_0)>0$ with 
$$
\l_\infty - \l(s)=\left\{\begin{array}{ll} \frac{c(u_0)(1+o_{s\to +\infty}(1))}{\log s} \ \ \text{ if } \ k=0,\\
\frac{c(u_0)(1+o_{s\to +\infty}(1))}{s^{k}(\log s)^2} \ \ \text{ if } \ k\ge1.
\end{array}\right.
$$ 
Since for large $s\gg1$ we have $\frac{ds}{dt} \sim \frac1{\l_{\infty}^2},$~\eqref{E:MELTINGRATEbis} and~\eqref{E:MELTINGRATEbisk} follow. Finally, recalling that $\hat \eta_{B,j} = e^{-\frac{By^2}2}\eta_{B,j}$ we estimate:
\bee
\int_{|y|\geq 1} |\nabla (\etah_{B,j})|^2&=& \int_{y\geq 1}\left|\pa_y\eta_{B,j}-By\eta_{B,j}\right|^2 e^{-B y^2}y\,dy =\int_{z\geq \sqrt{B}}\left|\pa_z\psi_{B,j}-z\psi_{B,j}\right|^2e^{-z^2}z\,dz\\
& = & |\ln B|^2[\frac 14+o(1)] \int_{z\ge 0} |P_j' - zP_j|^2e^{-z^2} z\,dz=  c_k\left[1+o(1)\right]|\log B|^2
\eee
for some universal constant $c_k>0$. Note that we used~\eqref{deftk}. Hence using \eqref{coercunubisbis}:
$$
\int\left|\nabla(v-B_k\etah_{B,k})\right|^2\lesssim |\log B|^{2}\sum_{j=0}^{k-1}B_j^2+\|\nabla \e\|_{L^2}^2\lesssim (B_k^e)^2.
$$
Therefore, since the self-similar rescaling preserves the Dirichlet energy, it follows that 
\[
\int_{\Omega(t)}|\nabla u|^2  = \int_{|y|\ge 1} |\nabla v|^2 = B_k^2\int_{|y|\ge 1}|\nabla\etah_{B,k}|^2 + O((B_k^e)^2)
\]
which yields~\eqref{E:MELTINGRATEbis2}. To prove~\eqref{E:FORMULA}, note that by 
integrating~\eqref{E:STEFAN} and using the Stokes theorem, we arrive at the following conservation law:
\be\label{E:CONSERVATIONLAW}
\frac{d}{dt} \left(\int_{\Omega(t)} u (t,x)\,dx - \pi \l^2(t)\right) = 0,
\ee
which holds as long as $u\in L^1(\Omega(t)).$ To see this and evaluate $\lim_{t\to\infty}\|u\|_{L^1(\Omega(t))}$ 
first observe that  $v$ satisfies
\begin{align*}
\|v\|_{L^1(\Omega)}& = \int_{y\ge 1}|v|\, ydy 
 \le \|v\|_B\left(\int_{y\ge1}e^{-\frac{By^2}{2}}\,ydy\right)^{1/2}\lesssim\frac 1{\sqrt B}\|v\|_B.
 \end{align*}
On the other hand by~\eqref{orthoebis} it follows that
\begin{align*}
\|v\|_B & \le \sum_{j=0}^k|B_j|\|\eta_{B,j}\|_B + \|\e\|_B \lesssim \frac{B}{|\log B|^2} \frac{|\log B|}{\sqrt B} + \frac{1}{B} \|\H_B\e\|_B,
 \lesssim \frac{\sqrt B}{|\log B|}
\end{align*}
where we used~\eqref{defbbs},~\eqref{scalirpeoducetabis},~\eqref{bveuibvibibis},~\eqref{coercunubisbisbis}, and~\eqref{E:BOOTBOUNDbis}.
The two previous inequalities  lead to 
$$
\|v\|_{L^1(\Omega)} \lesssim \frac 1{|\log B|} \rightarrow 0  \ \ \text{ as $s\to\infty.$}
$$
Therefore 
$$
\|u\|_{L^1(\Omega(t))} = \l^2(t) \|v\|_{L^1(\Omega)} \to 0 \ \ \text{ as $t\to\infty$}
$$
since $|\l(t)|$ remains bounded.
It follows in particular that the conservation law~\eqref{E:CONSERVATIONLAW} holds and formula~\eqref{E:FORMULA} follows.\\
This concludes the proof of Theorem~\ref{T:MAINbis}.


\begin{appendix}



\section{Coercivity estimates in the melting case}
\label{appendcoerc}

This appendix is devoted to the derivation of various coercivity estimates in the melting regime $$\pm=-, \ \ \rho=\rho_-, \ \ b>0$$ which are used along the proof. We start with the standard compactness of the harmonic oscillator.

\begin{lemma}[Weighted $L^2$ estimate]\label{L:HARMONICOSCILLATOR}
Let $u,\pa_z u\in L^2_{\rho}(\RR^2).$ 
Then $\forall k\geq 0$, 
 \be
\label{weightedesimate}
\int z^{2k}u^2z\rho dz\lesssim_k\int (\pa_zu)^2\rho zdz+\int u^2z\rho dz.\ee
\end{lemma}

\begin{proof}[Proof of Lemma \ref{L:HARMONICOSCILLATOR}] Indeed, we use $\pa_z\rho=-z\rho$ and integrate by parts to compute:
\bee
&&\int(\pa_zu-\delta z^ku)^2\rho zdz=  \int (\pa_zu)^2\rho zdz+\delta^2\int z^{2k}u^2\rho zdz-2\delta \int z^{k+1}u\pa_z u\rho z dz\\
& = & \int (\pa_zu)^2\rho zdz+\delta^2\int z^{2k}u^2\rho zdz-\delta[z^{k+1}\rho u^2]_0^{+\infty}+\delta\int u^2[(k+1)^{k-1}z\rho -z^{2k}\rho]dz
\eee
and hence for $0<\delta=\delta(k)\ll1$ small enough:
$$\int z^{2k}u^2\rho zdz\lesssim_k  \int (\pa_zu)^2\rho zdz+\int u^2\rho zdz+\int u^2z^{2k-2}\rho zdz$$ and \eqref{weightedesimate} follows by induction on $k\geq 1$.
\end{proof}

 We now claim the main coercivity property at the heart of the energy estimate.
 \begin{lemma}[Coercivity of $H_b$]
 \label{lemmacoercun}
  Let $k\in \Bbb N$ and $0<b<b^*(k)$ small enough. Let $u\in H^3_{\rho}(r\geq \sqrt{b})$ satisfy 
 
  $$
  u(\sqrt{b})=0, \ \ \la u,\psi_{b,j}\ra_b=0, \ \ 0\le j\le k.
  $$ 
  Then the following inequality holds: 
\bea
\label{coercunu}
\|H_b u\|_{L^2_{\rho,b}}^2&\gtrsim &\|\Delta u\|_{L^2_{\rho,b}}^2+\|(1+z)\pa_{z}u\|_{L^2_{\rho,b}}^2\\
\nonumber &+& \|(1+z)u\|_{L^2_{\rho,b}}^2+b|\ln b|^2(\pa_zu)^2(\sqrt{b}).
\eea
Moreover, 
there exists a constant $c_k>0$ such that
\be
\label{coercdeux}
\|\pa_z H_bu\|_{L^2_{\rho,b}}^2\geq\left[2k+2+O\left(\frac{1}{|\ln b|}\right)\right] \|H_b u\|_{L^2_{\rho,b}}^2-c_kb^2|H_bu(\sqrt{b})|^2
\ee
and 
\be
\label{coiehenoenevo}
\|zH_bu\|_{L^2_{\rho,b}}^2\lesssim  \|\pa_zH_bu\|_{L^2_{\rho,b}}^2+b|H_bu(\sqrt{b})|^2.
\ee
\end{lemma}

\begin{proof}[Proof of Lemma \ref{lemmacoercun}] This lemma is a simple perturbative consequence of the harmonic oscillator estimate \eqref{estimationcoreor}, \eqref{spcetralgap}, and a careful integration by parts to track the boundary term in \eqref{coercunu}.\\

\noindent{\bf step 1} Proof of \eqref{coercunu}. Pick a small constant $\delta>0$, then from $u(\sqrt{b})=0$, we may integrate by parts and compute:
$$\|(H_b-\delta) u\|_{L^2_{\rho,b}}^2=\la (H_b-\delta)u,(H_b-\delta)u\ra_b=\|H_b u\|^2_{L^2_{\rho,b}}-2\delta\la H_bu,u\ra_b+\delta^2 \|u\|_{L^2_{\rho,b}}^2.$$ We may now use the spectral gap bound~\eqref{spcetralgap} with \eqref{weightedesimate}
and conclude that for $\delta$ small enough:
\be\label{E:INTER}
\|H_b u\|^2_{L^2_{\rho,b}}\gtrsim  \|\pa_zu\|_{L^2{\rho,b}}^2+\|(1+z)u\|_{L^2_{\rho,b}}^2+\|(H_b-\delta) u\|_{L^2_{\rho,b}}^2.
\ee
We integrate by parts using the general formula
\be
\label{generalformula}
\la \pa_zu,v\ra_b=-\sqrt{b}u(\sqrt{b})v(\sqrt{b})\rho(\sqrt{b})-\la u,\pa_z v\ra_b-\la u,\frac{v}{z}\ra_b+\la u,zv\ra_b
\ee
to compute:
\begin{align}
\|H_b u\|_{L^2_{\rho,b}}^2 &  = \|\Delta u\|_{L^2_{\rho,b}}^2 - 2\la u_{zz}+\frac{\pa_z u}{z},z\pa_z u\ra_b + \|\Lambda u\|_{L^2_{\rho,b}}^2 \notag \\
&= \|\Delta u\|_{L^2_{\rho,b}}^2  -  \la\pa_z(\pa_zu)^2,z\ra_b - 2\|\pa_zu\|_{L^2_{\rho,b}}^2 +   \|\Lambda u\|_{L^2_{\rho,b}}^2\notag \\
& = \|\Delta u\|_{L^2_{\rho,b}}^2+ b(\pa_zu)^2(\sqrt{b})\rho(\sqrt{b}) +\la (\pa_zu)^2,1\ra_b + \la(\pa_zu)^2,1\ra_b \notag \\
& \ \ \ - \|\Lambda u\|_{L^2_{\rho,b}}^2-2\|\pa_zu\|_{L^2_{\rho,b}}^2 +   \|\Lambda u\|_{L^2_{\rho,b}}^2 \notag \\
& = \|\Delta u\|_{L^2_{\rho,b}}^2 +b(\pa_zu)^2(\sqrt{b})\rho(\sqrt{b}) \label{E:INTER2}.
\end{align}
On the other hand, integrating by parts:
$$
\la H_bu,\ln z\ra_b=\la u,1\ra_b+\sqrt{b}\pa_zu(\sqrt{b})\ln(\sqrt{b})
$$
and thus: 
$$
b|\pa_zu(\sqrt{b})|^2\lesssim \frac{1}{|\ln b|^2}\left[\|H_bu\|_{L^2_{\rho,b}}^2+\|u\|_{L^2_{\rho,b}}^2\right] \lesssim \frac{1}{|\ln b|^2} \|H_b u\|_{L^2_{\rho,b}}^2,
$$
where we used~\eqref{E:INTER}. 
Together with~\eqref{E:INTER2},~\eqref{E:INTER}, claim~\eqref{coercunu} follows.\\

\noindent{\bf step 2} Proof of \eqref{coercdeux}, \eqref{coiehenoenevo}. Define the radially symmetric function 
\be
\label{defguibgeogeog}
v(z)=\left\{\begin{array}{ll} H_bu(\sqrt{b})\ \ \mbox{for}\ \ 0\leq z\leq \sqrt{b}\\ H_bu(z)\ \ \mbox{for}\ \ z\geq \sqrt{b}\end{array}\right.
\ee 
and note that $v\in H^1_{\rho,0}.$
Consider 
$$
w: =v-\sum_{j=0}^{k}\frac{\la v,P_j\ra_0}{\la P_j,P_j\ra_0}P_j.
$$ 
Then from~\eqref{estimationcoreor}:
\be
\label{estzero}
\int_{z\geq 0}|\pa_z w|^2e^{-\frac{z^2}{2}}zdz\geq (2k+2)\int_{z\geq 0}|w|^2e^{-\frac{z^2}{2}}zdz
\geq (2k+2)\int_{z\geq \sqrt{b}}|w|^2e^{-\frac{z^2}{2}}zdz.
\ee
On the other hand, from \eqref{deftk}: $$\left\|P_k-\frac{2}{|\ln b|}\psi_{b,k}\right\|_{H^2_{\rho,b}}\lesssim \frac{1}{|\ln b|}$$ from which for $0\leq j\leq k$:
\bee
&&\la v,P_j\ra_0=H_bu(\sqrt{b})\int_{0\leq z\leq \sqrt{b}}P_je^{-\frac{z^2}{2}}zdz+\la H_bu,\frac{2}{|\ln b|}\psi_{b,j}\ra_b+O\left(\frac{\|H_bu\|_{L^2_{\rho,b}}}{|\ln b|}\right)\\
& = &
\frac 2{|\ln b|}
\la u,\l_{b,j}\psi_{b,j}\ra_b +O\left(b|H_bu(\sqrt{b})|+\frac{\|H_bu\|_{L^2_{\rho,b}}}{|\ln b|}\right)=  O\left(b|H_bu(\sqrt{b})|+\frac{\|H_bu\|_{L^2_{\rho,b}}}{|\ln b|}\right),
\eee
where we used the orthogonality $\la u,\psi_{b,j}\ra_b=0,$ $0\le j \le k.$
Therefore 
$$
\|v-w\|_{H^1_{\rho,b}}\lesssim b|H_bu(\sqrt{b})|+\frac{\|v\|_{L^2_{\rho,b}}}{|\ln b|}.
$$ 
Injecting this into \eqref{estzero} yields \eqref{coercdeux}. We now apply \eqref{weightedesimate} to $v$ and conclude from \eqref{coercdeux}, \eqref{defguibgeogeog}:
\bee
\|zH_bu\|_{L^2_{\rho,b}}^2&\lesssim &\|zv\|_{L^2_{\rho,0}}^2 + b^2|H_bu(\sqrt b)|^2
\lesssim \|\pa_zv\|_{L^2_{\rho,0}}^2+\|v\|_{L^2_{\rho,0}}^2+ b^2|H_bu(\sqrt b)|^2\\
&\lesssim &\|\pa_zH_bu\|_{L^2_{\rho,b}}^2+\|H_bu\|_{L^2_{\rho,b}}^2+b|H_bu(\sqrt{b})|^2
 \lesssim  \|\pa_zHu\|_{L^2_{\rho,b}}^2+b|H_bu(\sqrt{b})|^2
\eee
and \eqref{coiehenoenevo} is proved.
\end{proof}

We now renormalise Lemma \ref{lemmacoercun} by letting $\e(y)=u(\sqrt{b}y)$ which yields exactly:

 \begin{lemma}[Coercivity of $\H_b$]
 \label{lemmacoercunbis}
  Let $k\in \Bbb N$ and $0<b<b^*(k)$ small enough. If $\e\in H^3_{\rho_b}(y\geq 1)$ satisfies 
  $$
  \e(1)=0, \ \ (\e,\eta_{b,j})_b=0\ \ \mbox{for}\ \ 0\leq j\leq k,
  $$ 
  then 
\bea
\label{coercunubis}
\|\H_b \e\|_b^2&\gtrsim &\|\Delta \e\|_{b}^2+b\|\pa_y\e\|_{b}^2+b^2\|\Lambda \e\|_{b}^2\\
\nonumber &+& b^2\|(1+\sqrt{b}y)\e\|_{b}^2+b|\ln b|^2(\pa_y\e)^2(1).
\eea
Moreover,
\be
\label{coercdeuxbis}
\|\pa_y \H_b\e\|_{L^2_{\rho,b}}^2\geq \left[2k+2+O\left(\frac{1}{|\ln b|}\right)\right] b\|\H_b \e\|_{b}^2-c_kb^2|\H_b\e(1)|^2
\ee
and 
\be
\label{E:YHBEPSILONBOUND}
b\|y\H_b\e\|_{b}^2\lesssim  \frac{1}{b}\|\pa_y\H_b\e\|_{b}^2+|\H_b\e(1)|^2.
\ee
\end{lemma}


\section{Coercivity estimates in the freezing case}
\label{appendcoercbis}

We now consider the freezing regime $$\pm=+, \ \ \rho=\rho_+, \ \ B>0$$ and the operator $$H_B=-\Delta -\Lambda, \ \ u(\sqrt{B})=0.$$ We claim the analogue of Lemma \ref{lemmacoercun}:

 \begin{lemma}[Coercivity of $H_B$]
 \label{lemmacoercunbisbis}
  Let $k\in \Bbb N$ and $0<B<B^*(k)$ small enough. Let 
  \be
  \label{defpsibhat}
  \psih_{B,j}=\psi_{B,j}e^{-\frac{B|y|^2}{2}}
  \ee
  and $u\in H^3_{\rho}(r\geq \sqrt{B})$ satisfy 
 $$
  u(\sqrt{B})=0, \ \ \la u,\psih_{B,j}\ra_B=0, \ \ 0\le j\le k,
  $$ 
  then the following inequality holds: 
\bea
\label{coercunubisbis}
\|H_B u\|_{L^2_{\rho,B}}^2&\gtrsim &\|\Delta u\|_{L^2_{\rho,B}}^2+\|(1+z)\pa_{z}u\|_{L^2_{\rho,B}}^2\\
\nonumber &+& \|(1+z)u\|_{L^2_{\rho,B}}^2+B|\ln B|^2(\pa_zu)^2(\sqrt{B}).
\eea
Moreover, 
there exists a constant $c_k>0$ such that
\be
\label{coercdeuxtris}
\|\pa_z H_bu\|_{L^2_{\rho,B}}^2\geq\left[2k+4+O\left(\frac{1}{|\ln b|}\right)\right] \|H_B u\|_{L^2_{\rho,B}}^2-c_kB^2|H_bu(\sqrt{B})|^2
\ee
and 
\be
\label{coiehenoenevobis}
\|zH_Bu\|_{L^2_{\rho,B}}^2\lesssim  \|\pa_zH_Bu\|_{L^2_{\rho,B}}^2+B|H_Bu(\sqrt{B})|^2.
\ee
\end{lemma}

\begin{proof} We follow the proof of Lemma \eqref{lemmacoercunbis}.\\

\noindent{\bf step 1} Proof of \eqref{coercunu}. Pick a small constant $\delta>0$, then from $u(\sqrt{B})=0$, we may integrate by parts and compute:
$$\|(H_B-\delta) u\|_{L^2_{\rho,b}}^2=\la (H_B-\delta)u,(H_B-\delta)u\ra_B=\|H_B u\|^2_{L^2_{\rho,B}}-2\delta\la H_Bu,u\ra_B+\delta^2 \|u\|_{L^2_{\rho,B}}^2.$$ We now use the isometry \eqref{isometryhb} and \eqref{comptuationintergation}. We first obtain from  \eqref{spcetralgap} the spectral gap:  
$$
\forall u\ \ \mbox{with}\ \ \la u,\psi_{B,j}\ra_B=0, \ \ 0\leq  j\leq k,$$
then
\be
\label{spectralgapbis}
\la H_Bu,u\ra_B\geq \left[2k+4+O\left(\frac{1}{|\log B|}\right)\right] \|u\|^2_{L^2_{\rho,B}},
\ee
and similarly from \eqref{weightedesimate}:
\be
\label{weightedesimatebis}
\int z^2u^2z\rho dz\lesssim \int (\pa_zu)^2\rho zdz+\int u^2z\rho dz.
\ee
We therefore conclude that for $\delta$ small enough:
\be\label{E:INTERbis}
\|H_B u\|^2_{L^2_{\rho,B}}\gtrsim  \|\pa_zu\|_{L^2{\rho,B}}^2+\|(1+z)u\|_{L^2_{\rho,B}}^2.
\ee
We integrate by parts using the general formula
\be
\label{generalformulabis}
\la \pa_zu,v\ra_B=-\sqrt{B}u(\sqrt{b})v(\sqrt{B})\rho(\sqrt{B})-\la u,\pa_z v\ra_B-\la u,\frac{v}{z}\ra_B-\la u,zv\ra_B
\ee
to compute:
\begin{align}
\|H_B u\|_{L^2_{\rho,B}}^2 &  = \|\Delta u\|_{L^2_{\rho,B}}^2 +2\la u_{zz}+\frac{\pa_z u}{z},z\pa_z u\ra_B + \|\Lambda u\|_{L^2_{\rho,B}}^2 \notag \\
&= \|\Delta u\|_{L^2_{\rho,B}}^2  + \la\pa_z(\pa_zu)^2,z\ra_B +2\|\pa_zu\|_{L^2_{\rho,B}}^2 +   \|\Lambda u\|_{L^2_{\rho,B}}^2\notag \\
& = \|\Delta u\|_{L^2_{\rho,B}}^2- B(\pa_zu)^2(\sqrt{b})\rho(\sqrt{b}) -\la (\pa_zu)^2,1\ra_B -\la(\pa_zu)^2,1\ra_b \notag \\
& \ \ \ - \|\Lambda u\|_{L^2_{\rho,B}}^2+2\|\pa_zu\|_{L^2_{\rho,B}}^2 +   \|\Lambda u\|_{L^2_{\rho,B}}^2 \notag \\
& = \|\Delta u\|_{L^2_{\rho,B}}^2 -B(\pa_zu)^2(\sqrt{B})\rho(\sqrt{B}) \label{E:INTER2bis}.
\end{align}
On the other hand, let $$\chi(z)=\left\{\begin{array}{ll} 1 \ \ \mbox{for}\ \ \sqrt{B}\leq z\leq 1\\ 0\ \ \mbox{for}\ \ z\geq 2\end{array}\right.$$ then integrating by parts:
$$
\la H_Bu,\chi(z)\ln z\ra_B=-\la u,H_B(\chi(z)\log z)\ra_B+\sqrt{B}\pa_zu(\sqrt{B})\ln(\sqrt{B})
$$
and thus: 
$$
B|\pa_zu(\sqrt{B})|^2\lesssim \frac{1}{|\ln B|^2}\left[\|H_Bu\|_{L^2_{\rho,B}}^2+\|u\|_{L^2_{\rho,B}}^2\right] \lesssim \frac{1}{|\log B|^2} \|H_B u\|_{L^2_{\rho,B}}^2,
$$
where we used~\eqref{E:INTERbis}. Together with~\eqref{E:INTER2bis},~\eqref{E:INTERbis}, claim \eqref{coercunubisbis} follows.\\

\noindent{\bf step 2} Proof of \eqref{coercdeuxtris}, \eqref{coiehenoenevobis}. Define the radially symmetric function 
\be
\label{defguibgeogeogbis}
v(z)=\left\{\begin{array}{ll} H_Bu(\sqrt{B})\ \ \mbox{for}\ \ 0\leq z\leq \sqrt{B}\\ H_Bu(z)\ \ \mbox{for}\ \ z\geq \sqrt{B}\end{array}\right.
\ee 
and note that $v\in H^1_{\rho,0}.$
Consider 
$$
w: =v-\sum_{j=0}^{k}\frac{\la v,\hat{P}_j\ra_0}{\la \hat{P}_j,\hat{P}_j\ra_0}\hat P_j.
$$ 
Then from~\eqref{estimationcoreorbis}:
\bea
\label{estzerobis}
\nonumber \int_{z\geq 0}|\pa_z w|^2e^{\frac{z^2}{2}}zdz&\geq& (2k+4)\int_{z\geq 0}|w|^2e^{\frac{z^2}{2}}zdz\\
&\geq& (2k+4)\int_{z\geq \sqrt{b}}|w|^2e^{\frac{z^2}{2}}zdz.
\eea
On the other hand, from \eqref{deftk}, \eqref{defpsibhat}: $$\left\|\hat{P_k}-\frac{2}{|\ln B|}
\hat{\psi}_{B,k}\right\|_{H^2_{\rho,B}}\lesssim \frac{1}{|\ln B|}$$ from which for $0\leq j\leq k$:
\bee
\la v,\hat{P}_j\ra_0&=&H_Bu(\sqrt{B})\int_{0\leq z\leq \sqrt{B}}\hat P_je^{\frac{z^2}{2}}zdz+\la H_Bu,\frac{2}{|\ln B|}\psih_{B,j}\ra_B+O\left(\frac{\|H_Bu\|_{L^2_{\rho,B}}}{|\ln B|}\right)\\
& = & O\left(B|H_Bu(\sqrt{B})|+\frac{\|H_Bu\|_{L^2_{\rho,B}}}{|\ln B|}\right),
\eee
where we used the orthogonality $\la u,\psih_{B,j}\ra_B=0,$ $0\le j \le k.$
Therefore 
$$
\|v-w\|_{H^1_{\rho,B}}\lesssim B|H_Bu(\sqrt{B})|+\frac{\|v\|_{L^2_{\rho,B}}}{|\ln B|}.
$$ 
Injecting this into \eqref{estzerobis} yields \eqref{coercdeuxtris}. We now apply \eqref{weightedesimatebis} to $v$ and conclude from \eqref{coercdeuxbis}, \eqref{defguibgeogeogbis}:
\bee
&&\|zH_Bu\|_{L^2_{\rho,B}}^2\lesssim \|zv\|_{L^2_{\rho,0}}^2 + B^2|H_Bu(\sqrt B)|^2
\lesssim \|\pa_zv\|_{L^2_{\rho,0}}^2+\|v\|_{L^2_{\rho,0}}^2+ B^2|H_Bu(\sqrt B)|^2\\
&\lesssim &\|\pa_zH_Bu\|_{L^2_{\rho,B}}^2+\|H_Bu\|_{L^2_{\rho,B}}^2+B|H_Bu(\sqrt{B})|^2
 \lesssim  \|\pa_zH_Bu\|_{L^2_{\rho,B}}^2+B|H_Bu(\sqrt{B})|^2
\eee
and \eqref{coiehenoenevobis} is proved.
\end{proof}

In analogy to Lemma~\ref{lemmacoercunbis}, we now renormalise Lemma \ref{lemmacoercunbis} by letting $\e(y)=u(\sqrt{B}y)$ and obtain:

 \begin{lemma}[Coercivity of $\H_B$]
 \label{lemmacoercunbisbistris}
  Let $k\in \Bbb N$ and $0<B<B^*(k)$ small enough. If $\e\in H^3_{\rho_B}(y\geq 1)$ satisfies 
  $$
  \e(1)=0, \ \ (\e,\eta_{B,j})_b=0\ \ \mbox{for}\ \ 0\leq j\leq k,
  $$ 
  then 
\bea
\label{coercunubisbisbis}
\|\H_B \e\|_B^2&\gtrsim &\|\Delta \e\|_{B}^2+b\|\pa_y\e\|_{B}^2+B^2\|\Lambda \e\|_{B}^2\\
\nonumber &+& B^2\|(1+\sqrt{B}y)\e\|_{B}^2+B|\log B|^2(\pa_y\e)^2(1).
\eea
Moreover,
\be
\label{coercdeuxbisbis}
\|\pa_y \H_B\e\|_{L^2_{\rho,B}}^2\geq \left[2k+4+O\left(\frac{1}{|\log B|}\right)\right] B\|\H_B \e\|_{B}^2-c_kB^2|\H_B\e(1)|^2
\ee
and 
\be
\label{E:YHBEPSILONBOUNDbis}
B\|y\H_B\e\|_{B}^2\lesssim  \frac{1}{B}\|\pa_y\H_B\e\|_{B}^2+|\H_B\e(1)|^2.
\ee
\end{lemma}


\section{Non trivial melting initial data}
\label{data}


In this appendix, we show that our set of initial data for melting is non empty and contains compactly supported data arbitrarily small in $\dot{H}^1$. We show the construction for $k=0$, an analogous construction holds for $k\geq 1$ and is left to the reader.\\
To see this define a cut-off function $\chi(y)=1$ for $y\leq 1$ and $\chi(y)=0$ for $y\geq 2$, and set $\chi_B=\chi(\frac{y}{B}),$ where 
 $$B^2=\frac{|\ln b_0|}{2b_0}.$$ 
 By abuse of notation, we denote by $b_0$ the initial value of $b_0$ in this section.
 Let 
 \be\label{E:ALPHA}
 \alpha : = \frac {\|\eta_{b_0}\|_{b_0^2}}{(\chi_B\eta_{b_0},\eta_{b_0})_{b_0}}.
 \ee
 Note that $\alpha-1 =  \frac {((1-\chi_B)\eta_{b_0},\eta_{b_0})_{b_0}}{(\chi_B\eta_{b_0},\eta_{b_0})_{b_0}}.$
Furthermore
\begin{align}
& \big|((1-\chi_B)\eta_{b_0},\eta_{b_0})_{b_0}\big|  \lesssim \int_{y\ge B} \eta_{b_0}^2 \rho_{b_0} 
 \lesssim \int_{y\ge B} \ln^2(y) \rho_{b_0} + \|\eta_{b_0,1}\|_{b}^2 \notag \\
& = \frac 1{b_0}\int_{\sqrt{b_0}B}^\infty \ln^2(\sqrt{b_0}z) \rho(z)\,dz + \frac 1{b_0|\ln b_0|^2} \notag \\
& \lesssim  \frac {\ln^2b_0}{b_0}\int_{\sqrt{b_0}B}^\infty  \rho(z)\,dz  + \frac 1{b_0}\int_{\sqrt{b_0}B}^\infty  \sqrt{\rho(z)}\,dz + \frac 1{b_0|\ln b_0|^2}\notag \\
& \lesssim \frac {\ln^2b_0}{b_0} e^{-b_0B^2/2} + \frac 1{b_0} e^{-b_0B^2/4} +  \frac 1{b_0|\ln b_0|^2} \notag \\
& \lesssim \frac {\ln^2b_0}{b_0} e^{-|\ln b_0|/4} + \frac 1{b_0} e^{-|\ln b_0|/8} +  \frac 1{b_0|\ln b_0|^2} 
 \lesssim \frac 1{b_0|\ln b_0|^2},\label{E:ALPHA1}
\end{align}
where we used~\eqref{deftketa} and~\eqref{calculeignvectoreta}.
Therefore
\begin{align}
|\alpha-1| & \lesssim  \frac{\frac 1{b_0|\ln b_0|^2}}{\|\eta_{b_0}\|_{b_0}^2 -((1-\chi_B)\eta_{b_0},\eta_{b_0})_{b_0} } 
 \lesssim  \frac{\frac 1{b_0|\ln b_0|^2}}{\|\eta_{b_0}\|_{b_0}^2 }  
 \lesssim \frac{\frac 1{b_0|\ln b_0|^2}}{\frac{|\ln b_0|^2}{b_0} } = \frac 1{|\ln b_0|^4}, \label{E:ALPHA2}
\end{align}
where we used~\eqref{E:ALPHA1} and~\eqref{calculeignvectoreta}.
Consider the initial data 
$$
u(0)=v(0)=b_0\alpha \chi_B \eta_{b_0}. 
$$
Note that by~\eqref{E:ALPHA1} we have the bound $|\alpha|\lesssim1.$
Then using \eqref{deftketa}, \eqref{calculeignvectoreta}: 
\bee
 \int|\pa_yu(0)|^2&\lesssim&  b_0^2\left[\int_{1\le y\leq 2B}|\pa_y\eta_{b_0}|^2+\int_{B_0\leq y\leq 2B_0}\frac{|\eta_{b_0}|^2}{B^2}\right]\\
& \lesssim & b_0^2e^{b_0 B^2/2}\left[\|\pa_y \eta_{b_0}\|_{b_0}^2+\frac{b_0}{K}\|\eta_{b_0}\|_{b_0}^2\right]
\lesssim b_0^2e^{\frac{|\ln b_0|}{4}}\left[|\ln b_0|+\frac{b_0}{K}\frac{|\ln b_0|}{2b_0}\right]\\
&\ll& 1.
\eee
On the other hand, we have by definition $\e(0)=-(1-\alpha \chi_B)b_0\eta_{b_0} = -(1-\alpha)b_0\eta_{b_0}-\alpha (1-\chi_B)b_0\eta_{b_0}$ and hence:
\begin{align*}
\|\mathcal H_{b_0}\e(0)\|_{b}^2
 \lesssim & b_0^2\int_{y\geq B}|b_0\l_{b_0}\eta_{b_0}|^2\rho_{b_0}y\,dy + |1-\alpha|^2 b_0^4\l_{b_0}^2\|\eta_{b_0}\|_{b_0}^2\\
&+b_0^2\int_{B\leq y\leq 2B}\left[\frac{|\pa_y\eta_{b_0}|^2}{y^2}+\frac{\eta_{b_0}^2}{y^4}+b_0^2\eta_{b_0}^2\right]\rho_{b_0} y\,dy.
\end{align*}
Using~~\eqref{E:ALPHA1} we can estimate the first term:
\bee
&&b_0^2\int_{y\geq B}|b_0\l_{b_0}\eta_{b_0}|^2\rho_{b_0}ydy\lesssim b_0^4\l_{b_0}^2 \lesssim b_0^4\l_{b_0}^2 \frac 1{b_0|\ln b_0|^2}  \lesssim \frac{b_0^3}{|\ln b_0|^4},
\eee
For the second term we use~\eqref{E:ALPHA2} and~\eqref{calculeignvectoreta} and readily obtain
\[
|1-\alpha|^2 b_0^4\l_{b_0}^2\|\eta_{b_0}\|_{b_0}^2 \lesssim \frac{b_0^3}{|\ln b_0|^8}.
\]
similarly using the decomposition~\eqref{deftketa} (with $k=0$), we have
\begin{align*}
b_0^2\int_{B\leq y\leq 2B}\frac{|\pa_y\eta_{b_0}|^2}{y^2} \rho_{b_0}y\,dy & \lesssim b_0^2\int_{B\leq y\leq 2B}\frac{1}{y^4} \rho_{b_0} y\,dy + b_0^2\int_{B\leq y\leq 2B}\frac{|\pa_y\eta_{b_0,1}|^2}{y^2} \rho_{b_0} y\,dy \\
& \lesssim b_0^2 B^{-4} e^{-b_0B^2/2} + b_0^2 B^{-2} \|\pa_y\eta_{b_0,1}\|_{b_0}^2 \\
& \lesssim \frac{b_0^4}{|\ln b_0|^4} + \frac{b_0^3}{|\ln b_0|^4} \lesssim   \frac{b_0^3}{|\ln b_0|^4},
\end{align*}
where we used~\eqref{calculeignvectoreta} in the last line.
In a similar fashion
\begin{align*}
b_0^2\int_{B\leq y\leq 2B}\frac{\eta_{b_0}^2}{y^4}\rho_{b_0} y\,dy & \lesssim b_0^2\int_{B\leq y\leq 2B}\frac{|\ln y|^2}{y^4}\rho_{b_0} y\,dy + b_0^2\int_{B\leq y\leq 2B}\frac{\eta_{b_0,1}^2}{y^4}\rho_{b_0} y\,dy \\
& \lesssim b_0^2 B^{-4} |\ln B|^2 \int_{B\leq y\leq 2B}\rho_{b_0} y\,dy + b_0^2 B^{-4}\|\eta_{b_0,1}\|_{b_0}^2 \\
& \lesssim \frac{b_0^3}{|\ln b_0|^2} e^{-b_0 B^2/2} + \frac{b_0^3}{|\ln b_0|^6} \\
& \lesssim  \frac{b_0^3}{|\ln b_0|^2} e^{- |\ln b_0|/4} + \frac{b_0^3}{|\ln b_0|^6}\lesssim  \frac{b_0^3}{|\ln b_0|^6} 
\end{align*}
if $b_0$ is sufficiently small.
Finally,
\begin{align*}
b_0^2\int_{B\leq y\leq 2B}b_0^2\eta_{b_0}^2\rho_{b_0} y\,dy
& \lesssim  b_0^4\int_{B\leq y\leq 2B}(\ln y)^2\rho_{b_0} y\,dy + b_0^4\|\eta_{b_0,1}\|_{b_0}^2 \\
&\lesssim  b_0^4 |\ln B|^2 \frac 1{b_0} e^{-b_0B^2/2} + b_0^4 \frac{1}{b_0 |\ln b_0|^2} \\
& \lesssim  b_0^3|\ln b_0|^2e^{-\frac{|\ln b_0|}{4}} +  b_0^3 \frac{1}{ |\ln b_0|^2} \lesssim \frac{b_0^3}{|\ln b_0|^2}.
\end{align*}
and hence \eqref{E:ESMALL} is satisfied.
Moreover, 
\[
(\e(0),\eta_{b_0})_{b_0} =  (-(1-\alpha \chi_B)b_0\eta_{b_0} ,\eta_{b_0})_{b_0} = -b_0 \|\eta_{b_0}\|_{b_0}^2 + \alpha b_0 (\chi_{B_0}\eta_{b_0},\eta_{b_0}) = 0
\]
by~\eqref{E:ALPHA}, and therefore the orthogonality condition from~\eqref{orthoe} is satisfied.

\begin{remark}
Observe that  by our construction, the initial temperature $u_0$ is non-negative in $\Omega.$ In this case, the solution $u(t,\cdot)$ remains non-negative by the maximum principle.
\end{remark}


\section{Cauchy theory in $\dot{H}^1\times \dot{H}^2$}
\label{appendixcauchy}


\begin{theorem}[Well-posedness in $H^2$]\label{T:WELLPOSEDNESS}
Let $u_0\in H^2(\Omega(0)),$  $\l_0>0,$ $u_0(\l_0)=0.$ 
Then there exists a time $T=T(\| u_0\|_{H^2(\Omega)}, \ \l_0)>0,$ a constant $C>0,$ and a solution $(u,\l)$ to the Stefan problem~\eqref{E:STEFANPOLAR} on the time interval $[0,T]$ such that 
\begin{align}
& u\in C([0,T),H^2(\Omega)) \cap L^2([0,T),H^3(\Omega)), \notag \\
&u_t\in C((0,T),L^2(\Omega)) \cap L^2([0,T],H^1(\Omega)), \notag\\
& \l \in C^1([0,T),\RR),\label{E:REGULARITY}
\end{align}
and the following bounds hold:
\[
 \|u\|_{H^2(\Omega(t))} \le C=C(\|u_0\|_{H^2(\Omega(0)),\l_0}), \ \ \lambda(t)>\frac {\l_0}2
\]
for some universal polynomial function C of the initial data.
Moreover, if $\mathcal{T}$ is the maximal time of existence of a solution $(w,\l)$ satisfying~\eqref{E:REGULARITY}, then
\[
\text{either } \ \lim_{t\to\mathcal{T}^-}\|u(t,\cdot)\|_{H^2(\Omega(t))} = \infty \ \ \text{ or } \ \ \ \lim_{t\to\mathcal{T}^-}\l(t) =0.
\]
\end{theorem}


\begin{remark}
The Stefan problem allows for an instant smoothing effect. It is well-known that the solution $u$ becomes infinitely smooth on $(0,\mathcal{T})$ in both the time- and the space variable, see for instance~\cite{Ko98, PrSaSi, KiNi78}.
\end{remark}


The proof of Theorem~\ref{T:WELLPOSEDNESS} is presented at the very end of this section, as a simple consequence of Theorem~\ref{T:WELLPOSEDNESSW}.

We start by pulling-back the problem~\eqref{E:STEFANPOLAR} onto the fixed domain
$\Omega :=  \{ {\bf y} \in\mathbb{R}^2, \ |{\bf y}|\ge 1\}.$ 
We denote the points in $\Omega$ by bold ${\bf y},$  while the radial coordinate $|{\bf y}|$ is denoted by $y.$
We define the pull-back temperature function $w:\Omega\to\RR$ by
\[
w(t,y)  = u (t, \l(t)y) 
\]
A simple application of the chain rule gives the following system of equations for $w:$
\begin{subequations}
\label{E:STEFANPOLARW}
\begin{alignat}{2}
w_t-\frac{\dot{\l}}{\l}\Lambda w - \frac 1{\l^2}\Delta w&=0&& \text{ in }  \Omega\,;\label{E:HEATW}\\
w_y(t,1)&=-\dot{\lambda}(t)\l(t);&& \label{E:NEUMANNW}\\
w(t,1)&=0,&&\label{E:DIRICHLETW}\\
w(0,\cdot) =w_0\,, & \ \l(0)=\l_0.&& \label{E:INITIALW}
\end{alignat}
\end{subequations}


\begin{lemma}[Energy identities]\label{L:ENERGYIDENTITIES}
Assume that $(w,\l)$ is a smooth solution to the Stefan problem~\eqref{E:STEFANPOLARW} on some interval $[0,T].$  
Assume that $\l(0)>0,$ $w_0\big|_{y=1}=0,$
and that  $w(t,\cdot)\in H^2(\Omega)$ for $t\in[0,T].$
Then on the interval $[0,T]$ the following energy identities hold:
\begin{align}
& \frac 12\frac{d}{dt}\|w\|_{L^2(\Omega)}^2  +\frac 1{\l^2} \|\nabla w\|_{L^2(\Omega)}^2   = - \frac {\dot \l}{\l} \|w\|_{L^2(\Omega)}^2,\label{E:ENERGY0}\\
 & \frac 12 \frac{d}{dt} \|\nabla w\|_{L^2(\Omega)}^2 +\frac 1{\l^2} \|\Delta w\|_{L^2(\Omega)}^2   =  \pi \l \dot \l^3  \label{E:ENERGY1},\\
&  \frac 12 \frac{d}{dt} \|\Delta w\|_{L^2(\Omega)}^2 +\frac {2\pi}3 \frac{d}{dt} \left( \l |\dot \l|\right)^3
 + \frac 1{\l^2} \|\nabla\Delta w\|_{L^2(\Omega)}^2  \notag \\
 & \ \ \  = \frac {2\dot{\lambda}}{\l}  \|\Delta w\|_{L^2(\Omega)}^2 + \pi \dot{\lambda}^5\l^3.  \label{E:ENERGY2}
\end{align}
\end{lemma}

\begin{proof}
Multiply~\eqref{E:HEATW} by $ w$ and integrate over $\Omega.$ We obtain
\begin{align*}
0 &  = \frac 12\frac{d}{dt} \int_{\Omega} w^2\, d{\bf y} - \frac{\dot \l}\l \int_{\Omega}\Lambda w w\, d{\bf y} + \frac 1{\l^2}\|\nabla w\|_{L^2(\Omega)}^2  \\
& =  \frac 12\frac{d}{dt} \|w\|_{L^2(\Omega)}^2 - \frac{\dot \l}{2\l} \int_{1}^\infty y^2\pa_y(w^2)\, dy + \frac 1{\l^2}\|\nabla w\|_{L^2(\Omega)}^2 \\
& = \frac 12\frac{d}{dt}\|w\|_{L^2(\Omega)}^2 + \frac{\dot \l}{\l}\|w\|_{L^2(\Omega)}^2 + \frac 1{\l^2}\|\nabla w\|_{L^2(\Omega)}^2 ,
\end{align*} 
which is precisely~\eqref{E:ENERGY0}.
Multiply~\eqref{E:HEATW} by $-\Delta w$ and integrate by parts. We obtain
\begin{align*}
0& = \int_\Omega \nabla w\cdot \nabla w_t \, d{\bf y}- 2\pi (\pa_nw w_t)\big|_{y=1} + \frac {\dot{\lambda}}{\l} \int_\Omega \Lambda w \Delta w \, d{\bf y} + \frac 1{\l^2} \|\Delta w\|_{L^2(\Omega)}^2 \\
& = \frac 12 \frac{d}{dt} \|\nabla w\|_{L^2(\Omega)}^2 + 2\pi( w_yw_t)\big|_{y=1} + 2\pi  \frac {\dot{\lambda}}{\l} \int_1^\infty \Lambda w \pa_y\Lambda w \,dy +\frac 1{\l^2} \|\Delta w\|_{L^2(\Omega)}^2 \\
& =  \frac 12 \frac{d}{dt} \|\nabla w\|_{L^2(\Omega)}^2 - \pi  \frac {\dot{\lambda}}{\l} (\Lambda w)^2\big|_{y=1} +\frac 1{\l^2} \|\Delta w\|_{L^2(\Omega)}^2 \\
& = \frac 12 \frac{d}{dt} \|\nabla w\|_{L^2(\Omega)}^2 - \pi \l \dot \l^3 +\frac 1{\l^2} \|\Delta w\|_{L^2(\Omega)}^2,
\end{align*}
where we used the fact that $w_t(t,1)=0,$ $\Delta w = \frac 1y\pa_y\Lambda w,$ and $\Lambda w\big|_{y=1} = w_y\big|_{y=1} = -\l \dot{\l}.$ 
Note also that $\pa_n w|_{y=1} = -\pa_y w|_{y=1}.$This proves~\eqref{E:ENERGY1}.
To prove~\eqref{E:ENERGY2} we apply $\nabla$ to~\eqref{E:STEFANPOLARW}, multiply by $-\nabla\Delta w$ and integrate-by-parts.
We obtain
\begin{align}
0 & = - \int_\Omega \nabla w_t \cdot\nabla\Delta w \,d{\bf y} + \frac {\dot{\lambda}}{\l} \int_\Omega \nabla\Lambda w\cdot\nabla \Delta w\, d{\bf y}
+\frac 1{\l^2} \|\nabla\Delta w\|_{L^2(\Omega)}^2 \notag \\
& = \frac 12 \frac{d}{dt}\int_\Omega(\Delta w)^2\, d{\bf y} +2\pi (\pa_yw_t\Delta w)\big|_{y=1}  - \frac {\dot{\lambda}}{\l} \int_\Omega \Delta\Lambda w\cdot \Delta w\, d{\bf y} \notag \\
& \ \ \ 
 - 2\pi \frac {\dot{\lambda}}{\l} ( \pa_y\Lambda w\Delta w)\big|_{y=1} + \frac 1{\l^2} \|\nabla\Delta w\|_{L^2(\Omega)}^2. \label{E:PIECE0}
\end{align}
From~\eqref{E:NEUMANNW} it follows that 
\[
\pa_yw_t\big|_{y=1} = - \pa_t (\l\dot \l)
\]
Restricting~\eqref{E:HEATW} to $y=1$ we conclude that $- \frac {\dot{\lambda}}{\l} w_y\big|_{y=1} - \frac 1 {\l^2} (\Delta w)\big|_{y=1}  = 0.$ Using~\eqref{E:NEUMANNW} this implies that 
\be\label{E:DELTAWBOUNDARY}
(\Delta w)\big|_{y=1} = \l^2\dot{\l}^2.
\ee
The previous two boundary identities imply that 
\be\label{E:PIECE1}
2\pi (\pa_yw_t\Delta w)\big|_{y=1} = -\frac {2\pi}3 \frac{d}{dt} \left(\l \dot \l\right)^3 . 
\ee
Note that $\pa_y\Lambda w = w_y + y w_{yy}$ and therefore, when restricted to $y=1$ we conclude that $(\pa_y\Lambda w)\big|_{y=1} = (\Delta w)\big|_{y=1}.$ 
From~\eqref{E:DELTAWBOUNDARY} we infer that 
\[
(\pa_y\Lambda w)\big|_{y=1} = \l^2\dot{\l}^2.
\]
Therefore
\be\label{E:PIECE2}
- 2\pi \frac {\dot{\lambda}}{\l} ( \pa_y\Lambda w\Delta w)\big|_{y=1} = - 2\pi \l^3 \dot{\l}^5.
\ee
It remains to evaluate the term $ - \frac {\dot{\lambda}}{\l} \int_\Omega \Delta\Lambda w\cdot \Delta w\, d{\bf y} .$
A  direct calculation yields
\[
\Delta \Lambda w =  \Lambda \Delta w + 2 \Delta w.
\]
Therefore
\begin{align}
- \frac {\dot{\lambda}}{\l} \int_\Omega \Delta\Lambda w\cdot \Delta w\, d{\bf y}
& = -\frac {2\dot{\lambda}}{\l}  \|\Delta w\|_{L^2(\Omega)}^2 - 2\pi  \frac {\dot{\lambda}}{\l} \int_1^\infty \pa_y \Delta w \Delta w\, dy \notag \\
& = -\frac {2\dot{\lambda}}{\l}  \|\Delta w\|_{L^2(\Omega)}^2 + \pi \frac {\dot{\lambda}}{\l} (\Delta w)^2\big|_{y=1}. \label{E:PIECE3}
\end{align}
Plugging~\eqref{E:PIECE1},~\eqref{E:PIECE2}, and~\eqref{E:PIECE3} into~\eqref{E:PIECE0}, we obtain~\eqref{E:ENERGY2}.
\end{proof}


Let us define the energy-like quantities
\begin{align}
E(t)  =  \sup_{0\le s\le t} &\Big\{\frac 12 \|w(s,\cdot)\|_{L^2(\Omega)}^2+ \frac 12 \|\nabla w(s,\cdot)\|_{L^2(\Omega)}^2 + \frac 12\|\Delta w(s,\cdot)\|_{L^2(\Omega)}^2\Big\}, \label{E:ENERGYDEFINITION}
\end{align}
\be\label{E:DISSIPATIONDEFINITION}
D(t)= \frac 1{\l(t)^2}\|\nabla w(t,\cdot)\|_{L^2(\Omega)}^2 + \frac 1{\l(t)^2} \|\Delta w(t,\cdot))\|_{L^2(\Omega)}^2 + \frac 1{\l(t)^2} \|\nabla\Delta w(t,\cdot)\|_{L^2(\Omega)}^2.
\ee


\begin{lemma}[A priori estimate]\label{L:APRIORI}
Assume that $(w,\l)$ is a smooth solution to~\eqref{E:STEFANPOLARW} on some interval $[0,T^*].$  
Assume that  $\l_0>0,$ $w_0\big|_{y=1}=0,$
and that  $w(t,\cdot)\in H^2(\Omega)$ for $t\in[0,T^*].$ 
Then there exists a $T=T(E(0),\l_0)>0,$ $T\le T^*,$ such that for any $t\in[0,T]$ the following a priori bounds hold:
\be
E(t) \le 4 E(0), 
\ee
\be
\l(t) > \frac {\l_0} 2.
\ee
\end{lemma}

\begin{proof}
From Lemma~\ref{L:ENERGYIDENTITIES} we infer that 
\begin{align}
E(t) + \int_0^t D(s)\,ds \le & E(0) + \frac{2\pi}3 \l_0^3\dot\l(0)^3 - \frac{2\pi}3 \l(t)^3\dot\l(t)^3 
+ \int_0^t \frac{|\dot \l(s)|}{\l(s)}\,ds \, E(t)\notag\\
& + \pi \int_0^t \left( \l^3(s)|\dot{\l}(s)|^5 +\l(s)\dot\l(s)^3\right)\,ds \label{E:ENERGYINEQUALITY1}
\end{align}
Note that 
\[
w_y^2\big|_{y=1} \lesssim \|\Delta w\|_{L^2(\Omega)}^2 + \|\nabla w\|_{L^2(\Omega)}^2.
\]
Therefore
\be\label{E:TRACE}
|\dot \l(t)|^2 \le \frac C {\l^2} E(t), \ \ C>1.
\ee
Using~\eqref{E:TRACE} we obtain
\begin{align}
&E(t) \int_0^t \frac{|\dot \l(s)|}{\l(s)}\,ds + \pi \int_0^t \left( \l^3(s)|\dot{\l}(s)|^5 +\l(s)|\dot\l(s)|^3\right)\,ds \notag \\
& \ \ \ \ \le C' t \left(E^2(t)+E^{3/2}(t)+E^{5/2}(t)\right) \sup_{0\le s\le t}\frac 1{\l^2(s)}. \label{E:ERROR1}
\end{align}
To bound the error term $- \frac{2\pi}3 \l(t)^3\dot\l(t)^3$ we need a more refined estimate than~\eqref{E:TRACE} due to the absence of the integral-in-time.
Note that by the trace inequality and the interpolation between fractional Sobolev spaces, we have
\[
|w_y(1)| \lesssim \|\nabla w\|_{H^{1/2}(\Omega)} \lesssim \|\nabla w\|_{L^2(\Omega)}^{1/2}  \|\nabla w\|_{H^1(\Omega)}^{1/2}.
\]
Therefore, upon using the Young inequality
\begin{align}
|\l \dot \l|^3 & \lesssim \|\nabla w\|_{L^2(\Omega)}^{3/2} \|\nabla w\|_{H^1(\Omega)}^{3/2} \le \delta \|\nabla w\|_{H^1(\Omega)}^{2} + C_\delta \|\nabla w\|_{L^2(\Omega)}^6 \notag \\
& \le \delta E + C_\delta \|\nabla w\|_{L^2(\Omega)}^6. \label{E:REFINED}
\end{align}
Integrating~\eqref{E:ENERGY1} over $[0,t],$ we have
\begin{align*}
 \sup_{0\le s \le t}\|\nabla w\|_{L^2(\Omega)}^2 & \le \|\nabla w_0\|_{L^2(\Omega)}^2 + \pi \int_0^t \l(s)|\dot\l(s)|^3 \,ds \\
& \le E_0 + C^*t E^{3/2}(t)\sup_{0\le s\le t}\frac 1{\l^2(s)}.
\end{align*}
Therefore,  we obtain from~\eqref{E:REFINED} that 
\[
|\l \dot \l|^3 \le C_0 + \delta E + C t p(E)q(\sup_{0\le s\le t}\frac 1{\l^2(s)})
\]
where $p$ and $q$ are increasing polynomial functions of their arguments.
Plugging this bound back into~\eqref{E:ENERGYINEQUALITY1}, using~\eqref{E:ERROR1}, and the definition of $E(t)$ we conclude that 
\be\label{E:ENERGYINEQUALITY}
E(t) + \int_0^t D(s)\,ds \le C_0 + C t p(E(t)) q(\sup_{0\le s\le t}\frac 1{\l^2(s)}),
\ee
where $C_0=C_0(E(0),\l_0).$
Since $\lambda(t) = \l(0) +\int_0^t\dot \l(s)\,ds$, it follows that 
\[
\l(t) \ge \l_0 - t \sup_{0\le s \le t}|\dot \l(s)| \ge \l_0 - \frac{Ct}{\l} E^{1/2}(t).
\]
Let 
\[
T' = \sup\{ t \ge0 \big| E(t) \le 4 E(0), \ \l(t)>\l_0/2\}.
\]
By the continuity of $E(\cdot)$ and $\l(\cdot)$ it follows that $T'>0.$
On $[0,T']$ we therefore have 
\be\label{E:ENERGYINEQUALITY2}
E(t) + \int_0^t D(s)\,ds \le C_0 + Ctp(E(t))  q(\frac 1{\l_0 - 4\frac{\sqrt C}{\l(0)} t  E(0)^{1/2}})
\ee
By a standard continuity argument, there exists a sufficiently small
$T=T(E(0),\l_0),$ $T\le \frac {\l_0}{16 \sqrt {C}E(0)^{1/2}}$ such that 
\[
E(t) \le 2 C_0,  \ \ t\in[0,T].
\]
By the choice of $T,$ we also have the bound $\lambda (t) \ge \frac 34\l_0>\frac 12\l_0$ for $t\in[0,T]$
and this concludes the proof of the lemma.
\end{proof}


\begin{theorem}[Local well-posedness]\label{T:WELLPOSEDNESSW}
Let $w_0\in H^2(\Omega),$ $\l_0>0,$ and $w_0\big|_{y=1}=0.$
Then there exists a time $T=T(\| w_0\|_{H^2(\Omega)}, \ \l(0))>0$ and a solution $(w,\l)$ to the Stefan problem~\eqref{E:STEFANPOLARW} on the time interval $[0,T]$ such that 
\begin{align}
& w\in C([0,T],H^2(\Omega)) \cap L^2([0,T],H^3(\Omega)), \notag \\
& w_t\in C((0,T], L^2(\Omega)) \cap L^2([0,T],H^1(\Omega)), \notag \\
&  \l \in C^1([0,T],\RR), \label{E:REGULARITYW}
\end{align}
satisfying the energy estimate 
\[
E(t) \le C_0=C_0(E(0)), \ \ t\in[0,T]
\]
and the lower bound 
\[
\lambda(t)>\frac {\l(0)}2 \ \ t\in[0,T],
\]
where the energy $E(\cdot)$ is defined by~\eqref{E:ENERGYDEFINITION}.
Moreover, if $\mathcal{T}$ is the maximal time of existence of a solution $(w,\l)$ satisfying~\eqref{E:REGULARITYW}, then
\[
\text{either } \ \lim_{t\to\mathcal{T}^-}\|w(t,\cdot)\|_{H^2(\Omega)} = \infty \ \ \text{ or } \ \ \ \lim_{t\to\mathcal{T}^-}\l(t) =0.
\]
\end{theorem}

\begin{proof}
The proof of existence follows a standard iteration argument for the sequence of approximations $(\l_n(t),w_n(t)),$ $n\in\mathbb{N}.$ For a given $\l_n(\cdot)$ we define
$w_{n+1}$ by solving
\begin{subequations}
\begin{alignat}{2}
\pa_tw_{n+1}-\frac{\pa_t\l_n}{\l_n}\Lambda w_{n+1} - \frac 1{\l_n^2}\Delta w_{n+1}&=0&& \text{ in }  \Omega\,;\label{E:ITERATIONHEAT} \\
w_{n+1}(t,1)&=0.&&\label{E:DIRICHLETW}\notag
\end{alignat}
\end{subequations} 
We then update $\l_{n+1}(\cdot)$ by solving
\[
\pa_yw_{n+1}(t,1)=-\pa_t\lambda_{n+1}(t)\l_n(t).
\]
Estimates analogous to 
the a priori estimates of Lemma~\ref{L:APRIORI} can be used to obtain uniform bounds on $E(w_n,\l_n) + \int_0^tD(w_n,\l_n)\,ds,$ where $E$ and $D$ are defined
by~\eqref{E:ENERGYDEFINITION} and~\eqref{E:DISSIPATIONDEFINITION} respectively. Note that 
 we can also get uniform bounds on $\|\pa_tw_{n}\|_{L^\infty([0,T], L^2(\Omega))}+\|\pa_tw_{n}\|_{L^2([0,T], H^1(\Omega))}$
as the latter norms are controlled by $E(w_n,\l_n) + \int_0^t D(w_n,\l_n)\,ds$ from~\eqref{E:ITERATIONHEAT}.
Upon passing to the limit, we obtain a solution to which the energy estimate of Lemma~\ref{L:APRIORI} applies.
The proof of uniqueness is standard.
The breakdown criterion is a simple consequence of~\eqref{E:ENERGYDEFINITION}, and~\eqref{E:TRACE}.
\end{proof}


\subsection*{Proof of Theorem~\ref{T:WELLPOSEDNESS}}
Let $w$ be the solution to~\eqref{E:STEFANPOLARW} as given by Theorem~\ref{T:WELLPOSEDNESSW}.
It is easy to check that $\|w(t,\cdot)\|_{H^2(\Omega)} \lesssim E(t),$ $t\in [0,\mathcal{T}).$
Theorem~\ref{T:WELLPOSEDNESS} now follows from  the  change of variables $u(t,r) = w(t, \frac r {\l(t)})$ and 
Theorem~\ref{T:WELLPOSEDNESSW}.

\end{appendix}


\begin{thebibliography}{9}
\bibitem{Lu}
Luckhaus, S.:
Solutions for the two-phase Stefan problem with the Gibbs-Thomson law for the
melting temperature.
{\em Europ. J. Appl. Math.}
{\bf 1}, 101-111 (1990)
\end{thebibliography}
\end{document}